\documentclass[11pt,a4paper]{article}
\usepackage{amssymb, amsthm, amsmath, amsfonts, mathrsfs}
\usepackage{array, epsfig}
\usepackage{bbm}
\usepackage[hidelinks]{hyperref}
\usepackage[numbers,square]{natbib}
\usepackage{color}
\usepackage{enumitem}

  \evensidemargin
\oddsidemargin
\theoremstyle{plain}
\newtheorem{theorem}{Theorem}[section]
\newtheorem{lemma}[theorem]{Lemma}
\newtheorem{corollary}[theorem]{Corollary}
\newtheorem{proposition}[theorem]{Proposition}

\theoremstyle{definition}
\newtheorem{definition}[theorem]{Definition}

\numberwithin{equation}{section}

\theoremstyle{remark}
\newtheorem{remark}[theorem]{Remark}

\def\BB{\mathbb{B}}

\def\EE{\mathbb{E}}

\def\NN{\mathbb{N}}

\def\PP{\mathbb{P}}
\def\QQ{\mathbb{Q}}
\def\RR{\mathbb{R}}
\def\SS{\mathbb{S}}
\def\TT{\mathbb{T}}

\def\ZZ{\mathbb{Z}}

\def\cA{\mathcal{A}}
\def\cB{\mathcal{B}}
\def\cC{\mathcal{C}}
\def\cD{\mathcal{D}}
\def\cE{\mathcal{E}}
\def\cF{\mathcal{F}}

\def\cH{\mathcal{H}}

\def\cK{\mathcal{K}}
\def\cL{\mathcal{L}}

\def\cP{\mathcal{P}}

\def\cS{\mathcal{S}}
\def\cT{\mathcal{T}}

\def\cV{\mathcal{V}}

\def\cY{\mathcal{Y}}

\newcommand{\dd}{{\rm d}}
\newcommand{\conv}{\mathop{\mathrm{conv}}\nolimits}
\newcommand{\frI}{\mathop{\mathrm{I}}\nolimits}

\newcommand{\pow}{\mathop{\mathrm{pow}}\nolimits}

\newcommand{\proj}{\operatorname{proj}}
\newcommand{\inter}{\operatorname{int}}

\newcommand{\ver}{\operatorname{vert}}
\newcommand{\skel}{\operatorname{skel}}

\newcommand{\bd}{\operatorname{bd}}

\newcommand{\dint}{\textup{d}}

\newcommand{\apex}{\mathop{\mathrm{apex}}\nolimits}
\newcommand{\ext}{\mathop{\mathrm{ext}}\nolimits}
\newcommand{\cl}{\mathop{\mathrm{cl}}\nolimits}

\newcommand{\unif}{\text{\rm unif}}

\DeclareMathOperator*{\argmin}{arg\,min}

\newcommand{\ton}{\overset{}{\underset{n\to\infty}\longrightarrow}}

\newcommand{\ind}{\mathbbm{1}}

\newcommand{\by}{\mathbf{y}}

\newcommand{\bv}{\mathbf{v}}
\newcommand{\bh}{\mathbf{h}}

\newcommand{\eee}{{\rm e}}

\newcommand{\sk}[1]{\mathscr{#1}}
\newcommand{\origin}{\ensuremath{\mathbf{o}}}

\makeatletter
\let\@fnsymbol\@alph
\makeatother

\begin{document}

\title{\bfseries Limits of Poisson-Laguerre tessellations}

\author{Anna Gusakova\footnotemark[1],\; Mathias in Wolde-L\"ubke\footnotemark[2]}

\date{}
\renewcommand{\thefootnote}{\fnsymbol{footnote}}
\footnotetext[1]{M\"unster University, Germany. Email: gusakova@uni-muenster.de}
\footnotetext[2]{M\"unster University, Germany. Email: miwluebke@uni-muenster.de}

\maketitle

\begin{abstract}

For sequences of Poisson-Laguerre tessellations and their duals in $\RR^d$, generated by Poisson point processes $(\eta_n)_{n\in\NN}$ in $\RR^d \times \RR$, we prove limit theorems as $n\to \infty$. The intensity measure of $\eta_n$ has density of the form $(v,h)\mapsto f_n(h)$ with respect to the Lebesgue measure, where $v\in \RR^d$ and $h\in \RR$. Identifying a tessellation with its skeleton (the union of the boundaries of all its cells) we provide verifiable conditions on $(f_n)_{n\in\NN}$ that ensure convergence either to the classical Poisson-Voronoi/Poisson-Delaunay tessellation or to another Poisson-Laguerre tessellation. We also discuss convergence of the corresponding typical cells. As a corollary, we show that the Poisson-Voronoi and the Poisson-Delaunay tessellations arise as limits of the $\beta$-Voronoi and the $\beta$-Delaunay tessellations, respectively, as $\beta\to -1$. \\

\noindent {\bf Keywords}. {Convergence in probability; Poisson-Laguerre tessellation; Poisson-Voronoi tessellation; Poisson-Delaunay tessellation; Poisson point process; random closed sets; skeleton; stabilization; typical cell; weak convergence}  \\
{\bf MSC}. Primary 60D05; 60G55; Secondary 60B10
\end{abstract}


\section{Introduction}

A tessellation in $\RR^d$ is a locally finite collection of convex polytopes, which have non-empty disjoint interior, and cover the entire space $\RR^d$. Random tessellations are classical objects in stochastic geometry, but they appear naturally in other areas of mathematics, like percolation theory, optimal transport, computational geometry, high-dimensional probability and have a number of applications in machine learning, cellular networks and materials science; see e.g \cite{ BookCompGeom, Percolation1, BPR23, Dir19, ORT24, PC19_OT, RJ25, Network} and references therein. In this article we study the convergence of skeletons (i.e.~the union of cell boundaries) of Poisson-Laguerre tessellations and we give conditions ensuring the convergence to the classical Poisson-Voronoi/Poisson-Delaunay tessellations or to another Poisson-Laguerre tessellation.

\paragraph{The model (Poisson-Laguerre tessellation).} 

Consider a (random) set $A\subset \RR^d\times\RR$ of weighted points $(v,h)$, where $v\in\RR^d$ denotes the spatial coordinate and $h\in\RR$ is a weight (interpretable, for instance, as time or height). For any $(v,h)\in A$, define its Laguerre cell by
\[
    C((v,h),A):=\{w\in \RR^d: \|w-v\|^2+h\le \|w-v'\|^2+h' \text{ for all } (v',h')\in A\}.
\]
The associated Laguerre diagram $\cL(A)$ is the collection of all Laguerre cells with non-empty interior. In the special case of constant weights (e.g.\ $h\equiv 0$),  $\cL(A)$ coincides with the classical Voronoi diagram.

In this article we take $A$ to be a Poisson point process $\eta$ in $\RR^d\times \RR$, yielding the \textit{Poisson-Laguerre tessellation}. The classical Poisson-Voronoi case is included as follows: if $\eta=\eta^{\gamma}$ is a homogeneous Poisson point process in $\RR^d$ with intensity $\gamma>0$, we identify it with a process in $\RR^d \times \RR$ via the embedding $v\mapsto (v,0)$. Then $\cL(\eta^{\gamma})$ coincides with the \textit{Poisson-Voronoi tessellation}, one of the most studied examples of a random tessellation  \cite{BookVoronoi}.

The setting of an independently marked homogeneous Poisson process in $\RR^d$ with marks in $\RR$ was studied in \cite{LZ08}. A further extension was introduced in \cite{gusakova2024PLT}, where the authors considered a Poisson point process $\eta=\eta_f$ in $\RR^d\times \RR$ with an intensity measure of the form
\[
    \Lambda_f(\cdot) = \int_{\RR^d}\int_{E} f(h) \ind((v,h)\in \cdot)\, \dd h\, \dd v,
\]
for some locally integrable and non-negative function $f:E\to \RR_+$, where $E\subset \RR$, satisfying some natural integrability assumption, i.e. $f$ is admissible \cite[Definition 3.4]{gusakova2024PLT} (see also Section \ref{subsec:laguerreTess}). A few special cases have been studied in detail, namely 
$f_{\beta}(h)=c_{d+1,\beta}h^{\beta}$,  $h>0$ for $\beta>-1$; $f'_{\beta}(h)=c'_{d+1,\beta}(-h)^{-\beta}$, $h<0$ for $\beta >d/2+1$; and $\widetilde f(h)=(2\pi)^{-d/2-1}e^{h/2}$.
These yield the so-called $\beta$-Voronoi, $\beta'$-Voronoi and Gaussian-Voronoi tessellations, respectively; see \cite{GKT20, GKT21b, sectional, GKT21, GKT21a}. 

Together with $\cL(\eta)$ one often considers the \textit{dual Poisson-Laguerre tessellation} $\cL^*(\eta)$. It is defined as the collection of cells obtained by taking the convex hull of points $v_1,\ldots, v_{m}$ with $(v_i,h_i)\in \eta$, $1\leq i\leq m$, such that the corresponding Laguerre cells $C((v_i,h_i),\eta)$ have a single common point. For $\eta=\eta_f$, with admissible $f$, and likewise for $\eta=\eta^{\gamma}$, one has $m=d+1$ almost surely, and hence each cell of $\cL^*(\eta)$ is a simplex. In particular, $\cL^*(\eta^{\gamma})$ is called \textit{Poisson-Delaunay tessellation}.

We note that, although the classical Poisson-Voronoi tessellation $\cL(\eta^{\gamma})$ does not belong to the family $\{\cL(\eta_f)\colon f\text{ admissible}\}$, it is conjectured to arise as a limiting case of the $\beta$-Voronoi tessellations as $\beta\to -1$ (see \cite[Remark 2]{GKT20}) and there is substantial evidence supporting this convergence (see Section \ref{sec:Applications} for more details). Our general convergence results for Poisson-Laguerre tessellations verify this conjecture.

\paragraph{Convergence of random tessellations.} To formulate limit theorems for tessellations, we first need to specify an appropriate notion of convergence. This is a non-trivial task, since it is not known whether the space of tessellations is even Polish (see \cite[Section 2.3]{HT19} and Section \ref{subsec:randomTessAndClosedSets}). A tessellation $X$ may be represented in several natural ways: as a particle process (the collection of its cells), as a random closed set via its skeleton
\[
    \skel(X) := \bigcup_{c\in X}\bd c;
\]
or as a random graph \cite{BC20, BG25}. These representations lead to different notions of convergence. In this article we adopt the skeleton representation and study convergence of the skeletons as random closed sets.

Results on convergence of tessellations are comparatively rare. Convergence of skeletons has been studied, for instance, for STIT tessellations \cite{NW03, NW05}, and rescaled limits for $\cL^*(\eta_{f_{\beta}})$ and $\cL^*(\eta_{f'_{\beta}})$ as $\beta\to\infty$ were obtained in \cite{GKT21}. In a different direction, law-intensity limits of Poisson–Voronoi tessellations in hyperbolic and other non-Euclidean settings have recently been investigated in \cite{IdealPV24,  IdealPV25, IdealPV23}. 

In this article we provide conditions on the sequence of height densities $(f_n)_{n\in\NN}$ that ensure convergence of the skeletons of the corresponding Poisson-Laguerre and dual Poisson-Laguerre tessellations, including a regime in which the limit is the classical Poisson-Voronoi/Poisson-Delaunay tessellation.

\paragraph{Main results.} Our main results cover two different regimes. 

\smallskip

\textbf{Limits within the Poisson–Laguerre class:} Let $(f_n)_{n\in\NN}$ be a sequence of admissible functions, which converges to an admissible function $f$ in local $L^1$ sense, and satisfies the uniform tail/moment bound
\[
    \sup_{n\in\NN}\int_{-\infty}^{x_0}|h|^{{d\over 2}+\delta}f_n(h)\dd h<\infty,
\]
for some $x_0\in\RR$ and $\delta>0$ (see Condition \hyperref[item:C1]{(C1)}). Then we will prove in Theorem \ref{thm:StrongProbConvLaguerre} that there exists a coupling of point processes $\eta_f$ and $(\eta_{f_n})_{n\in\NN}$, such that for any ball $B_R$ of radius $R>0$ we have
\begin{align*}
\lim_{n\to\infty}&\PP\big(\skel (\cL(\eta_{f_n}))\cap B_R= \skel (\cL(\eta_{f}))\cap B_R\big)=1,\\ \lim_{n\to\infty}&\PP\big(\skel (\cL^*(\eta_{f_n}))\cap B_R= \skel (\cL^*(\eta_{f}))\cap B_R\big)=1.
\end{align*}
In this case we say that convergence holds \textit{locally with high probability} and it in particular implies, that both skeletons converge weakly as random closed sets (see Section \ref{subsec:randomTessAndClosedSets} and Theorem \ref{thm:weakConvPPPonSameSpaceDual} for more details). Moreover, these criteria recover the results from \cite{GKT21} as shown in Section \ref{sec:Applications}. 

The proof of Theorem \ref{thm:StrongProbConvLaguerre} relies on stabilization and coupling arguments. Passing from convergence of the point processes $(\eta_{f_n})_{n\in\NN}$ to convergence of the skeletons is a delicate problem, since small changes in the configuration of the point process may create or remove cells at arbitrarily large distances, resulting in inherently non-local dependencies. This leads to the new obstacles in the proof.

\smallskip

\textbf{Homogeneous limits:} Let $(f_n)_{n\in \NN}$ be a sequence of non-negative, locally integrable 
functions such that the height intensity measure $f_n(h)\dd h$ converges vaguely to the weighted Dirac measure $\gamma\delta_0$, or, equivalently,
\[
\lim_{n\to\infty} \int_0^x f_n(h)\dd h = \gamma\in (0,\infty) \quad \text{ for all } x>0
\]
(see Condition \hyperref[item:C2]{(C2)}). Geometrically, this means that the
points of $\eta_{f_n}$ concentrate near the hyperplane $\{(v,h):h=0\}$ as $n\to\infty$. Then in Theorem \ref{thm:StrongConvergenceVoronoi} we show that there exists a coupling of $\eta^{\gamma}$ and $(\eta_{f_n})_{n\in\NN}$, such that for any ball $B_R$ of radius $R>0$ we have
\begin{align}\label{eq:SkelDelaunayConv}
\lim_{n\to\infty}&\PP\big(\skel(\cL^*(\eta_{f_n}))\cap B_R= \skel(\cL^*(\eta^{\gamma}))\cap B_R\big)=1,
\end{align}
 and the sequence of random closed sets $\skel(\cL(\eta_{f_n}))$ converges in probability to $\skel(\cL(\eta^{\gamma}))$ as $n\to\infty$. In particular, both sequences converge weakly as random closed sets (see Theorem \ref{thm:weakConv}). As an application, we obtain that the classical Poisson-Voronoi and Poisson-Delaunay tessellations arise as weak limits of the $\beta$-Voronoi and $\beta$-Delaunay tessellations, respectively, as $\beta\to -1$; see Section \ref{sec:Applications}. Let us emphasize that this case differs substantially from the previous one since $\eta_{f_n}$ is a point process in $\RR^d\times\RR_+$, while $\eta^{\gamma}$ is a point process in $\RR^d$. This makes the coupling step much more involved and prevents us from proving the statement of type \eqref{eq:SkelDelaunayConv} for $\skel(\cL(\eta_{f_n}))$ (see Remark \ref{rem:NoStrongConvVoronoi} for further details).

\smallskip

\textbf{Typical cell:} One may view the convergence at the level of skeletons as ``global" convergence of the boundary structure of the tessellations, whereas convergence of typical cells has a ``local" nature, describing the similarity of the tessellations in terms of the averaged (typical) cell. Despite this intuition, convergence of skeletons alone - even locally with high probability - does not, in general, imply weak convergence of the typical cells; see Section \ref{sec:TypicalCell}. The reason is that small perturbations of the skeleton can create a huge number of additional tiny cells, which may significantly change the distribution of the typical cell. It was shown in \cite[Proposition 6.4]{sectional}, however, that local convergence with high probability together with convergence of cell intensities does imply weak convergence of the typical cells, and we apply this result to Poisson-Laguerre tessellations (see Theorem \ref{thm:ConvergenceTypicalCell}).

\paragraph{Structure of the paper.} 
The remainder of the article is organized as follows. Section \ref{sec:Preliminaries} introduces the notation
and necessary background on random Laguerre tessellations used throughout. In Section \ref{sec:limitTess} we discuss the concept of random closed sets and convergence of random tessellation, and then state our main results, Theorem \ref{thm:StrongProbConvLaguerre} and Theorem \ref{thm:StrongConvergenceVoronoi}. Section \ref{sec:TypicalCell} is devoted to convergence of typical cells, including Theorem \ref{thm:ConvergenceTypicalCell}, while in Section \ref{sec:Applications} we consider a number of applications of our general convergence criteria. Section \ref{sec:Technical} collects the stabilization lemmas needed for the proofs. Section \ref{sec:ProofStrongProbConvLaguerre} contains the proofs of Theorem \ref{thm:StrongProbConvLaguerre} and Theorem \ref{thm:ConvergenceTypicalCell}, and in Section \ref{sec:ProofStrongConvergenceVoronoi} we prove Theorem \ref{thm:StrongConvergenceVoronoi}. Finally, Section \ref{sec:proofs} provides proofs of the technical lemmas.

\section{Preliminaries}\label{sec:Preliminaries}

\subsection{Frequently used notation}\label{subsec:FUN}
Let $d\geq 2$ and $A\subset\RR^d$. We denote by $\inter A$ the interior of $A$, by $\bd A$ its boundary, by $\cl A$ its closure and by $A^c=\RR^d\setminus A$ its complement. Further by $\conv(A)$ we denote the convex hull of $A$. If $A$ is a countable set we denote by $\# A$ its cardinality. For any $a>0$ we also set $aA:=\{ax\colon x\in A\}$ and for $A,B\subset \RR^d$ we denote by $A+B:=\{a+b\colon a\in A, b\in B\}$ their Minkowski sum. 

A centered closed Euclidean ball in $\RR^d$ with radius $r>0$ is denoted by $B_r^d$ and we put $\BB^d:=B_1^d$. When it is obvious from the context we might leave out the superscript $d$ and just write $B_r$. The volume of $\BB^d$ is given by
\[
    \kappa_d:=\lambda_d(\BB^d)=\frac{\pi^{d/2}}{\Gamma(\frac{d}{2}+1)},
\]
where we denote by $\lambda_d$ the Lebesgue measure on $\RR^d$. By $\tilde\sigma_{d-1}$ we denote the uniform distribution on the unit sphere $\SS^{d-1}$, this is the unique rotationally invariant probability measure on $\SS^{d-1}$ (as discussed in \cite[p.584]{SW}).

We denote by $\cF$ and $\cC$ the set of closed and compact subsets of $\RR^d$, respectively. Here, the empty set is always included and we write $\cF'$ and $\cC'$ if we explicitly want to exclude $\varnothing$. There is a natural topology on $\cF$, called the Fell topology (see \cite[Definition 2.1.1]{SW}). By a convex body in $\RR^d$ we mean a compact and convex subset of $\RR^d$. The set of all convex bodies in $\RR^d$ is denoted by $\cK$, and again, we write $\cK'$ if we want to exclude the empty set. 

In what follows we shall represent points $x\in\RR^{d}\times \RR$ in the form $x=(v,h)$ with $v\in\RR^{d}$ (called \textit{spatial} coordinate) and $h\in\RR$ (called \textit{weight}, \textit{height} or \textit{time} coordinate). By $\proj_{\RR^d}: (v,h)\mapsto v$ we denote an orthogonal projection operator from $\RR^{d}\times \RR$ to $\RR^d$.

\subsection{Random Laguerre tessellations}\label{subsec:laguerreTess}

In this subsection we recall the concept of a random tessellation and introduce the class of Poisson-Laguerre tessellations. For a more detailed discussion on general random tessellations we refer the reader to \cite[Chapter 2, Chapter 10]{SW}. 

\paragraph{Tessellations.} By a tessellation $T$ in $\RR^d$ we understand a locally finite system of convex polytopes (cells) that cover the whole space, have non-empty disjoint interiors and such that for any distinct $t,t'\in T$ their intersection $t\cap t'$ is either empty or is a face of $t$ and $t'$. Due to the letter property we may define the set of $k$-dimensional faces of a tessellation $T$ as the set of $k$-dimensional faces of all its cells. Then $T$ is called {\textit{normal}} if any $k$-dimensional face is contained in precisely $d+1-k$ cells for all $k\in\{0,1,\ldots,d-1\}$. We denote by $\TT$ the set of all tessellations in $\RR^d$ and by a {\textit{random tessellation}} we understand a particle process $X$ (i.e. point process in $\mathcal{C}'$) satisfying $X\in\TT$ almost surely (see \cite[Section~4.1]{SW} for more details). 

\paragraph{Construction of Laguerre tessellations and their duals.} Next we introduce Laguerre tessellations and give a construction based on the so called power function. Given two points $v, w \in \RR^d$ and $h \in \RR$ we define the \textit{power} of $w$ with respect to the pair $(v,h)$ as
\begin{align*}
    \pow\left(w,(v,h)\right) := \|w-v\|^2 + h.
\end{align*} 
Let $A$ be a countable set of points of the form $(v,h) \in \RR^d\times \RR$. We define the \textit{Laguerre cell} of 
$(v,h) \in A$ as
\begin{align*}
    C\left((v,h), A\right) := \big\{w \in \RR^d: \pow(w,(v,h)) \le \pow(w,(v',h')) \text{ for all } (v',h') \in A\big\}. 
\end{align*} 
The collection of all Laguerre cells of $A$ with non-vanishing topological interior is called the \textit{Laguerre diagram}
\[
    \cL(A):=\{C((v,h),A)\colon (v,h)\in A,\, \inter{C((v,h),A)}\neq\varnothing \}.
\]

An alternative description of the Laguerre cell as an intersection of closed half-spaces (which will be used later) can be obtained as follows. Given two distinct points $(v_1,h_1), (v_2,h_2) \in \RR^d \times \RR$ the set of solutions of the equation $\pow(z,(v_1,h_1)=\pow(z,(v_2,h_2))$ is a hyperplane
\[
    H\big((v_1,h_1),(v_2,h_2)\big):=\{z\in \RR^d:2\langle z, v_1-v_2\rangle = \|v_1\|^2-\|v_2\|^2+h_1-h_2\}
\]
which is perpendicular to the line passing through $v_1$ and $v_2$. We write
\[
    H^+\big((v_1,h_1),(v_2,h_2)\big):=\{z\in \RR^d:2\langle z, v_1-v_2\rangle \ge \|v_1\|^2-\|v_2\|^2+h_1-h_2\}
\]
for the closed half-space. The Laguerre cell of $(v,h)\in A$ can then be obtained as
\begin{align}\label{eq:halfspaces}
    C((v,h),A)=\bigcap_{(v',h')\in A} H^+\big((v,h),(v',h')\big).
\end{align}

It should be mentioned that by \eqref{eq:halfspaces} any non-empty Laguerre cell is convex and closed, but it is not necessarily bounded. Moreover, a Laguerre diagram is not necessarily a locally finite covering of the whole space. Thus, whether $\cL(A)$ is a tessellation or not, depends strongly on the geometric properties of the point set $A$ and it is known \cite[Proposition 1]{Sch93} that under mild regularity assumption on $A$ it holds that $\cL(A)\in\TT$. If $\cL(A)$ is a normal tessellation, then we denote by $\cL^*(A)$ the \textit{dual Laguerre tessellation}. This tessellation arises from $\cL(A)$ as follows: for distinct points $(v_1,h_1),\ldots,(v_{d+1},h_{d+1})$ the simplex $\conv(v_1,\ldots,v_{d+1})$ belongs to $\cL^*(A)$ if and only if 
\[
\bigcap_{i=1}^{d+1}C((v_i,h_i),A)\neq \varnothing.
\]
Note, that since $\cL(A)$ is normal the above intersection (if not empty) consists of a single point, which is a vertex of $\cL(A)$. It follows from \cite[Proposition 2]{Sch93} that $\cL^*(A)$ is a tessellation. 

One interesting special case arises, when all points $(v,h)\in A$ have the same weight $h\equiv h_0\in \RR$. In this case the diagram $\cL(A)$ coincides with the well-known \textit{Voronoi diagram} $\cV(A')$ of the set $A':=\{v\colon (v,h)\in A\}$. Again, if $\cV(A')$ is a normal tessellation, then the dual tessellation $\cD(A')$ is called \textit{Delaunay tessellation}.

\medskip

\paragraph{Models of random tessellations.} In what follow we will be interested in random Laguerre tessellations, which are obtained by choosing $A$ to be a Poisson point processes in $\RR^d\times \RR$. More precisely, we consider for an interval $E\subset \RR$ the class
\[
    L_{{\rm loc}}^{1,+}(E):=\Big\{f\colon \inter E \to \RR_+ \colon \int_Kf(x)\dd x<\infty \, \forall K\subset E, \, K\text{ compact}\Big\},
\] 
of non-negative, locally integrable functions on $E$ and given a function $f\in L_{\rm loc}^{1,+}(E)$ we choose $A$ to be a Poisson point process $\eta_{f}$ with intensity measure of the form
\begin{equation}\label{eq:IntensityMeasure}
    \Lambda_{f}(\cdot)= \int_{\RR^d}\int_{E} f(h) \ind((v,h)\in \cdot)\,\dd h \,\dd v,
\end{equation} 
which is a locally finite diffuse measure on $\RR^d\times E$. We additionally assume, that $\Lambda_f$ is non-trivial, namely $\Lambda_f([0,1]^d\times \RR)>0$. These settings have been introduced in \cite{gusakova2024PLT}, where it has been shown that the diagram $\cL(\eta_{f})$ is an almost surely normal random tessellation, if $f$ fulfills some natural integrability conditions. In order to describe these conditions we will need the notion of a fractional integral. For $\alpha >0$ and a measurable function $f\colon \RR\to \RR_+$ we denote by
\begin{equation*}
    \big(\frI^{\alpha}f\big)(x):=\frac{1}{\Gamma\left(\alpha\right)} \int_{-\infty}^{x} f(t) (x-t)^{\alpha-1}\dd t,\qquad x>0,
\end{equation*}
the \textit{fractional integral} of $f$ of order $\alpha$. We also note that any non-negative measurable function $f$ fulfills the semigroup property, i.e. for $\alpha, \beta > 0$ we have
\begin{equation}\label{eq:SemigroupProp}
    \frI^\alpha \frI^\beta f=\frI^\beta \frI^\alpha f = \frI^{\alpha+\beta} f.
\end{equation}

We will call a function $f\in L_{\rm loc}^{1,+}(E)$ \textit{admissible} if \cite[Definition 3.4]{gusakova2024PLT}: 
\begin{itemize}
    \item (i) $E=[a,\infty)$ for some $a\in \RR$; (ii) $E=(-\infty,b)$ for some $b\in \RR$; or (iii) $E=\RR$,
    \item $\big(\frI^{{d\over 2}+1}f\big)(t)<\infty$ for all $t\in E$,
    \item if $E$ is of type (ii), then there exists $\varepsilon>0$ and $n_0\in\NN$, such that for all $n>n_0$ it holds that
    \[
        \big(\frI^{{d\over 2}+1}f\big)(b-1/n)\ge n^{\varepsilon}.
    \]
\end{itemize} 
If $f$ is admissible, then $\cL(\eta_{f})$ is an almost surely normal random tessellation and its dual tessellation $\cL^*(\eta_{f})$ is an almost surely simplicial random tessellation in $\RR^{d}$ (see \cite[Theorem 3.3]{gusakova2024PLT}). It should also be noted, that if $E=[a,\infty)$, then any function $f\in L_{\rm loc}^{1,+}(E)$ is admissible \cite[Proposition 3.6]{gusakova2024PLT}. We will call the class of Laguerre tessellations $\cL(\eta_{f})$ with $f$ admissible the \textit{Poisson-Laguerre tessellations} with density. The following useful identity describes the behavior of the Poisson-Laguerre tessellation $\cL(\eta_{f})$ under linear transformations of the density $f$ (see \cite[Proposition 3.12]{gusakova2024PLT}): for a function $\varphi:\RR\mapsto\RR$, $\varphi(x)=\lambda x+c$ with $\lambda>0$ and $c\in\RR$, we have
\begin{equation}\label{eq:LinearTransformation}
    \cL\big(\eta_{f}\big)\overset{d}{=}\sqrt{\lambda}\,\cL\Big(\eta_{\lambda^{{d\over 2}+1}(f\circ \varphi)}\Big),
\end{equation}
where for $T\in \TT$ and $a>0$ we write $aT:=\{at\colon t\in T\}$.

Another model we will be working with is obtained as follows. Let $\eta^{\gamma}$ be a homogeneous Poisson point process on $\RR^d$ with intensity $\gamma >0$. We will view $\eta^{\gamma}$ as a point process in $\RR^d\times \RR$ by identifying every point $v\in  \eta^{\gamma}$ with the point $(v,0)\in \RR^d\times \RR$. The random tessellation $\cL(\eta^{\gamma})$ is called \textit{Poisson-Voronoi tessellation} and its dual tessellation $\cL^*(\eta^{\gamma})$ \textit{Poisson-Delaunay tessellation}. These are the classical models of random tessellations which have been studied in detail, see \cite[Section~10.2]{SW}. It should be noted here, that the Poisson-Voronoi tessellation does not belong to the class of Poisson-Laguerre tessellations with density.

As it is common to identify a simple stationary point process $\eta$ with its support, we write $\eta\cap K$ for the restriction of $\eta$ to $K$, where $K\subset\RR^d\times \RR$. When $K$ is bounded, we have that $\eta\cap K$ is almost surely finite. We note that if $\cL(\eta)$ is a tessellation, then the Laguerre diagram $\cL\big(\eta\cap{K}\big)$ can still be constructed and is as well a locally finite system of convex sets, that covers $\RR^d$, but these sets are not necessarily compact anymore. Constructing the dual diagram $\cL^*\big(\eta\cap{K}\big)$ is still possible but now leads to a locally finite system of convex compact sets, which do not cover the whole space. 

\section{Limit theorems for random Laguerre tessellations}\label{sec:limitTess}

In this section we present the main results of this article on convergence of Poisson-Laguerre tessellations. We begin by introducing the notions of convergence for random tessellations that we will use throughout.

\subsection{Convergence of random tessellations and random closed sets}\label{subsec:randomTessAndClosedSets}

It should be noted, that the space of tessellations $\TT$ does not admit a natural topological structure and it is not known, whether this space is Polish (see discussion in \cite[Section 2.3]{HT19}). In order to overcome this difficulty it will often be convenient to identify a stationary random tessellation $X$ with its skeleton, i.e.
\[
    \sk{X}=\skel(X) := \bigcup_{c\in X}\bd c.
\]
By the discussion in \cite[p.\ 464]{SW} it follows that $\sk{X}$ is a stationary \textit{random closed set}, namely a measurable map from some probability space $\Omega$ to $\cF$ equipped with the Fell topology. In particular, it is known that $\cF$ is a Polish space \cite[p. 567]{Mol17} with possible metric given by \cite[Equation (C.1)]{Mol17}. In what follows, we will denote the skeleton of the random Poisson-Laguerre tessellation $\cL(\eta)$ by $\sk{L}(\eta)$ and the skeleton of the dual Poisson-Laguerre tessellation $\cL^*(\eta)$ by $\sk{L}^*(\eta)$.

In this article, by convergence of random tessellations we mean the convergence of the corresponding skeletons. There are several different notions for convergence for a sequence $(Z_n)_{n\in\NN}$ of random closed sets (see \cite{Mol17} more details). The first natural notion is \textit{weak convergence} \cite[Definition 1.7.1]{Mol17}. Instead of the definition we will rely on the following useful criteria \cite[Theorem 1.7.7]{Mol17} formulated in terms of pointwise convergence of the corresponding capacity functionals. The \textit{capacity functional} $T_Z$ of a random closed set $Z$ is defined by $T_Z(C):=\PP(Z\cap C\neq \varnothing)$ for $C\in \cC$. Then a family of random closed sets $(Z_n)_{n\in\NN}$ converges weakly to a random closed set $Z$ if and only if
\begin{equation}\label{eq:CapacityConvergence}
    \lim_{n\to \infty}T_{Z_n}(C) = T_Z(C)
\end{equation}
for all $C \in \cC$ such that $T_Z(\inter C)=T_Z(\cl C)$. Another notion is \textit{convergence in probability} \cite[Definition 1.7.25]{Mol17}. We say that a family of random closed sets $(Z_n)_{n\in\NN}$ converges in probability to a random closed set $Z$ if for any $\varepsilon>0$  and $C\in\cC$ we have
\begin{equation}\label{eq:ConvProb}
    \lim_{n\to\infty}\PP\big([(Z_n\setminus Z^{\varepsilon-})\cup (Z\setminus Z_n^{\varepsilon-})]\cap C\neq \varnothing\big)=0,
\end{equation}
where for a closed set $F$ the set $F^{\varepsilon-}$ is the open $\varepsilon$-envelop of $F$, namely
\[
F^{\varepsilon-}=\{z\in\RR^d\colon \inf_{x\in F}\|z-x\|<\varepsilon\}.
\]
This notion of convergence is stronger than the weak convergence, see \cite[Corollary 1.7.29]{Mol17}. 

We note that the Fell topology and the corresponding weak convergence is widely used in the theory of random closed sets. At the same time, when applied to skeletons of random tessellations, weak convergence does not imply the convergence of many interesting quantities and in particular it does not imply the convergence of the typical cell (see Section \ref{sec:TypicalCell} for further discussion). This motivates us to introduce another much stronger notion of convergence.

\begin{definition}\label{def:StrongConvergence}
    We say that a family of random closed sets $(Z_n)_{n\in\NN}$ \textit{converges locally with high probability} to a random closed set $Z$ if for every $R>0$ it holds that
    \[
    \lim_{n\to\infty}\PP(Z_n\cap B_{R}\neq Z\cap B_R)=0.
    \]
\end{definition}

\begin{proposition}\label{prop:ImplicationWeakConv}
    If $(Z_n)_{n\in\NN}$ converges to $Z$ locally with high probability, then $(Z_n)_{n\in\NN}$ converges to $Z$ in probability and weakly.
\end{proposition}

\begin{proof}
    Let $C\in\cC$ be arbitrary and let $R>0$ be such that $C\subset B_R$. Further let $\varepsilon >0$, $Y_n:=Z_n\cap B_{R}$ and $Y:=Z\cap B_{R}$. We note that
    \begin{align*}
    (Z_n\setminus Z^{\varepsilon-})\cap C&=(Z_n\cap C)\setminus (Z^{\varepsilon-}\cap C)=(Y_n\cap C)\setminus (Z^{\varepsilon-}\cap B_R\cap C)=(Y_n\setminus (Z^{\varepsilon-}\cap B_R))\cap C.
    \end{align*}
    Similarly 
    \begin{align*}
    (Z\setminus Z_n^{\varepsilon-})\cap C&=(Y\setminus (Z_n^{\varepsilon-}\cap B_R))\cap C.
    \end{align*}
    Assume $Y_n=Y$, then $(Y_n\setminus (Z^{\varepsilon-}\cap B_R))\cup (Y\setminus(Z_n^{\varepsilon-}\cap B_R))=\varnothing$ for any $\varepsilon >0$. Hence,
    \[
    \lim_{n\to\infty}\PP\big([(Z_n\setminus Z^{\varepsilon-})\cup (Z\setminus Z_n^{\varepsilon-})]\cap C\neq \varnothing\big)\leq \lim_{n\to\infty}\PP(Y_n\neq Y)=0,
    \]
    implying that $(Z_n)_{n\in\NN}$ converges to $Z$ in probability and as a consequence of \cite[Corollary 1.7.29]{Mol17} also weakly.
\end{proof}

In what follows, for simplicity, we will sometimes abuse the notation and say that the sequence of random tessellations $(\cT)_{n\in\NN}$ converges weakly/in probability/locally with high probability if the sequence of corresponding skeletons converge in this sense.

\subsection{Convergence in the class of Poisson-Laguerre tessellations with densities}\label{sec:LaguerreConvergence}

We start by establishing sufficient conditions for a sequence of Poisson-Laguerre tessellations with densities that ensure the convergence to another Poisson-Laguerre tessellation with a density. Let $(f_n)_{n\in\NN}$ be a sequence ofadmissible functions satisfying the following conditions:
\begin{enumerate}
    \item[(C1)] 
    \begin{enumerate}
        \item[(i)] For any $n\in\NN$ we have that $f_n\in L_{\rm loc}^{1,+}(E_n)$ with $E_n\uparrow E$ as $n\to\infty$;
        \item[(ii)]  for all $x\in E$ we have 
        \begin{align*}
            \lim_{n\to\infty} \int_{-\infty}^x |f_n(h)-f(h)|\dd h = 0, 
        \end{align*}
        with some admissible function $f\in L_{\rm loc}^{1,+}(E)$, satisfying $(\frI^1f)(1/2)>0$;
        \item[(iii)] and for some fixed $x_0\in\RR$ and some $\delta>0$ it holds that
        \[
            M:=\sup_{n\in\NN}\int_{-\infty}^{x_0}|h|^{{d\over 2}+\delta}f_n(h)\dd h<\infty.
        \]
    \end{enumerate}
    \label{item:C1}
\end{enumerate}

We note that the condition $(\frI^1 f)(1/2)>0$ is not restrictive since $\lim_{x\to\infty}(\frI^1 f)(x)>0$ and $\frI^1 f$ is monotone. Hence, there exists $a\in\RR$ such that $(\frI^1 f)(a)>0$. Since by \eqref{eq:LinearTransformation} the distribution of $\cL(f)$ does not change under applying the transformation $\varphi(x)=x+(1/2-a)$ to $f$ we may without loss of generality assume that $a=1/2$. This choice will be convenient for the proof.

The main result of this subsection is the following theorem, which will be proven in Section \ref{sec:ProofStrongProbConvLaguerre}.

\begin{theorem}\label{thm:StrongProbConvLaguerre}
    Let $\eta_{n}:=\eta_{f_n}$, $n\in\NN$, be a sequence of Poisson point processes with densities $(f_{n})_{n\in\NN}$ satisfying \hyperref[item:C1]{(C1)} and let $\eta:=\eta_{f}$. Then there exist point processes ${\widetilde\eta}$ and $({\widetilde\eta}_n)_{n\in\NN}$ defined on a common probability space with $\widetilde\eta\overset{d}{=}\eta$ and ${\widetilde\eta}_n\overset{d}{=}\eta_n$ for all $n\in \NN$, such that 
    \begin{enumerate}
        \item[(i)] $\sk{L}(\widetilde\eta_n)$ converges locally with high probability to $\sk{L}(\widetilde\eta)$ as $n\to\infty$;
        \item[(ii)] $\sk{L}^*(\widetilde\eta_n)$ converges locally with high probability to $\sk{L}^*(\widetilde\eta)$ as $n\to\infty$.
    \end{enumerate}
\end{theorem}

\begin{remark}
    We believe that our approach extends to the situation when the intensity measure of Poisson point processes $\eta_n$ has the form $\lambda_d\otimes \Lambda_n$ for some locally finite measure $\Lambda_n$ on $\RR$ satisfying additional integrability assumption. One would have to ensure first that the corresponding Laguerre tessellation $\cL(\eta_n)$ is well-defined following the steps of \cite[Theorem 3.3]{gusakova2024PLT}. In order to keep the article short and for simplicity we omit considering this general setup here and concentrate of the model introduced in \cite{gusakova2024PLT}.
\end{remark}

As a direct consequence of Theorem \ref{thm:StrongProbConvLaguerre} and Proposition \ref{prop:ImplicationWeakConv} we obtain.

\begin{theorem}\label{thm:weakConvPPPonSameSpaceDual}
   Let $\eta_{n}:=\eta_{f_n}$, $n\in\NN$, be a sequence of Poisson point processes with densities $(f_{n})_{n\in\NN}$ satisfying \hyperref[item:C1]{(C1)} and let $\eta:=\eta_{f}$. Then the random closed sets $\sk{L}(\eta_n)$ and $\sk{L}^*(\eta_n)$ converge weakly to $\sk{L}(\eta)$ and $\sk{L}^*(\eta)$, respectively, as $n\to\infty$.
\end{theorem}

\subsection{Poisson-Voronoi tessellation as weak limit of Poisson-Laguerre tessellations}\label{sec:VoronoiLimit}

In this section we show that the skeletons of the classical Poisson-Voronoi and Poisson-Delaunay tessellations arise as a weak limit of $\sk{L}(\eta_{f_n})$ and $\sk{L}^*(\eta_{f_n})$, respectively, for a sequence of functions $(f_n)_{n\in\NN}$ satisfying the following condition:
\begin{enumerate}
    \item[(C2)] Let $(f_n)_{n\in \NN}\subset L_{\rm loc}^{1,+}(\RR_+)$ be a sequence of 
    functions satisfying 
    \begin{align*}
        \lim_{n\to\infty} (\frI^1 f_n)(x) = \gamma\in (0,\infty) \text{ for all } x>0.
    \end{align*} \label{item:C2}
    \end{enumerate}
    We note, that this setting is different from the one considered in the previous section, since the Poisson-Voronoi tessellation does not belong to the class of Poisson-Laguerre tessellations with densities.

The main result of this section is the following theorem, which will be proven in Section \ref{sec:ProofStrongConvergenceVoronoi}.

\begin{theorem}\label{thm:StrongConvergenceVoronoi}
    Let $\eta_n:=\eta_{f_n}$, $n\in\NN$, be a sequence of Poisson point processes with densities $(f_{n})_{n\in\NN}$ satisfying \hyperref[item:C2]{(C2)} with $\gamma>0$ and let $\eta^{\gamma}$ be a homogeneous Poisson point process in $\RR^d$ with intensity $\gamma$. Then there exist point processes ${\widetilde\eta}^{\gamma}$ and $({\widetilde\eta}_n)_{n\in\NN}$ defined on a common probability space with ${\widetilde\eta}^{\gamma}\overset{d}{=}\eta^\gamma$ and ${\widetilde\eta}_n\overset{d}{=}\eta_n$ for all $n\in \NN$, such that 
    \begin{enumerate}
        \item[(i)] $\sk{L}^*(\widetilde\eta_n)$ converges locally with high probability to $\sk{L}^*(\widetilde\eta^{\gamma})$ as $n\to\infty$;
        \item[(ii)] $\sk{L}(\widetilde\eta_n)$ converges in probability to $\sk{L}(\widetilde\eta^{\gamma})$ as $n\to\infty$.
    \end{enumerate}
\end{theorem}

\begin{remark}\label{rem:NoStrongConvVoronoi}
    We note that the part (ii) of Theorem \ref{thm:StrongConvergenceVoronoi} cannot be extended to the stronger notion of "local convergence with high probability" using our method. The obstacle lies in the construction of the tessellation $\sk{L}(\widetilde\eta_n)$. The intuition behind the proof is as follows. We find a coupling of Poisson point processes $((\widetilde\eta_n)_{n\in\NN},\widetilde\eta^{\gamma})$, such that (locally) we may view the point process $\widetilde\eta_n$ for $n$ large enough as the point process $\widetilde\eta^{\gamma}$ where each point $(v,0)\in \widetilde\eta^{\gamma}$ is lifted independently and randomly along the height axes and this "distortion" is getting small as $n\to\infty$. 
    
    Note that the vertices of the dual Laguerre tessellation $\sk{L}^*(\widetilde\eta_n)$ are formed by the spatial coordinates $v$ of the points $(v,h)\in \widetilde\eta_n$, although not every point $(v,h)$ give rise to a vertex. This means that the small "distortion" of the point process $\widetilde\eta^{\gamma}$ will either not change the tessellation at all or will result in disappearance of some vertices. Thus, with our construction $\sk{L}^*(\widetilde\eta_n)$ and $\sk{L}^*(\widetilde\eta^{\gamma})$ will not coincide (locally) only if not every $(v,h)\in \widetilde\eta_n$ give rise to a vertex of $\sk{L}^*(\widetilde\eta_n)$. The probability of this event is getting small as the "distortion" is getting small.

    At the same time in case of Laguerre tessellation every small "distortion" of the point process $\widetilde\eta^{\gamma}$ will result in small "distortion" of the boundaries of the cells, meaning that with our construction the probability that $\sk{L}(\widetilde\eta_n)$ and $\sk{L}(\widetilde\eta^{\gamma})$ coincide locally is always $0$ for any $n\in\NN$.
\end{remark}

As a direct consequence of the above theorem and Proposition \ref{prop:ImplicationWeakConv} we obtain the following result.

\begin{theorem}\label{thm:weakConv}
   Let $\eta_n:=\eta_{f_n}$, $n\in\NN$, be a sequence of Poisson point processes with densities $(f_{n})_{n\in\NN}$ satisfying \hyperref[item:C2]{(C2)} with $\gamma>0$ and let $\eta^{\gamma}$ be a homogeneous Poisson point process in $\RR^d$ with intensity $\gamma$. Then the random closed sets $\sk{L}(\eta_n)$ and $\sk{L}^*(\eta_n)$ converge weakly to $\sk{L}(\eta^{\gamma})$ and $\sk{L}^*(\eta^{\gamma})$, respectively, as $n\to\infty$.
\end{theorem}

\begin{remark}
Note, that the Poisson point process $\eta^{\gamma}$ (embedded in $\RR^d\times\RR$ by identifying $v\in\eta^{\gamma}$ with $(v,0)$) has intensity measure $\gamma \lambda_d\otimes \delta_0$, where $\delta_0$ denotes the Dirac mass at $0$, which is not of the form \eqref{eq:IntensityMeasure}. Thus, the Poisson-Voronoi tessellation $\cL(\eta^{\gamma})$ does not belong to the family of Poisson-Laguerre tessellations with densities. On the other hand, by 
Theorem \ref{thm:weakConv} the classical Poisson-Voronoi and Poisson-Delaunay tessellations can be included in the family as limiting cases.
\end{remark}

\subsection{Convergence of typical cells}\label{sec:TypicalCell}

In this section let $\cT$ be stationary random tessellation. An important characteristic of $\cT$ is its typical cell, which may intuitively be viewed as a random polytope representing a cells of $\cT$ chosen "uniformly at random". More precisely, let $z:\cC'\times \TT\to \RR^d$ be a measurable function satisfying $z(C+v,T+v)=z(C,T)$ for any $x\in\RR^d$, $C\in \cC'$ and $T\in\TT$. Such a function is called (generalized) center function. The cell intensity of a stationary random tessellation $\cT$ is defined as
\[
    \gamma_d(\cT):=\EE \sum_{t\in \cT} \ind(z(t,\cT)\in [0,1]^d).
\]
For $\gamma_d(\cT)\in(0,\infty)$ the typical cell $Z(\cT)$ of $\cT$ is defined as the random polytope with distribution
\[
    \PP_{\cT}^0(\cdot):=\frac{1}{\gamma_d(\cT)} \EE \sum_{t\in\cT} \ind(t-z(t,\cT)\in \cdot)\ind(z(t,\cT)\in [0,1]^d).
\]
For more details we refer the reader to \cite[Chapter 10]{SW}. It is known that the cell intensity $\gamma_d(\cT)$ does not depend on the choice of the center function $z$, while the dependence of the distribution of the typical cell $\PP_{\cT}^0$ on $z$ is only up to shift operation $C\mapsto C-z(C)$ (see \cite[Theorem 4.2.1]{SW}).

The following question now arises naturally. Let $(\cT_n)_{n\in\NN}$ be the sequence of stationary random tessellations, which converge weakly/in probability/locally with high probability to a  stationary random tessellation $\cT$ as $n\to\infty$. Does the distribution of the typical cell $\PP_{\cT_n}^0$ converges weakly to $\PP_{\cT}^0$ as $n\to\infty$? Conversely, does the weak convergence of the typical cells imply weak convergence of the corresponding tessellations? When the sequence $(\cT_n)_{n\in\NN}$ converges to $\cT$ locally with high probability the following result was proven in \cite{sectional}.

\begin{proposition}[Proposition 6.4 in \cite{sectional}]\label{prop:TypicalCellConvergence}
    Let $(\cT_n)_{n\in \NN}$ be a sequence of stationary random tessellations in $\RR^d$, which converges locally with high probability to a stationary random tessellation $\cT$ as $n\to\infty$. Also, suppose that $\lim\limits_{n\to\infty} \gamma_d(\cT_n) = \gamma_d(\cT)$, and that all these intensities are finite. Then
    \[
		\PP^{0}_{\cT_n} {\overset{}{\underset{n\to\infty}\longrightarrow}} \PP^{0}_{\cT},
		\qquad
		\text{ weakly on } \cC'.
    \]
\end{proposition}

\begin{remark}
    We conjecture that the above proposition holds true under weaker assumptions on convergence of random tessellations, namely that convergence locally with high probability can be replaced by convergence in probability or even by weak convergence. We do not explore this direction further in this article and leave it for future research.
\end{remark}

We note that an extra condition regarding the convergence of cell intensities is indeed necessary (see Remark \ref{rem:CellIntNecessary}). At the same time it might be hard to verify, since often we do not have access to the exact formulas of $\gamma_d(\cT)$. Even for $\cT=\cL(\eta_f)$ or $\cT=\cL^*(\eta_f)$ such formulas are available only for very special choices of $f$ (see \cite[Section 6.3]{GKT20}, \cite[Section 5.2]{GKT21} and \cite[Theorem 5.1]{gusakova2024PLT}). 

Using the duality relation, instead of the cell intensity of $\cT$ we may consider the vertex intensity for the dual tessellation $\cT^*$, which is often easier to analyze. More precisely, given a stationary random tessellation $\cT$ denote by $\cF_0(\cT)$ the set of all vertices of all its cells and by $\gamma_0(\cT):=\EE[\#(\cF_0(\cT)\cap [0,1]^d)]$ the vertex intensity of $\cT$. It holds that $\gamma_d(\cT)=\gamma_0(\cT^*)$. The following lemma provides a simple sufficient condition for the convergence of vertex intensities.

\begin{lemma}\label{lem:convCellIntensities}
    Let $(\cT_n)_{n\in \NN}$ be a sequence of stationary random tessellations on $\RR^d$, which converges locally with high probability to a stationary random tessellation $\cT$ as $n\to\infty$. Assume that $\gamma_0(\cT)$ and $\gamma_0(\cT_n)$ are finite for any $n\in\NN$. Moreover, assume that for some $p>1$ we have
    \[
        \sup_{n\in\NN}\EE\Big[\#(\cF_0(\cT_n)\cap[0,1]^d)^p\Big]<\infty.
    \]
    Then $\lim_{n\to\infty}\gamma_0(\cT_n)=\gamma_0(\cT)$.
\end{lemma}
\begin{proof} 
Let $p>1$ and set  $N_n:=\#(\cF_0(\cT_n)\cap[0,1]^d)$ and $N:=\#(\cF_0(\cT)\cap[0,1]^d)$. We start by noting that for $n\in \NN$ the event $\cD_n:=\{\cT_n\cap [0,1]^d= \cT\cap [0,1]^d\}$ implies that $N_n=N$ and by Definition \ref{def:StrongConvergence} we have $\lim_{n\to\infty}\PP(\cD_n^c)=0$. Then 
\[
    \PP(N\neq N_n)\leq \PP(\cD_n^c)\overset{n\to\infty}{\longrightarrow} 0,
\]
and, hence, $N_n\to N$ in probability as $n\to\infty$. Then there is a subsequence $(N_{n_k})_{k\in\NN}$, such that $N_{n_k}\to N$ almost surely as $k\to\infty$ and by Fatou's lemma 
\[
    \EE[N^p]\leq \liminf_{k\to\infty}\EE[N_{n_k}^p]\leq \sup_{n\in\NN}\EE[N_n^p]<\infty.
\]
Then setting $C:= \sup_{n\in\NN}\EE[N_n^p]<\infty$ and using H\"older's inequality we conclude
\[
    \begin{aligned}
        |\gamma_0(\cT_n)-\gamma_0(\cT)|&= \EE \Big[\big|\#(\cF_0(\cT_n)\cap [0,1]^d)- \#(\cF_0(\cT)\cap [0,1]^d)\big|\ind(\cD_n^c)\Big]\\
        &\le 2C^{1/p}\,\big(\PP(\cD_n^c)\big)^{1-1/p}\overset{n\to\infty}{\longrightarrow} 0,
    \end{aligned}
\]
and the proof follows.
\end{proof}

Now using Theorem \ref{thm:StrongProbConvLaguerre}, Proposition \ref{prop:TypicalCellConvergence} and Lemma \ref{lem:convCellIntensities} we can formulate sufficient conditions for a sequence of densities $(f_n)_{n\in\NN}$ satisfying \hyperref[item:C1]{(C1)} such that the typical cells of the corresponding Laguerre tessellations converge. The proof of this theorem will be given in Section \ref{sec:ProofStrongProbConvLaguerre}.

\begin{theorem}\label{thm:ConvergenceTypicalCell}
     Let $d\ge 2$. Let $\eta_{n}:=\eta_{f_n}$, $n\in\NN$, be a sequence of Poisson point processes with densities $(f_{n})_{n\in\NN}$ satisfying \hyperref[item:C1]{(C1)} and let $\eta:=\eta_{f}$. Additionally, we assume that either of the following two conditions hold:
     \begin{enumerate}
         \item [(i)] There exists $n_0\in \NN$, constants $C\in (0,\infty)$, $\alpha >0$, $\varepsilon>0$, and $x_0\in E$
         such that for all $n\ge n_0$ and all $x\ge x_0$ we have
         \[
            \big(\frI^1f_n\big)(x+4d+2)\le C \big(\frI^1 f_n\big)(x),\qquad \big(\frI^1f_n\big)(x)\ge \alpha x^\varepsilon;
         \]
        \item[(ii)] For some $p>1$ we have $\sup_{n\in \NN}a_n^p<\infty$, where $a_n=\int_\RR f_n(x)\dd x$. 
     \end{enumerate}
     
     Then we have
    \[
        \PP_{\cL(\eta_n)}^0{\overset{}{\underset{n\to\infty}\longrightarrow}} \PP_{\cL(\eta)}^0, \qquad
		\text{ weakly on } \cC'.
    \]
\end{theorem}

Next we will show that the convergence of cell intensities is a necessary condition in Proposition \ref{prop:TypicalCellConvergence} and that the weak convergence of the typical cell does not imply weak convergence of the skeletons.

\begin{remark}[Convergence of cell intensities is necessary]\label{rem:CellIntNecessary}
    Consider the following example. Let $U_n\sim\unif([0,2^{-n})^d)$, $n\in\NN_0$, be a sequence of independent random vectors and let $\cT^{(n)}:=\{U_n+2^{-n}z+[0,2^{-n}]^d\colon z\in \ZZ^d\}$ be a stationary random tessellation, i.e. a uniformly shifted lattice at scale $2^{-n}$. We define the sequence of stationary random tessellations $(\cT_n)_{n\in\NN}$ as follows:
    \[
    \cT_n:=\begin{cases}
        \cT^{(n)}\text{ with probability }p_n=2^{-nd/2},\\
        \cT^{(0)}\text{ with probability }1-p_n,
    \end{cases}
    \]
    and set $\cT=\cT^{(0)}$ almost surely. Then, for any $R>0$, we have
    \begin{align*}
        \PP(\skel(\cT_n)\cap B_R\neq \skel(\cT)\cap B_R)\leq \PP(\cT_n=\cT^{(n)})=p_n\to 0,
    \end{align*}
    as $n\to\infty$, implying that $(\cT_n)_{n\in\NN}$ converges to $\cT$ locally with high probability as $n\to\infty$. At the same time $\gamma_d(\cT)=1$, while 
    \[
        \gamma_d(\cT_n)=(1-p_n)+2^{nd}p_n=1-2^{-nd/2}+2^{nd/2}\to\infty,
    \]
    as $n\to\infty$. We will now show that $\PP^0_{\cT_n}$ does not converge to $\PP^0_{\cT}$ weakly. Indeed let us for simplicity choose the center function $z$ to be the lexicographically smallest vertex. Then $Z(\cT)=[0,1]^d$ almost surely, while for $n\in\NN$ we have
    \[
        \PP^0_{\cT_n}(\cdot)=\ind\{[0,2^{-n}]^d\in\cdot\}{2^{nd/2}\over 1-2^{-nd/2}+2^{nd/2}}+\ind\{[0,1]^d\in\cdot\}{1-2^{-nd/2}\over 1-2^{-nd/2}+2^{nd/2}},
    \]
    implying that $Z(\cT_n)$ converges to a singleton $\{o\}$ in distribution as $n\to\infty$.
    
    By Proposition \ref{prop:ImplicationWeakConv} this also means that weak convergence and convergence in probability for tessellations do not imply weak convergence of the corresponding typical cells.
\end{remark}

\begin{remark}[Weak convergence of the typical cells does not imply weak convergence of the corresponding tessellations] It is sufficient to show that there are two distinct stationary random tessellations whose typical cells have the same distribution. We consider the intervals 
\begin{align*}
    I_1^{(1)}&:=[0,1/6],\, I_2^{(1)}:=[1/6,1/3],\, I_3^{(1)}:=[1/3,2/3],\, I_4^{(1)}:=[2/3,1], \\
    I_1^{(2)}&:=[0,1/3],\, I_2^{(2)}:=[1/3,1/2],\, I_3^{(2)}:=[1/2,5/6],\, I_4^{(2)}:=[5/6,1], 
\end{align*}
and define two distinct tilings of $[0,1]^2$ by
\begin{align*}
    C_{j,k}^{(i)}:=I_j^{(i)} \times I_k^{(i)},\quad i=1,2,\,j,k=1,2,3,4.
\end{align*}
Further let $U\sim\unif([0,1)^2)$ and define the stationary random tessellations in $\RR^2$
\begin{align*}
    \cT_i&:=\big\{C_{j,k}^{(i)}+z+U\colon j,k=1,2,3,4,\,z\in\ZZ^2\big\},\qquad i=1,2.
\end{align*} 
Let us choose the center function $z$ to be the lexicographically smallest vertex. Note that $\gamma_2(\cT_1)=\gamma_2(\cT_2)=16$ and 
\[
B:=\{C^{(1)}_{j,k}-z(C^{(1)}_{j,k})\colon j,k=1,2,3,4\}=\{C^{(2)}_{j,k}-z(C^{(2)}_{j,k})\colon j,k=1,2,3,4\}.
\]
Then 
\begin{align*}
    \PP_{\cT_1}^0(\cdot)&=\frac{1}{16} \EE \sum_{C\in B} \ind\big(C\in \cdot\big)=\PP_{\cT_2}^0(\cdot),
\end{align*}
but $\skel(\cT_1)\neq \skel(\cT_2)$ almost surely.
    
\end{remark}

\subsection{Applications}\label{sec:Applications}

We will finish this section by considering a few applications of our results in some special settings considered in the literature before.
 
\paragraph{Laguerre tessellation of independently marked Poisson point process. } Let $q$ be a density of some probability measure $\QQ$ on $\RR_+$ and consider the family of functions $f_n(h):=\gamma n\,q(nh)\in  L^{1,+}_{\rm loc}(\RR_+)$, $n\in\NN$. Note, that by the marking theorem for Poisson point processes \cite[Theorem 5.6]{LP} the point process $\eta_n:=\eta_{f_n}$ has the same distribution as an independent $\QQ_n$-marking of a homogeneous Poisson point process $\eta^{\gamma}$, where
\[
\QQ_n(\cdot)=n\int_{\RR^+}q(nh)\ind(h\in\cdot)\dint h,
\]
is a probability measure on $\RR_+$. 

\begin{corollary}\label{cor:WeightsTo}
    Let $f_n(h)=\gamma n\,q(nh)$, $n \in \NN$, for some $q\in L^{1,+}_{\rm loc}(\RR_+)$ satisfying $\int_{\RR_+}q(h)\dint h=1$. The random closed sets $\sk{L}(\eta_n)$ and $\sk{L}^*(\eta_n)$ converge weakly to $\sk{L}(\eta^{\gamma})$ and $\sk{L}^*(\eta^{\gamma})$, respectively, as $n\to\infty$.
\end{corollary}
\begin{proof}
    For any $x>0$ 
    by the monotone convergence theorem we have 
    \[
    (\frI^1 f_n)(x)=\gamma n\int_{0}^xq(nh)\dint h=\gamma\int_{0}^{nx}q(h)\dint h \xrightarrow[n\to \infty]{}\gamma \int_{0}^{\infty}q(h)\dint h=\gamma.
    \]
    Hence, $f_n$ fulfills \hyperref[item:C2]{(C2)}  and the proof follows by Theorem \ref{thm:weakConv}.
\end{proof}

Geometrically, we can interpret this result as follows. As argued above $\eta_n$ can be identified with an independent $\QQ_n$-marking of $\eta_{\gamma}$. As $n\to\infty$ the mark (weight) of every point will approach $0$ which means that the corresponding Laguerre (weighted) cell will approach the Voronoi (unweighed) cell. This example has been considered in \cite{Ldoc}, where it was shown that the typical cell of $\cL(\eta_n)$ converges in distribution to the typical cell of $\cL(\eta^{\gamma})$. Corollary \ref{cor:WeightsTo} provides a counterpart of this result for convergence on the level of skeletons.

\paragraph{$\beta$-Voronoi and $\beta$-Delaunay tessellations.} Another application is related to the $\beta$-type tessellations (introduced in \cite{GKT20} and further studied in a series of papers \cite{GKT21b,sectional,GKT21,GKT21a}). For $\beta > -1$ consider a Poisson point process $\eta_{\beta}$ in $\RR^{d}\times \RR_{+}$ having intensity measure of the form \eqref{eq:IntensityMeasure} with density
\begin{equation*}
    f_\beta(h)=c_{d+1,\beta}h^{\beta},\qquad h\ge 0, \qquad c_{d+1,\beta}:=\frac{\Gamma\left(\frac{d}{2}+\beta+2\right)}{\pi^{\frac{d}{2}+1}\Gamma(\beta+1)}.
\end{equation*} 
Note, that in \cite{GKT20} a slightly different normalization, namely $\pi^{-\frac{d+1}{2}}\Gamma\left(\frac{d+1}{2}+\beta+1\right)/\Gamma(\beta+1)$ has been considered. The normal random tessellation $\cL(\eta_{\beta})$ is called \textit{$\beta$-Voronoi tessellation} and the dual tessellation $\cL^*(\eta_{\beta})$ is called the \textit{$\beta$-Delaunay tessellation}. The names $\beta$-Voronoi and $\beta$-Delaunay allude the connection of $\cL(\eta_\beta)$ and its dual $\cL^*(\eta_\beta)$ to the classical Poisson-Voronoi and Poisson-Delaunay tessellations, respectively. This was mentioned without a proof in \cite{GKT20}, where it says that the classical Poisson–Voronoi and Poisson–Delaunay tessellations arise as the limiting cases of $\beta$-Voronoi and $\beta$-Delaunay tessellations, respectively, when $\beta\to-1$. At the same time in \cite[Remark 6]{GKT20} it was shown that the typical cell of $\cL^*(\eta_{\beta})$ converges in distribution to the typical cell of $\cL(\eta^{\gamma_d})$ with $\gamma_d=\pi^{-\frac{d}{2}+1}\Gamma\left(\frac{d}{2}+1\right)$ as $\beta\to-1$. Further, in \cite[Proposition 3.1]{sectional} it was proven that the underlying point process $\eta_{\beta}$ converges weakly on the space of locally finite integer-valued measures on $\RR^d\times \RR_+$ to $\eta^{\gamma_d}$. Moreover, in \cite[Theorem 4.1]{sectional} it was shown that the intersection of $\cL(\eta^{\gamma_d})$ with an $\ell$-dimensional affine subspace  up to isometry a random tessellation in $\RR^{\ell}$, which has the same distribution as an $\ell$-dimensional $\beta$-Voronoi tessellation with $\beta=-1+(d-\ell)/2$. Each of these facts serve as an evidence for the Poisson-Voronoi and Poisson-Delaunay tessellations being the limit of $\beta$-Voronoi and $\beta$-Delaunay tessellations, respectively, when $\beta\to-1$. Using Theorem \ref{thm:weakConv} we can now give a proof of this convergence in a sense of convergence of the corresponding skeletons.

\begin{corollary}\label{cor:BetaTo-1}
    The random closed sets $\sk{L}(\eta_\beta)$ and $\sk{L}^*(\eta_\beta)$ converge weakly as $\beta\to -1$ to $\sk{L}(\eta^{\gamma_d})$ and $\sk{L}^*(\eta^{\gamma_d})$, respectively, where $\gamma_d=\pi^{-\frac{d}{2}-1}\Gamma\left(\frac{d}{2}+1\right)$.
\end{corollary}

\begin{remark}
This corollary shows, in particular, that the Poisson-Voronoi and Poisson-Delaunay tessellations can be considered as members of the $\beta$-family: They can be viewed as $\beta$-Voronoi and $\beta$-Delaunay tessellations, respectively, corresponding to the parameter value $\beta=-1$.
\end{remark}

\begin{proof}
    Consider some sequence $(\beta_n)_{n\in\NN}\subset (-1,\infty)$ such that $\beta_n \to -1$ as $n \to \infty$. By Theorem \ref{thm:weakConv} it is sufficient to show that 
    \[
        f_n(h):=\frac{\Gamma\left(\frac{d}{2}+\beta_n+2\right)}{\pi^{\frac{d}{2}+1}\Gamma(\beta_n+1)}h^{\beta_n}, \qquad h\ge 0
    \]
    fulfills condition  \hyperref[item:C2]{(C2)}. For $x>0$ we have 
    \[
        (\frI^1 f_n)(x)=\frac{\Gamma\left(\frac{d}{2}+\beta_n+2\right)}{\pi^{\frac{d}{2}+1}\Gamma(\beta_n+1)}\int_0^x h^{\beta_n} \dd h = \frac{\Gamma\left(\frac{d}{2}+\beta_n+2\right)}{\pi^{\frac{d}{2}+1}\Gamma(\beta_n+2)} x^{\beta_n+1} \underset{n\to \infty}{\longrightarrow} \frac{\Gamma\left(\frac{d}{2}+1\right)}{\pi^{\frac{d}{2}+1}}=\gamma_d \in (0,\infty).
    \]
\end{proof}

\paragraph{Gaussian-Voronoi and Gaussian-Delaunay tessellations.} Let $\widetilde\eta$ be a Poisson point process in $\RR^d\times \RR$ with intensity measure of the form \eqref{eq:IntensityMeasure} with density $\widetilde f(h)=(2\pi)^{-{d\over 2}-1}e^{h/2}$. The tessellation $\cL(\widetilde \eta)$ is called \textit{Gaussian-Voronoi tessellation} and its dual $\cL^*(\widetilde \eta)$ is called \textit{Gaussian-Delaunay tessellation}. These tessellations have been introduced in \cite{GKT21} (with normalization $(2\pi)^{-{d\over 2}}$), where it was also shown that $\sk{L}^*(\widetilde\eta)$ is a weak limit of the skeletons of $\beta$-Delaunay tessellation after scaling by $\sqrt{2\beta}$ and as $\beta\to\infty$. Using our general Theorem \ref{thm:weakConvPPPonSameSpaceDual} we can recover this result and prove that the same holds for the corresponding Voronoi-type tessellations. Furthermore, for the $\beta$-Voronoi tessellation we can also show weak convergence of the typical cells using Theorem \ref{thm:ConvergenceTypicalCell}. We note that the weak convergence of typical cell of $\beta$-Delaunay tessellation (rescaled by $\sqrt{2\beta}$) to the typical cell of Gaussian-Delaunay tessellation as $\beta\to\infty$ is a consequence of the exact stochastic representation of the corresponding typical cells obtained in \cite[Theorem 1]{GKT20} and \cite[Theorem 5.1]{GKT21}.

\begin{corollary}\label{cor:BetaToGaussian}
    The random closed sets $\sqrt{2\beta}\sk{L}(\eta_{\beta})$ and $\sqrt{2\beta}\sk{L}^*(\eta_{\beta})$ converge weakly to $\sk{L}(\widetilde\eta)$ and $\sk{L}^*(\widetilde\eta)$, respectively, as $\beta\to\infty$. Moreover, the typical cell of $\sqrt{2\beta}\cL(\eta_{\beta})$ converges in distribution to the typical cell of $\cL(\widetilde{\eta})$ as $\beta \to \infty$. 
\end{corollary}

\begin{proof}
    Let $(\beta_n)_{n\in\NN}\subset [1,\infty)$ be a monotone sequence such that $\beta_n\to\infty$ and define 
    \begin{align}\label{eq:f_n}
    f_n(h):=c_{d+1,\beta_n}(2\beta_n)^{-\frac{d}{2}-1} \Big(1+\frac{h}{2\beta_n}\Big)^{\beta_n} \ind(h\ge -2\beta_n), \quad c_{d+1,\beta_n}:=\frac{\Gamma\left(\frac{d}{2}+\beta_n+2\right)}{\pi^{\frac{d}{2}+1}\Gamma(\beta_n+1)}.
    \end{align}
    By \eqref{eq:LinearTransformation} we have that $\sqrt{2\beta_n}\cL(\eta_{\beta_n})\overset{d}{=}\cL(\eta_{f_n})$.
    
    We start by showing that $(f_n)_{n\in\NN}$ satisfies \hyperref[item:C1]{(C1)} with $f(h)=(2\pi)^{-{d\over 2}-1}e^{h/2}$. It is easy to ensure that $f_n\in L_{\rm loc}^{1,+}([-2\beta_n,\infty))$ and $[-2\beta_n,\infty)\uparrow \RR$ as $n\to\infty$. By \cite[Proposition 3.6]{gusakova2024PLT} this directly implies that $f_n$ is admissible for every $n\in\NN$.
    
    We further note that by \cite[Paragraph 5.11.12]{NIST}
    \begin{equation}\label{eq:28.11.25_1}
    \lim_{n\to\infty}c_{d+1,\beta_n}(2\beta_n)^{-\frac{d}{2}-1}=\lim_{\beta_n\to\infty}\frac{\Gamma\left(\frac{d}{2}+\beta_n+2\right)}{(2\beta_n\pi)^{\frac{d}{2}+1}\Gamma(\beta_n+1)}=(2\pi)^{-{d\over 2}-1},
    \end{equation}
    and, moreover, for $\beta_n\ge 1$ by \cite[Paragraph 5.5.1]{NIST} we have
    \begin{equation}\label{eq:28.11.25_2}
        c_{d+1,\beta_n}(2\beta_n)^{-\frac{d}{2}-1}=\frac{(\frac{d}{2}+\beta_n+1)\cdots (\beta_n + \{\frac{d}{2}\})\Gamma\left(\beta_n+\{\frac{d}{2}\}\right)}{(2\beta_n\pi)^{\frac{d}{2}+1}\Gamma(\beta_n+1)}<\Big(\frac{\frac{d}{2}+2}{2\pi}\Big)^{\frac{d}{2}+1},
    \end{equation}
    where we denote by $\{\frac{d}{2}\}$ the fractional part of $\frac{d}{2}$.
    Further, since $1+y\leq e^y$ for any $y>-1$ we have that for any $h\in \RR$ and any $n\in\NN$ it holds that
    \[
        \Big(1+\frac{h}{2\beta_n}\Big)^{\beta_n} \ind(h\ge -2\beta_n)\leq e^{h/2}.
    \]
    Since for any $x\in \RR$ we have that $e^{h/2}\ind(h\leq x)$ is integrable, we conclude  by the dominated convergence theorem and using \eqref{eq:28.11.25_1} and \eqref{eq:28.11.25_2} that for any $x\in\RR$ we have 
    \begin{align*}
        \lim_{n\to\infty}\int_{-\infty}^x |f_n(h)-f(h)|\dd h
        &\leq \Big(\frac{\frac{d}{2}+2}{2\pi}\Big)^{\frac{d}{2}+1}\lim_{n\to\infty} \int_{-\infty}^x  \Big(e^{h/2} - \Big(1+\frac{h}{2\beta_n}\Big)^{\beta_n} \ind(h\ge -2\beta_n)\Big) \dd h\\
        &\qquad+2e^{x/2}\lim_{n\to\infty}|c_{d+1,\beta_n}(2\beta_n)^{-\frac{d}{2}-1}-(2\pi)^{-\frac{d}{2}-1}|=0.
    \end{align*}
    In order to verify the third condition in \hyperref[item:C1]{(C1)}, we choose $x_0=0$, $\delta=1$ and note that by \eqref{eq:28.11.25_2} and the estimate $1+y\leq e^y$ for $y\ge 1$ it holds that
    \[
        \int_{-\infty}^0|h|^{{d\over 2}+1}f_n(h)\dd h\leq \Big(\frac{\frac{d}{2}+2}{2\pi}\Big)^{\frac{d}{2}+1}\int_{-\infty}^0|h|^{{d\over 2}+1}e^{h/2}\dd h=\Big(\frac{\frac{d}{2}+2}{\pi}\Big)^{\frac{d}{2}+1}\Gamma\Big({d\over 2}\Big),
    \]
    for any $n\in\NN$. Thus, the conditions in \hyperref[item:C1]{(C1)} are satisfied. By Theorem \ref{thm:weakConvPPPonSameSpaceDual} the convergence of skeletons follows immediately. To show the convergence of the typical cells we use Theorem \ref{thm:ConvergenceTypicalCell} and show that condition (i) therein holds. For $x\ge -2\beta_n$ we have 
    \[
        (\frI^1 f_n)(x)=c_{d+1,\beta_n}(2\beta_n)^{-\frac{d}{2}} \frac{1}{\beta_n+1} \Big(1+\frac{x}{2\beta_n}\Big)^{\beta_n+1},
    \]
    and $(\frI^1 f_n)(x)=0$ for $x<-2\beta_n$. Choose $x_0=0$. Then for $x\ge 0$ and $\beta_n\ge 1$ we have
    \[
        \frac{(\frI^1 f_n)(x+4d+3)}{(\frI^1 f_n)(x)} = \Big(\frac{1+\frac{x+4d+3}{2\beta_n}}{1+\frac{x}{2\beta_n}}\Big)^{\beta_n+1}=\Big(1+\frac{4d+3}{2\beta_n+x}\Big)^{\beta_n+1}\le \exp\Big(\frac{(4d+3)(\beta_n+1)}{2\beta_n+x}\Big) \le \eee^{4d+3}.
    \]
    Further, for $x\ge 0$ the function $f_n(x)$ is increasing. Hence, for all $x\ge 0$ we have
    \[
        (\frI^1 f_n)(x)\ge x f_n(0), 
    \]
    and
    \[
        f_n(0)=c_{d+1,\beta_n}(2\beta_n)^{-\frac{d}{2}-1}=\frac{\Gamma\left(\frac{d}{2}+\beta_n+2\right)}{(2\beta_n\pi)^{\frac{d}{2}+1}\Gamma(\beta_n+1)} \ton (2\pi)^{-{d\over 2}-1}>0.
    \]
    Then there exists $n_0\in \NN$ such that for all $n\ge n_0$ we have $f_n(0)\ge \frac{1}{2}(2\pi)^{-{d\over 2}-1}=:\alpha>0$ and, hence, $(\frI^1 f_n)(x)\ge \alpha x$.
\end{proof}

Together with the $\beta$-model yet another special model of Poisson-Laguerre tessellations was introduced in \cite{GKT20} and studied in \cite{GKT21b,sectional,GKT21,GKT21a}. For $\beta > d/2+1$ let $\eta'_{\beta}$ be a Poisson point process in $\RR^{d}\times (-\infty,0)$ having intensity measure of the form \eqref{eq:IntensityMeasure} with density
\begin{equation*}
    f'_\beta(h)=c'_{d+1,\beta}(-h)^{-\beta},\qquad h<0, \qquad c'_{d+1,\beta}:=\frac{\Gamma\left(\beta\right)}{\pi^{\frac{d}{2}+1}\Gamma(\beta-{d\over 2}-1)}.
\end{equation*} 
Again, note that a slightly different normalization $\pi^{-\frac{d+1}{2}}\Gamma\left(\beta\right)/\Gamma(\beta-{d+1\over 2})$ was used in \cite{GKT20}. The normal random tessellation $\cL(\eta'_{\beta})$ is called \textit{$\beta'$-Voronoi tessellation} and the dual tessellation $\cL^*(\eta'_{\beta})$ is called the \textit{$\beta'$-Delaunay tessellation}. Further in \cite{GKT21} it was shown that $\sk{L}^*(\widetilde\eta)$ is also a weak limit of the skeletons of $\beta'$-Delaunay tessellation after scaling by $\sqrt{2\beta}$ and as $\beta\to\infty$. With our general Theorem \ref{thm:weakConvPPPonSameSpaceDual} we can also recover this result and extend it to the corresponding Voronoi-type models.

\begin{corollary}\label{cor:BetaPrimeToGaussian}
    The random closed sets $\sqrt{2\beta}\sk{L}(\eta'_{\beta})$ and $\sqrt{2\beta}\sk{L}^*(\eta'_{\beta})$ converge weakly to $\sk{L}(\widetilde\eta)$ and $\sk{L}^*(\widetilde\eta)$, respectively, as $\beta\to\infty$.
\end{corollary}

\begin{proof} 
    The proof is similar to the first part of the proof of Corollary \ref{cor:BetaToGaussian}. Let $(\beta_n)_{n\in\NN}\subset [d+1,\infty)$ be a monotone sequence such that $\beta_n\to\infty$ and define 
    \[
        f_n(h):=c'_{d+1,\beta_n}(2\beta_n)^{-\frac{d}{2}-1} \Big(1-\frac{h}{2\beta_n}\Big)^{-\beta_n} \ind(h< 2\beta_n).
    \]
    By \eqref{eq:LinearTransformation} we have that $\sqrt{2\beta_n}\cL(\eta'_{\beta_n})\overset{d}{=}\cL(\eta_{f_n})$ and by Theorem \ref{thm:weakConvPPPonSameSpaceDual} it is enough to ensure that $(f_n)_{n\in\NN}$ satisfies \hyperref[item:C1]{(C1)} with $f(h)=(2\pi)^{-{d\over 2}-1}e^{h/2}$. We note that $f_n\in L_{\rm loc}^{1,+}((-\infty,-2\beta_n))$ and $(-\infty,-2\beta_n)\uparrow \RR$ as $n\to\infty$. By \cite[Example 3.9 and Proposition 3.12]{gusakova2024PLT} we have that $f_n$ is admissible for every $n\in\NN$.
    
    By \cite[Paragraph 5.11.12]{NIST} we obtain
    \begin{equation}\label{eq:28.11.25_3}
        \lim_{n\to\infty}c'_{d+1,\beta_n}(2\beta_n)^{-\frac{d}{2}-1}=\lim_{\beta_n\to\infty}\frac{\Gamma\left(\beta_n\right)}{(2\beta_n\pi)^{\frac{d}{2}+1}\Gamma(\beta_n-{d\over 2}-1)}=(2\pi)^{-{d\over 2}-1},
    \end{equation}
    and, by \cite[Paragraph 5.5.1]{NIST} we get
    \begin{equation}\label{eq:28.11.25_4}
        c'_{d+1,\beta_n}(2\beta_n)^{-\frac{d}{2}-1}=\frac{(\beta_n-1)\cdots (\beta_n - \lfloor \frac{d}{2} \rfloor)\Gamma\left(\beta_n-\lfloor \frac{d}{2} \rfloor\right)}{(2\beta_n\pi)^{\frac{d}{2}+1}\Gamma(\beta_n-{d\over 2}-1)}<(2\pi)^{-\frac{d}{2}-1}.
    \end{equation}
    Fix $x\in\RR$. Let $n_0=n_0(x)\in\NN$ be such that $x\leq \beta_{n_0}\leq \beta_n$ for any $n>n_0$. 
    Since $\frac{h}{2\beta_n}<1$ we have by \cite[Equation (4.5.2)]{NIST} that 
    \[
        -\log\big(1-\frac{h}{2\beta_n}\big)\ge \frac{h}{2\beta_n}.
    \]
    Multiplying by $\beta_n$ and taking exponentials yields
    \begin{equation*}\label{eq:28.11.25_5}
        e^{h/2}\ind(h\leq x)\leq \Big(1-\frac{h}{2\beta_n}\Big)^{-\beta_n} \ind(h< 2\beta_n)\ind(h\leq x)\leq \Big(1-\frac{h}{2\beta_{n_0}}\Big)^{-\beta_{n_0}}\ind(h\leq x),
    \end{equation*}
    for any $n\ge n_0$. Applying the dominated convergence theorem and \eqref{eq:28.11.25_3} and \eqref{eq:28.11.25_4} we conclude  for any $x\in\RR$ that
    \begin{align*}
        \lim_{n\to\infty}\int_{-\infty}^x |f_n(h)-f(h)|\dd h
        &\leq (2\pi)^{-\frac{d}{2}-1} \lim_{n\to\infty}\int_{-\infty}^x |\Big(1-\frac{h}{2\beta_n}\Big)^{-\beta_n} \ind(h< 2\beta_n)-\eee^{\frac{h}{2}}| \dd h\\
        &\quad+2\eee^{\frac{x}{2}}\lim_{n\to\infty}|c'_{d+1,\beta_n}(2\beta_n)^{-\frac{d}{2}-1} -(2\pi)^{-\frac{d}{2}-1}|=0.
    \end{align*}
    For the third condition in \hyperref[item:C1]{(C1)} we choose $x_0=0$, $\delta=1/2$ and note that by \eqref{eq:28.11.25_4} and \eqref{eq:28.11.25_5} for any $n\in\NN$ it holds
    \[
    \int_{-\infty}^0|h|^{{d+1\over 2}}f_n(h)\dd  h\leq ({2\pi})^{-\frac{d}{2}-1}\int_{-\infty}^0|h|^{{d+1\over 2}}\Big(1-{h\over 2(d+1)}\Big)^{-d-1}\dd h={(2(d+1))^{d+2}\Gamma({d-1\over 2})\Gamma({d+3\over 2})\over ({2\pi})^{\frac{d}{2}+1}d!},
    \]
    since $\beta_n\ge d+1>0$ for any $n\in\NN$. Hence, the conditions in \hyperref[item:C1]{(C1)} are satisfied.
    
\end{proof}

\section{Technical preparations}\label{sec:Technical}

The strategy of the proofs of Theorem \ref{thm:StrongProbConvLaguerre} and Theorem \ref{thm:StrongConvergenceVoronoi} (i) relies on two key steps. In the first step we consider the configuration of a given tessellation, restricted to the ball $B_R$, where $R>0$, and look for the region $K\subset \RR^d\times \RR$, such that with big probability the configuration of a tessellation in a ball depends only on the restriction of the corresponding Poisson point process to $K$. This kind of results are generally called "stabilization". The region $K$ is later used to construct a suitable sequence of point processes, defined on the same probability space using a coupling argument.

In this section we collect a few technical lemmas related to this stabilization. The proofs of these lemmas are postponed to Section \ref{sec:proofs}. 

\subsection{Paraboloid growth and paraboloid hull process}\label{subsec:Construction}

We begin with an alternative description of a Laguerre tessellation and its dual in terms of the paraboloid growth and hull processes, respectively. This description appeared to be very useful for understanding geometric properties of the corresponding tessellation models and provides a useful framework for stabilization results, stated in Section \ref{sec:Stabilization}. The growth processes were first introduced by Baryshnikov \cite{Bar00} and studied later in \cite{SY08}. The paraboloid growth and hull processes were formalized in \cite{CSY13} where they have been used in particular to study the asymptotic properties of random polytopes \cite{CY14, CY15, CY17}.

Let $\Pi^{-}$ (respectively, $\Pi^+$) be the standard downward  (respectively, upward) paraboloid, defined as
\begin{align*}
    \Pi^{\pm}
    &:=
    \{(v',h')\in\RR^{d}\times\RR\colon h'=\pm\|v'\|^2\}.
\end{align*}
Further, let $\Pi^{\pm}_{x}$ be the translation of $\Pi^{\pm}$ by a vector $x:=(v,h)\in\RR^{d+1}$, that is,
\[
    \Pi^{\pm}_x:=\{(v',h')\in\RR^{d}\times\RR\colon h'=\pm\|v'-v\|^2+h\}.
\]
The point $x$ is the \textit{apex} of the paraboloid $\Pi^{\pm}_x$ and we write $\apex(\Pi^{\pm}_x)=x$. Moreover, we denote by
\begin{align*}
(\Pi_x^-)^{\downarrow}:=&\{(v,h')\in\RR^{d}\times\RR\colon h'\leq-\|v'-v\|^2+h\},\\
(\Pi_x^+)^{\uparrow}:=&\{(v,h')\in\RR^{d}\times\RR\colon h'\ge \|v'-v\|^2+h\},
\end{align*}
the \textit{hypograph} and \textit{epigraph} of $\Pi_x^-$ and $\Pi_x^+$, respectively.

\paragraph*{Paraboloid growth process.}

For a given Poisson point process $\eta$ in $\RR^d\times \RR$ the we introduce (as in \cite{CSY13}) the \textit{paraboloid growth process}
\begin{align*}
    \Psi(\eta):=\bigcup\limits_{x\in \eta}\big(\Pi^{+}_{x}\big)^\uparrow.
\end{align*}
It shall be noted that most of these paraboloids do not contribute to the process because their epigraphs are fully covered by the epigraphs of other paraboloids. We call a point $x\in\eta$ \textit{extreme} in the paraboloid growth process $\Psi(\eta)$ if
\begin{align*}
    \big(\Pi^{+}_{x}\big)^\uparrow\not\subset \bigcup\limits_{y\in \eta, y\neq x}\big(\Pi^{+}_{y}\big)^\uparrow.
\end{align*}
We denote by $\ext(\Psi(\eta))$ the set of all extreme points of $\Psi(\eta)$. 

As shown in \cite[p.1486]{sectional} the paraboloid growth process yields an alternative construction of $\cL(\eta)$ if the latter is a random tessellation. Note, that for any $(v,h)\in \eta$,
\[
(\Pi^+_{(v,h)})^{\uparrow}=\{(w,t)\in \RR^d\times\RR\colon \pow(w,(v,h))\leq t\}
\]
and, hence,
\begin{equation}\label{eq:ParaboloidGrowthBoundary}
{\rm bd}\Psi(\eta)=\{(w,t)\in \RR^d\times \RR\colon t=\inf_{(v',h')\in\eta}\pow(w,(v',h'))\},
\end{equation}
where the minimum is achieved since every $w\in\RR^d$ belongs to at least one of the cells $C((v,h),\eta)$, $(v,h)\in\eta$. Then the Laguerre cell of $(v,h)\in \eta$ can be written as
\begin{align*}
    C((v,h),\eta)=\big\{w\in\RR^d\colon (v,h)=\argmin_{(v',h')\in\eta}\pow(w,(v',h'))\big\}=\proj_{\RR^d}\big( \bd  \Psi(\eta) \cap\Pi^+_{(v,h)}\big),
\end{align*}
which is non-empty if and only if $(v,h)\in\ext(\Psi(\eta))$. Hence,
\[
    \cL(\eta):=\{C((v,h),\eta):(v,h)\in \eta, \inter C((v,h),\eta)\neq \varnothing\}=\{C((v,h),\eta):(v,h)\in \ext(\Psi(\eta))\}.
\]
Using this relation, we may alternatively work with the paraboloid growth process $\Psi(\eta)$ instead of $\cL(\eta)$, which is more convenient in some situations. 

\paragraph*{Paraboloid hull process.}

For a given Poisson point process $\eta$ in $\RR^d\times\RR$ we define the \textit{paraboloid hull process} $\Phi(\eta)$, which can be seen as the dual to the paraboloid growth process (see Figure \ref{fig:GrowthAndHull})
\[
    \Phi(\eta)=\bigcup_{x\in \RR^{d}\times\RR\colon \inter (\Pi^-_x)^{\downarrow}\cap\eta=\varnothing}(\Pi^-_x)^{\downarrow}.
\]
For $d+1$ points $x_1=(v_1,h_1),\dots,x_{d+1}=(v_{d+1},h_{d+1})$ with affinely independent spatial coordinates $v_1,\dots,v_{d+1}$ we define by $\Pi^-(x_1,\dots,x_{d+1})$ the unique translate of $\Pi^-$ containing $x_1,\dots,x_{d+1}$ and denote by
\[
    \Pi^-[x_1,\dots,x_{d+1}]:=\Pi^-(x_1,\dots,x_{d+1})\cap(\conv(v_1,\dots,v_{d+1})\times \RR).
\] 
For points $x_1,\ldots,x_{d+1}\in\eta$ with affinely independent spatial coordinates we call the set 
\[
    \Pi^-(x_1,\dots,x_{d+1})\cap \bd \Phi(\eta)
\]
a \textit{paraboloid facet} of $\Phi(\eta)$ if 
\[
\inter\big((\Pi^-(x_1,\dots,x_{d+1}))^{\downarrow}\big)\cap \eta=\varnothing.
\]
We note that in general $\Pi^-(x_1,\dots,x_{d+1})\cap \bd \Phi(\eta)$ may contain further points of $\eta$, but for $\eta_f$ or $\eta^{\gamma}$ this won't happen with probability $1$ and we have
\[
    \Pi^-(x_1,\dots,x_{d+1})\cap \bd \Phi(\eta)=\Pi^-[x_1,\dots,x_{d+1}]
\]
for any paraboloid facet of $\Phi(\eta)$.
The points $x_1,\dots,x_{d+1}$ are called the {vertices} of the corresponding paraboloid facet. The collection of all vertices of $\Phi(\eta)$ is denoted by $\ver(\Phi(\eta))$ and in \cite[Equation (3.17)]{CSY13} it is shown that
\begin{equation}\label{eq:VertExtermal}
    \ext(\Psi(\eta))=\ver(\Phi(\eta)).
\end{equation}

Assuming that $\cL(\eta)$ is a random normal tessellation, we can construct the dual Laguerre tessellation $\cL^*(\eta)$ using the paraboloid hull process $\Phi(\eta)$ the following way: for any pairwise distinct points $x_1=(v_1,h_1),\dots,x_{d+1}=(v_{d+1},h_{d+1})$ from $\eta$ we say that the simplex $\conv(v_1,\dots,v_{d+1})$ belongs to $\cL^*(\eta)$ if and only if $\inter\big((\Pi^-(x_1,\dots,x_{d+1}))^{\downarrow}\big)\cap \eta=\varnothing$, i.e. if and only if $\Pi^-(x_1,\dots,x_{d+1})\cap \bd\Phi(\eta)$ is a paraboloid facet of $\Phi(\eta)$. This also implies that
\begin{equation}\label{eq:ParaboloidFacetSimplex}
    \conv(v_1,\dots,v_{d+1})=\proj_{\RR^d}\big(\Pi^-(x_1,\dots,x_{d+1})\cap \bd \Phi(\eta)\big)=\proj_{\RR^d}\big(\Pi^-[x_1,\dots,x_{d+1}]\big).
\end{equation}
\begin{figure}
\begin{center}
\includegraphics[width=0.8\textwidth]{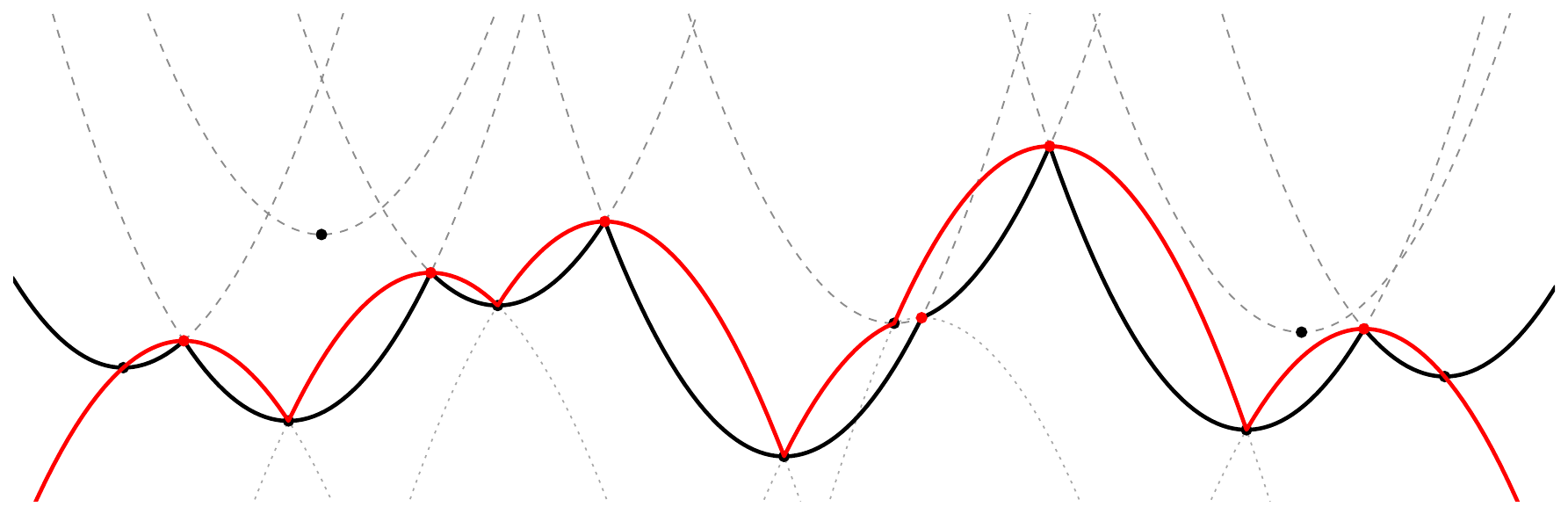}    
  
\end{center}
\caption{The paraboloid growth (black) and hull processes (red) in $\RR^2$.} \label{fig:GrowthAndHull}

\end{figure}

\subsection{Stabilization lemmas}\label{sec:Stabilization}

Let $\eta$ be a Poisson point process in $\RR^d\times \RR$. For $a>0$, $t,T\in \RR$ we define the events 
\begin{equation}\label{eq:Hmin}
\begin{aligned}
    \cH^{\rm{max}}(\eta,a,T)&:=\big\{\sup_{w\in B_a}\inf_{(v,h)\in \eta}\pow(w,(v,h))\le T\big\},\\
    \cH^{\rm{min}}(\eta,a,t)&:=\big\{\inf_{w\in B_a}\inf_{(v,h)\in \eta}\pow(w,(v,h))\ge t\big\}.
\end{aligned}
\end{equation}
By \eqref{eq:ParaboloidGrowthBoundary} the event $\cH^{\rm{max}}(\eta,a,T)$ can be described as follows: the boundary of $\Psi(\eta)$ restricted to $B_a\times \RR$ is contained in $B_a\times(-\infty,T]$, i.e. does not exceed $T$ within the ball $B_a$. Similarly, the event $\cH^{\rm{min}}(\eta,a,t)$ means that the boundary of $\Psi(\eta)$ restricted to $B_a\times \RR$ is contained in $B_a\times[t,\infty)$, i.e. lies above $t$ within the ball $B_a$. The next lemma provides the upper bounds for the probability of the events $\cH^{\rm{max}}(\eta,a,T)$ and $\cH^{\rm{min}}(\eta,a,t)$ with $\eta=\eta_f$, where $f$ is admissible, and $\eta=\eta^{\gamma}$, $\gamma>0$, defined as in Section \ref{subsec:laguerreTess}.

\begin{lemma}\label{lm:BoundaryBounds}
    The following assertions hold. 
        \begin{enumerate}
            \item 
            \begin{enumerate}
                \item Let $f$ be admissible, then for any $a>0$ and $T> 4a^2$ we have 
                    \[
                        \PP\big(\cH^{\rm{max}}(\eta_f,a,T)^c\big)\leq\exp\Big(-\pi^{\frac{d}{2}} 2^{-d}  (\frI^{\frac{d}{2}+1}f)(T-4a^2)\Big).
                    \]
                If moreover $(\frI^1f)(x_0)>0$ for some $x_0\in\RR$, then for any $T\ge 4 a^2+2 x_0$ we have
                \begin{align*}
                    \PP\big(\cH^{\rm{max}}(\eta_f,a,T)^c\big)&\leq \exp\Big(- \big(\frac{\pi}{8}\big)^{\frac{d}{2}} {(\frI^1f)(x_0)\over \Gamma(d/2+1)}(T-4a^2)^{d\over 2}\Big).
                \end{align*}
                \item  Let $f$ be admissible, then for any $a>0$ and $t\in \RR$ we have 
                \[
                    \PP\big(\cH^{\rm{min}}(\eta_f,a,t)^c\big)\leq 2^d\kappa_d \Big((\frI^{\frac{d}{2}+1}f)(t)+a^{d}(\frI^{1}f)(t)\Big).
                \]
            \end{enumerate}

            \item 
            \begin{enumerate}
                \item For any $a>0$ and $T\ge a^2$ we have
                    \[
                        \PP\big(\cH^{\rm{max}}(\eta^\gamma,a,T)^c\big)\leq
                            \exp(-\gamma \kappa_d (\sqrt{T}-a)^d).
                    \]
                \item For any $a>0$ and $t< 0$ we have $\PP\big(\cH^{\rm{min}}(\eta^\gamma,a,t)^c\big) = 0$.
            \end{enumerate}
        \end{enumerate}          
\end{lemma}

Next for a given Poisson point process $\eta$ in $\RR^d\times \RR$ and $R,r>0$ consider the event 
\begin{align}
     \cE(\eta, R,r):=\Big\{\bigcup_{S\in \cL^*(\eta), S\cap B_R\neq \varnothing} \apex (\Pi^-(S)) \in \big(\Pi^{-}_{(\origin,(R+r)^2)}\big)^{\downarrow}\Big\},\label{eq:Eevent}
\end{align}  
where $\Pi^{-}(S)$ determines the unique downward paraboloid containing $d+1$ points $x_i=(v_i,h_i)\in\eta$, $1\leq i\leq d+1$ with affinely independent spatial coordinates, such that $\conv(v_1,\ldots,v_{d+1})=S$. Roughly speaking, this event allows to control the location of points of $\eta$, which determine the skeleton of $\sk{L}^*(\eta)$ within the ball $B_R$.
Using Lemma \ref{lm:BoundaryBounds} we can now obtain an upper bound for the probability of the events $\cE(\eta,a,T)$ with $\eta=\eta_f$, where $f$ is admissible, and $\eta=\eta^{\gamma}$, $\gamma>0$.

\begin{lemma}\label{lm:StabRadius}
    Let $R>0$ and $r>2\max(1,R)$. Then the following assertions hold
    \begin{enumerate}
        \item Let $f$ be an admissible function satisfying $(\frI^1f)(1/2)>0$, then there exist $c_1,C_1\in (0,\infty)$, such that 
        \begin{align*}
            \PP\big(\cE(\eta_{f},R,r)^c\big)<C_1 \Big(\exp(-c_1(\frI^1f)(1/2)\,r^d) + \big(\frI^{\frac{d}{2}+1}f\big)\big(-r^2/8\big)+R^d(\frI^1 f)(-r^2/8)\Big).
        \end{align*}
        \item There exist constants $c_2, C_2\in(0,\infty)$, such that 
        \begin{align*}
            \PP\big(\cE(\eta^\gamma,R,r)^c\big)<C_2\exp(-c_2\gamma\, r^d).
        \end{align*}
    \end{enumerate}
\end{lemma}

Using the events $\cH^{\rm max}$, $\cH^{\rm min}$ and $\cE$ allows us to prove the following "stabilization" result, which states that under additional constrains (described in terms of $\cH^{\rm max}$, $\cH^{\rm min}$ and $\cE$) the restriction of $\eta$ to a certain region $K_0\subset \RR^{d}\times\RR$ determines $\sk{L}(\eta)$ and $\sk{L}^*(\eta)$ within a ball $B_R$.

\begin{lemma}\label{lm:Stabilisation}
    Given $R, r>0$ and $t\leq 0$ define the region
    \[
    K_0(R,r,t):=\{(v,h)\in \RR^d \times [t,(R+r)^2]: \|v\|\le 2\sqrt{(R+r)^2-t}\}.
    \]  
    Then the following assertions hold
    \begin{enumerate}
        \item On the event $\cE(\eta,R,r)\cap \cH^{\rm min}(\eta,2\sqrt{(R+r)^2-t},t)$ we almost surely have that  
        \begin{align*}
            \bigcup_{S\in \cL^*(\eta), S\cap B_R\neq \varnothing} \bd S = \bigcup_{S\in \cL^*(\eta\cap K_0(R,r,t)), S\cap B_R\neq \varnothing} \bd S.
        \end{align*}
        \item On the event $\cH^{\rm max}(\eta,R,(R+r)^2)\cap \cH^{\rm min}(\eta,2\sqrt{(R+r)^2-t},t)$ we almost surely have that 
        \begin{align*}
                    \sk{L}(\eta)\cap B_R=\sk{L}(\eta\cap K_0(R,r,t))\cap B_R.
        \end{align*}
    \end{enumerate}
\end{lemma}

\section{Proof of Theorem \ref{thm:StrongProbConvLaguerre} and Theorem \ref{thm:ConvergenceTypicalCell}}\label{sec:ProofStrongProbConvLaguerre}

In what follows we will need an auxiliary lemma about the properties of the sequence of functions $(f_n)_{n\in\NN}$ satisfying \hyperref[item:C1]{(C1)}.

\begin{lemma}\label{lm:C1properties}
    Let $(f_n)_{n\in\NN}$ be a sequence of functions satisfying \hyperref[item:C1]{(C1)}, then
    \begin{itemize}
        \item[1.] $\int_{-\infty}^{x_0}|h|^{d/2+\delta}f(h)\dd h<\infty$;
        \item[2.] For any positive sequence $(x_n)_{n\in\NN}$ with $x_n\to\infty$ as $n\to\infty$ it holds that
        \begin{align*}
            &\lim_{n\to\infty}x_n^{d/2}(\frI^1 f_n)(-x_n)=0,&\qquad &\lim_{n\to\infty}(\frI^{{d\over 2}+1} f_n)(-x_n)=0,\\
&\lim_{n\to\infty}x_n^{d/2}(\frI^1 f)(-x_n)=0,&\qquad &\lim_{n\to\infty}(\frI^{{d\over 2}+1} f)(-x_n)=0.
        \end{align*}
    \end{itemize}
\end{lemma}

The proof of this lemma uses standard truncation arguments and the monotone convergence theorem and is, thus, postponed to Section \ref{sec:proofs}.

\begin{proof}[Proof of Theorem \ref{thm:StrongProbConvLaguerre}]
    Throughout the proof $C$, $\tilde C$, $c$, $\tilde c$ will denote some positive constants, which only depend on $d$.  Their exact values might be different from line to line. To simplify the representation, in what follows we set $\eta_0:=\eta$,  $\widetilde{\eta}_0:=\widetilde{\eta}$ and $f_0:=f$. In order to prove the statement we need to find a suitable coupling of the sequence of Poisson point processes $(\widetilde{\eta}_n)_{n\in\NN_0}$, allowing us to gain control over the restriction of the corresponding (dual) Laguerre tessellations to the ball $B_R$ for any $R>0$. Thus, for given $n\in\NN_0$ and $R>0$ our first step is to estimate the probability that the restriction of $\eta_n$ to a certain region $K_1\subset \RR^{d}\times\RR$ determines $\sk{L}(\eta_n)\cap B_R$ and $\sk{L}^*(\eta_n)\cap B_R$. This region will be later used to construct $\widetilde{\eta}_n$.
    
    \medskip
    
    \textbf{Preparations:} Let $R, r>0$ and for $n\in\NN_0$ define the auxiliary events (see \eqref{eq:Eevent} and \eqref{eq:Hmin})
    \begin{align*}
            \cE_n&:=\cE({\eta}_n, R,r),\,&\cH^{{\rm min}}_{n}&:=\cH^{{\rm min}}({\eta}_n,3(R+r),-{5/4}(R+r)^2),\, &\cH^{\rm max}_{n}:=\cH^{{\rm max}}({\eta}_n,R,(R+r)^2),
    \end{align*}
    allowing us to control the configuration of $\sk{L}(\eta_n)$ and $\sk{L}^*(\eta_n)$ within the ball $B_R$. More precisely let 
    \[
        K_1(R,r):=\{(v,h)\in \RR^d \times [-5/4(R+r)^2,(R+r)^2]: \|v\|\le 3(R+r)\}.
    \]
    By Lemma \ref{lm:Stabilisation} with $t=-5/4(R+r)^2$ we get that $K_0(R,r,t)= K_1(R,r)$ and, hence,
    on $\cE_n\cap \cH^{{\rm min}}_{n}$ it holds that
    \begin{equation}\label{eq:29_12_25_eq1}
        \bigcup_{S\in \cL^*(\eta_n), S\cap B_R\neq \varnothing} \bd S = \bigcup_{S\in \cL^*(\eta_n\cap K_1(R,r)), S\cap B_R\neq \varnothing} \bd S,
    \end{equation}
    and on $\cH_{n}^{\rm max}\cap \cH_n^{\rm min}$ we have that 
    \begin{equation}\label{eq:02_01_26_eq2}
        \sk{L}(\eta_n)\cap B_R=\sk{L}(\eta_n\cap K_1(R,r))\cap B_R.
    \end{equation}

    Further note that if $(f_n)_{n\in\NN}$ satisfies \hyperref[item:C1]{(C1)} then $c:=(\frI^1 f_0)(1/2)>0$ and due to \hyperref[item:C1]{(C1)} we also have $\lim_{n\to \infty}(\frI^1f_{n})(1/2)=c$. Hence, there exits $n_0\in \NN$ such that $|(\frI^1f_{n})(1/2)-c\big| < c/2$ for all $n\ge n_0$. Then by Lemma \ref{lm:BoundaryBounds} (noting that by monotonicity of the fractional integrals we have $\frI^\alpha f_n(-5/4(R+r)^2)\le \frI^\alpha f_n(-r^2/8)$, for $\alpha=1,\frac{d}{2}+1$ and $R>0$) and Lemma \ref{lm:StabRadius} we conclude for all $n\ge n_0$ and $n=0$ that
    \begin{align}
        \PP((\cE_n\cap \cH^{{\rm min}}_{n})^c)&\leq \PP(\cE_n^c)+\PP((\cH^{{\rm min}}_{n})^c)\notag\\
        &\leq C \Big(\exp(-cr^d) + \big(\frI^{\frac{d}{2}+1}f_n\big)(-r^2/8)+(r+R)^d(\frI^1 f_n)(-r^2/8)\Big).\label{eq:29_12_25_eq2}
    \end{align}
    In the same way, using Lemma \ref{lm:BoundaryBounds} with $x_0=1/2$ and $r\ge 2\max(R,1)$ for $n\ge n_0$ and $n=0$ we obtain
    \begin{align}
        \PP((\cH^{{\rm max}}_{n}\cap \cH^{\rm min}_n)^c)&\leq C \Big(\exp(-cr^d) + \big(\frI^{\frac{d}{2}+1}f_n\big)(-r^2/8)+(r+R)^d(\frI^1 f_n)(-r^2/8)\Big).\label{eq:29_12_25_eq2_2}
    \end{align}
   \textbf{Construction of the point processes:} We are now ready to construct point processes $(\widetilde{\eta}_n)_{n\in\NN_0}$, which are defined on the same probability space and such that $\widetilde{\eta}_n\overset{d}{=}\eta_n$ for all $n\in\NN_0$. Let $(r_n)_{n\in\NN}$ be a monotone increasing sequence (to be specified later) such that $r_n\to\infty$ as $n\to\infty$. Let $\widetilde{\eta}_0$ be a Poisson point process with $\widetilde{\eta}_0\overset{d}{=}\eta_0$ and let
    \[
        \widetilde{\eta}_n=\xi_n+\hat\xi_n\overset{d}{=}\eta_n,\qquad n\in\NN,
    \]     
    where $\xi_n$ and $\hat\xi_n$ are independent, $\xi_n$ is a Poisson point process (on $K_1(r_n/2,r_n)$)  with intensity measure having density $(v,h)\mapsto f_n(h){\bf 1}_{K_1(r_n/2,r_n)}(v,h)$ and $\hat\xi_n$ is a Poisson point process (on $K_1(r_n/2,r_n)^c$) with intensity measure having density $(v,h)\mapsto f_n(h){\bf 1}_{K_1(r_n/2,r_n)^c}(v,h)$. Further we assume, that $\hat\xi_n$ is independent of $\widetilde{\eta}_0$ for all $n\in\NN$, while $\xi_n$, $n\in\NN$, is constructed in such a way that the random elements $\xi_n$ and $\widetilde{\eta}_0\cap K_1(r_n/2,r_n)$ coincide maximally (see \cite[Theorem 7.3, p.107]{thorisson2000coupling}). This leads to the desired construction.

    \medskip
    {\textbf{Estimate for the (dual) Laguerre tessellations:}} It remains to ensure that for any $R>0$ we have that 
    \begin{equation}\label{eq:30_12_25_eq1}
    \lim_{n\to\infty}\PP(\sk{L}^*(\widetilde{\eta}_n)\cap B_R\neq \sk{L}^*(\widetilde{\eta}_0)\cap B_R)=0,
    \end{equation}
    and 
    \begin{equation}\label{eq:02_01_26_eq3}
        \lim_{n\to\infty}\PP(\sk{L}(\widetilde{\eta}_n)\cap B_R\neq \sk{L}(\widetilde{\eta}_0)\cap B_R)=0.
    \end{equation}
    Let $n_1(R)>0$ be such that $r_n\ge 2\max(R,1)$ for any $n>n_1$. By the construction of $(\widetilde{\eta}_n)_{n\in\NN_0}$ and by \eqref{eq:29_12_25_eq1} we have on $\cH^{{\rm min}}_{n}\cap\cE_n\cap \cH^{\rm min}_0\cap \cE_0$ for any $n\ge n_1$ that
    \begin{align*}
    \PP(\{\sk{L}^*(\widetilde{\eta}_n)&\cap B_R\neq \sk{L}^*(\widetilde{\eta}_0)\cap B_R\}\cap \cH^{{\rm min}}_{n}\cap\cE_n\cap \cH_0^{\rm min}\cap \cE_0)\\
    &=\PP(\{\sk{L}^*(\xi_n)\cap B_R\neq \sk{L}^*(\widetilde{\eta}_0\cap K_1(r_n/2,r_n))\cap B_R\}\cap \cH^{{\rm min}}_{n}\cap\cE_n\cap \cH_0^{\rm min}\cap \cE_0),
    \end{align*}
    since $K_1(R,r_n)\subset K_1(r_n/2,r_n)$ and $\widetilde{\eta}_n\cap K_1(r_n/2,r_n)=\xi_n$ almost surely. By the law of total probability we get
    \begin{equation}\label{eq:29.04.25_1}
        \begin{aligned}
            \PP(\sk{L}^*&(\widetilde{\eta}_n)\cap B_R\neq \sk{L}^*(\widetilde{\eta}_0)\cap B_R)\\
            &\leq\PP\big(\sk{L}^*(\xi_n)\cap B_R\neq \sk{L}^*(\widetilde{\eta}_0\cap K_1(r_n/2,r_n))\cap B_R\big)+\PP((\cH^{{\rm min}}_{n}\cap\cE_n\cap\cH_0^{\rm min}\cap\cE_0)^c).  
        \end{aligned}
    \end{equation}
    Further note, that if $\xi_n=\widetilde{\eta}_0\cap K_1(r_n/2,r_n)$, then $\sk{L}^*(\xi_n)\cap B_R= \sk{L}^*(\widetilde{\eta}_0\cap K_1(r_n/2,r_n))\cap B_R$. Hence, again applying the law of total probability we obtain
    \begin{equation}\label{eq:29_12_25_eq4}
        \PP\big(\sk{L}^*(\xi_n)\cap B_R\neq \sk{L}^*(\widetilde{\eta}_0\cap K_1(r_n/2,r_n))\cap B_R\big)\le \PP\big(\xi_n\neq(\widetilde{\eta}_0\cap K_1(r_n/2,r_n))\big).
    \end{equation}
    By construction, $\xi_n$ and $\widetilde{\eta}_0\cap K_1(r_n/2,r_n)$ coincide maximally and, hence, by \cite[Equation (8.18), p.112]{thorisson2000coupling} and \cite[Theorem 3.2.2]{reiss1993course} we have
     \begin{equation}\label{eq:29.04.25_2}
        \begin{aligned}
            \PP\big(\xi_n\neq(\widetilde{\eta}_0\cap K_1(r_n/2,r_n))\big)&\le \frac{3}{2} \int_{{-5r_n^2}}^{4r_n^2} \int_{\|v\|\le 6r_n} |f_0(h)-f_n(h)| \dd v \dd h\\
            &\le C r_n^d \int_{-\infty}^{4 r_n^2} |f_0(h)-f_n(h)|\dd h.
        \end{aligned}
    \end{equation}
    Combining \eqref{eq:29.04.25_1} with \eqref{eq:29_12_25_eq2}, \eqref{eq:29_12_25_eq4} and \eqref{eq:29.04.25_2} for any $n\ge \max(n_0,n_1)$ yields
    \[
    \begin{aligned}
        \PP(\sk{L}^*(\widetilde{\eta}_n)\cap B_R\neq \sk{L}^*(\widetilde{\eta})\cap B_R)&\leq C\Big(\exp(-cr_n^d)+ r_n^d \int_{-\infty}^{4 r_n^2} |f_n(h)-f_0(h)|\dd h+\big(\frI^{\frac{d}{2}+1}f_n\big)(-r_n^2/8)\\
        &\qquad\quad+\big(\frI^{\frac{d}{2}+1}f_0\big)(-r_n^2/8)+r_n^d(\frI^1 f_n)(-r_n^2/8)+r_n^d(\frI^1 f_0)(-r_n^2/8)\Big).
    \end{aligned}
    \]
        
    \medskip

    \textbf{Choosing the sequence $(r_n)_{n\in\NN}$:} We now show that as $n\to\infty$ the right hand side tends to $0$ for a suitable choice of $(r_n)_{n\in\NN}$, implying \eqref{eq:30_12_25_eq1}. Due to \hyperref[item:C1]{(C1)}-(ii) there is a (strictly) increasing sequence $(N_k)_{k\in\NN}$ of natural numbers, such that for any $n\ge N_k$ we have that
    $
    \int_{-\infty}^k|f_n(h)-f_0(h)|\dd h\leq 2^{-k}k^{-d/2}.
    $
    Then we define 
    \begin{equation}\label{eq:rn}
        r_n=\sqrt{k}/2\text{ for }N_k\leq n<N_{k+1},\quad r_n=1/2\text{ for }1\leq n< N_k.
    \end{equation} 
    It is clear that $(r_n)_{n\in\NN}$ is monotone increasing and $r_n\to\infty$ as $n\to\infty$. With this choice we have
    \begin{align}\label{eq:30_12_25_eq2}
        r_n^d \int_{-\infty}^{4 r_n^2} |f_n(h)-f_0(h)|\dd h&=2^{-d}k^{d/2} \int_{-\infty}^{k} |f_n(h)-f_0(h)|\dd h\leq 2^{-k-d}\to 0,
    \end{align}
    as $n\to\infty$ (implying $k\to\infty$).  Further, since $r_n\to\infty$, we immediately get that $\exp(-c r_n^d)\to 0$ as $n\to\infty$, and by choosing $x_n=r_n^2/8$ Lemma \ref{lm:C1properties} implies the convergence to $0$ of the remaining terms. Combining these with \eqref{eq:30_12_25_eq2} leads to \eqref{eq:30_12_25_eq1} and, thus, finishes the proof of (ii).

    Similarly, by \eqref{eq:02_01_26_eq2}, the construction of $(\widetilde\eta_n)_{n\in\NN_0}$ with $(r_n)_{n\in\NN}$ defined as in \eqref{eq:rn} and by the law of total probability we get
    \begin{align*}
        \PP(\sk{L}&(\widetilde{\eta}_n)\cap B_R\neq \sk{L}(\widetilde{\eta})\cap B_R)\\
        &\leq\PP\big(\sk{L}(\xi_n)\cap B_R\neq \sk{L}(\widetilde{\eta}_0\cap K_1(r_n/2,r_n))\cap B_R\big)+\PP((\cH^{{\rm min}}_{n}\cap\cH_n^{\rm max}\cap\cH_0^{\rm min}\cap\cH_0^{\rm max})^c)\\
        &\leq \PP\big(\xi_n\neq(\widetilde{\eta}_0\cap K_1(r_n/2,r_n))\big)+\PP((\cH^{{\rm min}}_{n}\cap\cH_n^{\rm max}\cap\cH_0^{\rm min}\cap\cH_0^{\rm max})^c).
    \end{align*}
    Then the combination of \eqref{eq:29.04.25_2} and \eqref{eq:29_12_25_eq2_2} implies for any $n\ge \max(n_0,n_1)$ that
    \[
    \begin{aligned}
            \PP(\sk{L}(\widetilde{\eta}_n)\cap B_R\neq \sk{L}(\widetilde{\eta})\cap B_R)&\leq C\Big(\exp(-cr_n^d)+ r_n^d \int_{-\infty}^{4 r_n^2} |f_n(h)-f_0(h)|\dd h+\big(\frI^{\frac{d}{2}+1}f_n\big)(-r_n^2/8)\\
            &\qquad\quad+\big(\frI^{\frac{d}{2}+1}f_0\big)(-r_n^2/8)+r_n^d(\frI^1 f_n)(-r_n^2/8)+r_n^d(\frI^1 f_0)(-r_n^2/8)\Big),
    \end{aligned}
    \]
    which converges to $0$ as $n\to\infty$ by the arguments given above. Hence, \eqref{eq:02_01_26_eq3} follows and the proof of (i) is finished.
    \end{proof}

    \begin{proof}[Proof of Theorem \ref{thm:ConvergenceTypicalCell}] 
    Let $c,C,\widetilde{C}>0$ be constants depending only on $d$. Their values might differ from line to line. Let $(\eta_n)_{n\in \NN}$ be a sequence of Poisson point processes with densities $(f_n)_{n\in \NN}$ satisfying \hyperref[item:C1]{(C1)}. 
    By Theorem \ref{thm:StrongProbConvLaguerre} we have that there exist point processes ${\widetilde\eta}$ and $({\widetilde\eta}_n)_{n\in\NN}$ defined on a common probability space with $\widetilde\eta\overset{d}{=}\eta$ and ${\widetilde\eta}_n\overset{d}{=}\eta_n$ for all $n\in \NN$, such that $\sk{L}(\widetilde\eta_n)$ converges locally with high probability to $\sk{L}(\widetilde\eta)$ as $n\to\infty$. Note that $\PP_{\cL(\eta_n)}^0=\PP_{\cL(\widetilde{\eta}_n)}^0$. Using Proposition \ref{prop:TypicalCellConvergence} it is sufficient to show that 
    \[
        \lim_{n\to \infty} \gamma_d(\cL(\widetilde\eta_n))=\gamma_d(\cL(\widetilde\eta)).
    \]
    By \cite[Theorem 10.2.8]{SW} we have $\gamma_0(\cL^*(\widetilde\eta_n))=\gamma_d(\cL(\widetilde\eta_n))$ for any $n\in \NN$ and $\gamma_0(\cL^*(\widetilde\eta))=\gamma_d(\cL(\widetilde\eta))$.  Let $V_n:=\#(\cF_0(\cL^*(\widetilde{\eta}_n))\cap [0,1]^d)$, then showing that there exists some $p>1$ such that \[\sup_{n\in\NN}\EE\big[V_n^p\big]<\infty\] and applying Lemma \ref{lem:convCellIntensities} yields the convergence of cell intensities.
    In particular it is sufficient to show that there exists ${n_0}\in \NN$ such that 
    \begin{align}\label{eq:supremum}
        \sup_{n\ge {n_0}}\EE\big[V_n^2\big]<\infty.    
    \end{align}
    
    We first show that the claim holds given the conditions (i).
    For $k\in \NN$ and $n\in\NN_0$ we define, as in \eqref{eq:Hmin}, the auxiliary event
    \begin{align*}
            \cH_n^{\textrm{max}}(k):=\cH^{{\rm max}}({\eta}_n,\sqrt{d},k).
    \end{align*}
    First we note that $[0,1]^d\subset B_{\sqrt{d}}$. Hence, given the event $\cH_n^{\textrm{max}}(k)$ we have that every vertex of $\cL^*(\widetilde{\eta}_n)$ within $[0,1]^d$ is a point of $\widetilde{\eta}_n$ within the set $B_{\sqrt{d}}\times [-\infty,k]$, 
    and we conclude
    \begin{align*}
        V_n \ind\{\cH_n^{\textrm{max}}(k)\}\le \#\big(\cF_0(\cL^*(\widetilde{\eta}_n))\cap B_{\sqrt{d}}\big) \ind\{\cH_n^{\textrm{max}}(k)\}\le \widetilde{\eta}_n\big(B_{\sqrt{d}}\times [-\infty,k]\big)=:N_n(k)
    \end{align*}
    almost surely. Note that $\cH_n^{\textrm{max}}(k)\uparrow \Omega$ as $k\to \infty$. Applying the bound above and the Cauchy-Schwarz inequality we have
    \begin{align*}
        \EE[V_n^2]&=\EE[V_n^2\ind\{\cH_n^{\textrm{max}}(1)\}]+\sum_{k\ge 1} \EE[V_n^2\ind\{\cH_n^{\textrm{max}}(k+1))\setminus \cH_n^{\textrm{max}}(k)\}]\\
        &\le \EE[N_n(1)^2] + \sum_{k\ge 1} \EE[N_n(k+1)^4]^{1/2} \PP[\cH_n^{\textrm{max}}(k)^c]^{1/2}.
    \end{align*}    
    By Lemma \ref{lm:BoundaryBounds} we have
    \begin{align*}
        \PP[\cH_n^{\textrm{max}}(k)^c]&\leq \exp\big(-c(\frI^{\frac{d}{2}+1}f_n)(k-4d)\big)\le \exp\big(-c(\frI^{1}f_n)(k-4d-1)\big),
    \end{align*}
    for all $k> 4d$. Noting that $N_n(k)$ is Poisson distributed with parameter $C(\frI^1 f_n)(k)$, we have
    \begin{align*}
        \EE[N_n(k)^4]\le \widetilde{C}\big((\frI^1f_n)(k)\big)^4,
    \end{align*}
    and hence
        \begin{align*}
        \EE[V_n^2]&\le C \Big(\sum_{k=1}^{4d} ((\frI^1f_n)(k))^2 + \sum_{k\ge 4d+1} ((\frI^1f_n)(k+1))^2 \exp\big(-c(\frI^{1}f_n)(k-4d-1)\big)\Big)\\
        &= C \Big(\sum_{k=1}^{4d} ((\frI^1f_n)(k))^2 + \sum_{k\ge 0} ((\frI^1f_n)(k+4d+2))^2 \exp\big(-c(\frI^{1}f_n)(k)\big)\Big)
    \end{align*}
    
    By \hyperref[item:C1]{(C1)}-(ii) we have that there exists $n_1\in \NN$ such that for $n\ge n_1$ and all $k=1,\ldots,4d$ it holds
    \[
        (\frI^1 f_n)(k) \le (\frI^1 f_n)(4d) \le \frac{(\frI^1 f)(4d)}{2}\in (0,\infty),
    \]
    where $f\in L_{\rm loc}^{1,+}(E)$ is an admissible function (implying $(\frI^1 f)(4d)<\infty$ by \cite[Lemma 2.2]{gusakova2024PLT}). Hence, the first sum can be bounded by a constant for all $n\ge n_1$. By the same arguments it is sufficient to consider the second sum starting from $k_0:=\lceil x_0 \rceil$ (where $x_0$ is chosen as in the assumptions).  
    Note that $g(y):=y^2\exp(-(c/2)y)$ has a global maximum at $y=4/c$ on $[0,\infty)$ and hence 
    \[
        y^2\exp(-cy)=y^2\exp\big(-\frac{c}{2}y\big)\exp\big(-\frac{c}{2}y\big)\le \frac{16}{c^2 \eee^2} \exp\big(-\frac{c}{2}y\big),
    \]
    for all $y\ge 0$.
    First applying $(\frI^1f_n)(x+4d+2)\le C (\frI^1 f_n)(x)$ and then $(\frI^1f_n)(x)\ge \alpha x^\varepsilon$ yields
    \begin{align*}
        \sum_{k\ge k_0} ((\frI^1f_n)(k+4d+2))^2 \exp\big(-c(\frI^{1}f_n)(k)\big)&\le C \sum_{k\ge k_0} ((\frI^1f_n)(k))^2 \exp\big(-c(\frI^{1}f_n)(k)\big)
        \\
        &\le \frac{16C}{(c\eee)^2} \sum_{k\ge k_0} \exp\Big(-\frac{c}{2} (\frI^{1}f_n)(k) \Big)\\
        &\le \frac{16C}{(c\eee)^2} \sum_{k\ge k_0} \exp\Big(-\frac{c}{2} \alpha k^\varepsilon\Big)
        <\infty,
    \end{align*}
    and \eqref{eq:supremum} follows.

    Under condition (ii) we immediately have
    \begin{align*}
        \sup_{n\in\NN}\EE\Big[\#(\cF_0(\cL^*(\eta_n)\cap[0,1]^d)^p\Big]\le \sup_{n\in\NN}\EE\big[\eta_n([0,1]^d\times \RR^d)^p\big]\le C \max\{1,\sup_{n\in \NN} a_n^p\} <\infty,
    \end{align*}
    and the proof follows.
\end{proof}

\section{Proof of Theorem \ref{thm:StrongConvergenceVoronoi}}\label{sec:ProofStrongConvergenceVoronoi}

We start by formulating a few auxiliary lemmas. Let 
\[
    {\rm AI}^{d+1}:=\{\by:=(y_1,\dots,y_{d+1})\in \RR^d\colon y_1,\dots,y_{d+1}\text{ are affinely independent}\}.
\]
Further for $\by=(y_1,\dots,y_{d+1})\in {\rm AI}^{d+1}$ denote by $B(\by)$ the unique ball containing $y_1,\dots, y_{d+1}$ on its boundary, and let $\rho(\by)>0$ be its radius and $c(\by)\in \RR^d$ its center.

\begin{lemma}\label{lm:continuity}
    The maps (i) ${\rm AI}^{d+1} \to \RR_+,\,\by \mapsto \rho(\by)$ and (ii) ${\rm AI}^{d+1} \to \RR^d, \, \by \mapsto \|c(\by)\|$ are continuous.
\end{lemma}

\begin{lemma}\label{lm:conditionsImply}
    Let $(f_n)_{n\in \NN}\subset L_{\rm loc}^{1,+}(\RR_+)$ be a sequence of functions satisfying \hyperref[item:C2]{(C2)} and define 
    \begin{equation}\label{eq:probDensity}
        g_{n,s}(x):=\frac{s^d}{\pi^{d/2}(\frI^{d/2}f_n)(s^2)}f_n(s^2-s^2\|x\|^2),\quad x\in \BB^d,\, s>0.
    \end{equation}
    Then the following properties hold:
    \begin{itemize}
        \item[1.] $\PP_{n,s}(\cdot):=\int_{\BB^d}g_{n,s}(x)\ind(x\in\cdot)\dd x$ defines a probability measure on $\BB^d$;
        \item[2.] $\PP_{n,s} \underset{n\to\infty}\longrightarrow\tilde\sigma_{d-1}$ weakly for all $s>0$ where $\tilde\sigma_{d-1}$ denotes the uniform distribution on $\SS^{d-1}$.
    \end{itemize}
\end{lemma}

The proofs of these technical lemmas can be found in Section \ref{sec:proofs}.

\begin{proof}[Proof of Theorem \ref{thm:StrongConvergenceVoronoi}]

As in the proof of Theorem \ref{thm:StrongProbConvLaguerre} let $C$, $\tilde C$, $c$, $\tilde c$ be some positive constants, which only depend on $d$, and whose exact value might be different from line to line. For simplicity we set $\eta_0:=\eta^{\gamma}$ and $\widetilde{\eta}_0:=\widetilde{\eta}^{\gamma}$. The strategy of the proof is similar to the one used in Theorem \ref{thm:StrongProbConvLaguerre}. In the first step, for given $n\in\NN_0$ and $R>0$, we estimate the probability that the restriction of $\eta_n$ to a certain region $K_2\subset \RR^{d}\times\RR$ determines $\sk{L}(\eta_n)\cap B_R$ and $\sk{L}^*(\eta_n)\cap B_R$. This region will again be used to construct the processes $\widetilde{\eta}_n$, $n\in\NN_0$, but the additional difficulty in this case arises from the fact that $\widetilde \eta_0$ is a point process in $\RR^d$, while $\widetilde \eta_n$ is a point process in $\RR^d\times \RR_+$. Hence, we may not use the maximal coupling directly as in the proof of Theorem \ref{thm:StrongProbConvLaguerre} and additional arguments are required.

\medskip

\textbf{Preparations:} Let $R, r>0$ and for $n \in\NN_0$ consider the following auxiliary events $\cE_n:=\cE(\eta_n, R,r)$ (see \eqref{eq:Eevent}) and $\cH^{{\rm max}}_{n}:=\cH^{\rm{max}}(\eta_n, R, (R+r)^2)$ (see \eqref{eq:Hmin}). Further define the region
\[
    K_2(R,r):=\big\{(v,h)\in\RR^d\times \RR_+: \|v\|\leq 2(R+r),\,h\leq (R+r)^2\big\}.
\]
We note that since $\eta_n$ is a Poisson point process on $\RR^d\times \RR_+$ we have that $\cH^{\rm min}(\eta_n,a,0)$ holds almost surely for any $a>0$. Hence, by Lemma \ref{lm:Stabilisation} with $t=0$ we note $K_0(R,r,0)=K_2(R,r)$. Then, on $\cE_{n}$ it holds that
\begin{equation}\label{eq:02_01_26_eq4}
    \bigcup_{S\in \cL^*(\eta_n), S\cap B_R\neq \varnothing} \bd S = \bigcup_{S\in \cL^*(\eta_n\cap K_2(R,r)), S\cap B_R\neq \varnothing} \bd S,
\end{equation}
and on $\cH_{n}^{\rm max}$ we have  
\begin{equation}\label{eq:04_01_26_eq4}
    \sk{L}(\eta_n)\cap B_R=\sk{L}(\eta_n\cap K_2(R,r))\cap B_R.
\end{equation}

Further, note that if $(f_n)_{n\in\NN}$ satisfies \hyperref[item:C2]{(C2)} there exits $n_0\in \NN$ such that $|(\frI^1f_{n})(1/2)-\gamma\big| < \gamma/2$ for all $n\ge n_0$. Then for all $n\ge n_0$ and $n=0$ by 
 Lemma \ref{lm:BoundaryBounds} we get
 \begin{equation*}\label{eq:04_01_26_eq5}
    \PP\big(\cH_n^{\rm max})<C\exp(-c\gamma r^d),
 \end{equation*}
 and by Lemma \ref{lm:StabRadius} we obtain
\begin{align}\label{eq:eventBounds}
        \PP\big(\cE_n^c)<C\exp(-c\gamma r^d),
\end{align}    
where we additionally used that $(\frI^{\alpha}f_n)(t)=0$ for any $\alpha>0$, $n\in\NN$ and $t\leq 0$.

\medskip

\textbf{Construction of the point processes:} Next we construct point processes $(\widetilde{\eta}_n)_{n\in\NN_0}$, defined on the same probability space and such that $\widetilde{\eta}_n\overset{d}{=}\eta_n$ for all $n\in\NN_0$. Let $(r_n)_{n\in\NN}$ be a monotone increasing sequence (to be specified later) such that $r_n\to\infty$ as $n\to\infty$ and define $\widetilde{\eta}_0$ to be a homogeneous Poisson point process on $\RR^d$ with intensity $\gamma$ (i.e. $\widetilde{\eta}_0\overset{d}{=}\eta_0$). Further, let
\[
    \widetilde{\eta}_n=\xi_n+\hat\xi_n,\qquad n\in\NN,
\]     
where $\xi_n$ and $\hat\xi_n$ are independent, $\hat\xi_n$ is a Poisson point process (on $K_2(r_n/2,r_n)^c$) with intensity measure having density $(v,h)\mapsto f_n(h){\bf 1}_{K_2(r_n/2,r_n)^c}(v,h)$ and $\xi_n$ is an independent $\QQ_n$-marking of a homogeneous Poisson point process $\xi_n^0$ on $B_{3r_n}$ (note that $B_{3r_n}=\proj_{\RR^d}(K_2(\frac{r_n}{2},r_n))$) of intensity $(\frI^1 f_n)(9 r_n^2/4)$, with
\[
    \QQ_n(B) = \Big(\big(\frI^1 f_n\big)(9r_n^2/4)\Big)^{-1} \int_B f_n(h)\ind(h\in [0,9r_n^2/4]) \dd h, \qquad B \in \cB(\RR_+).
\]
By the marking theorem \cite[Theorem 5.6]{LP} and the superposition theorem \cite[Theorem 3.3]{LP} we have $\widetilde\eta_n\overset{d}{=}\eta_n$. Finally, we assume that $\hat\xi_n$ is independent of $\widetilde \eta_0$ for all $n\in\NN$ and $\xi_n^0$, $n\in\NN$, is constructed in such a way, that the random elements $\xi_n^0$ and $\widetilde \eta_0\cap B_{3r_n}$ coincide maximally (see \cite[Theorem 7.3, p.107]{thorisson2000coupling}). Considering the coupling event 
\[
    \cY_n:=\{{\widetilde\eta}_0\cap B_{3r_n}={\xi}_{n}^{0}\},
\]
by \cite[Equation (8.18), p.112]{thorisson2000coupling} and \cite[Theorem 3.2.2]{reiss1993course} we then obtain
\begin{align}
    \PP\left(\cY_n^c\right) &\le \frac{3}{2} \int_{\{v\in \RR^{d}: \|v\| \le 3r_n\}} \left\lvert \gamma -\big(\frI^1 f_n\big)(9r_n^2/4)\right\rvert \dd v \le C  \left\lvert \gamma -\big(\frI^1 f_n\big)(9r_n^2/4)\right\rvert r_n^d.\label{eq:02_01_26_eq5}
\end{align}

In order to prove (i) we now need to show that for any $R>0$ we have that 
\begin{equation}\label{eq:02_01_26_eq6}
\lim_{n\to\infty}\PP(\sk{L}^*(\widetilde{\eta}_n)\cap B_R\neq \sk{L}^*(\widetilde{\eta}_0)\cap B_R)=0,
\end{equation}
for a suitable choice of the sequence $(r_n)_{n\in\NN}$. Similarly, in order to prove (ii) by \eqref{eq:ConvProb} it is sufficient to ensure that for any $R>0$ and $\varepsilon>0$ we have
\begin{equation}\label{eq:04_02_26_eq6}
\lim_{n\to\infty}\PP\Big(\big[(\sk{L}(\widetilde{\eta}_n)\setminus (\sk{L}(\widetilde{\eta}_0))^{\varepsilon-})\cup (\sk{L}(\widetilde{\eta}_0)\setminus (\sk{L}(\widetilde{\eta}_n))^{\varepsilon-})\big]\cap B_R\neq \varnothing\Big)=0,
\end{equation}
since any compact set $C\in\cC$ is contained in some ball of radius $R>0$. Here we recall that 
\[
    X^{\varepsilon-} = \{z\in \RR^d:\inf_{x\in X} \|z-x\| < \varepsilon\}.
\]

In the next two parts of the proof we will derive bounds for the above probabilities, while in the final part of the proof we will show that both bounds tend to $0$ as $n\to\infty$ for a suitable choice of the sequence $(r_n)_{n\in\NN}$.

\medskip

{\textbf{Estimate for the dual Laguerre tessellations:}} Let $n_1(R)>0$ be such that $r_n\ge 2\max(R,1)$ for any $n>n_1$. By the law of total probability in combination with \eqref{eq:02_01_26_eq4} and the fact that $K_2(R,r_n)\subset K_2(r_n/2,r_n)$ and $\widetilde\eta_n\cap K_2(r_n/2,r_n)=\xi_n$ almost surely, we get
\begin{equation}\label{eq:Intermediate}
    \begin{aligned}
        \PP\big(\sk{L}^*(\widetilde{\eta}_n)&\cap B_R\neq \sk{L}^*(\widetilde{\eta}_0)\cap B_R\big)\\
        &\le\PP\Big(\bigcup_{S\in \cL^*(\widetilde{\eta}_n), S\cap B_R\neq \varnothing} \bd S\neq \bigcup_{S\in \cL^*(\widetilde{\eta}_0), S\cap B_R\neq \varnothing} \bd S\Big)\\
        &= \PP\Big(\Big\{\bigcup_{S\in \cL^*(\xi_n), S\cap B_R\neq \varnothing} \bd S\neq \bigcup_{S\in \cL^*(\widetilde{\eta}_0\cap B_{3r_n}), S\cap B_R\neq \varnothing}\bd S\Big\} \cap \cE_{n}\cap \cE_0\cap \cY_n\Big)\\
        &\hspace{10cm}+\PP\big((\cE_n\cap \cE_0\cap \cY_n)^c\big)\\
        &\le\PP\Big(\Big\{\bigcup_{S\in \cL^*(\xi_n)} \bd S\neq \bigcup_{S\in \cL^*(\xi_n^0)}\bd S\Big\} \cap \cE_{n}\Big)+\PP\big(\cE_n^c\big)+\PP\big(\cE_0^c\big)+\PP\big(\cY_n^c\big),
    \end{aligned}    
\end{equation}
where in the last step we used that $\xi_n^0=\widetilde\eta_0\cap B_{3r_n}$ almost surely on $\cY_n$.

Writing
\begin{equation}\label{eq:Sn}
    \cS_n:=\Big\{\bigcup_{S\in \cL^*(\xi_n)}\bd S= \bigcup_{S\in \cL^*(\xi_{n}^0)} \bd S\Big\}=\big\{\sk{L}^*(\xi_{n}^0)=\sk{L}^*(\xi_{n})\big\},
\end{equation}
our next goal is to estimate $\PP(\cS_n^c\cap \cE_n)$.

We start by noting that by construction $\xi_n^0=\proj_{\RR^d}(\xi_n)$ and, hence, by \eqref{eq:VertExtermal} for any  
$(v,h)\in \ver(\Phi(\xi_n))$ it holds that $(v,0)\in \ver(\Phi(\xi_n^0))$ since all points are extremal in the latter case and $\ver(\Phi(\xi_n^0))=\xi_n^0$.
Moreover, all cells of $\cL^*(\xi_n^0)$ and $\cL^*(\xi_n)$ are simplices almost surely. Hence, if for any $d+1$ distinct points $(v_1,h_1),\ldots,(v_{d+1},h_{d+1})\in\xi_n$, such that $\conv(v_1,\ldots,v_{d+1})$ forms a cell of the tessellation $\cL^*(\xi_n^0)$, it holds that $\conv(v_1,\ldots,v_{d+1})$ forms a cell of the tessellation $\cL^*(\xi_n)$ as well, the event $\cS_n$ occurs (almost surely). The event $\cS_n^c$, thus, implies that there exist $d+1$ distinct points $(v_1,h_1),\ldots,(v_{d+1},h_{d+1})\in\xi_n$ such that
$\conv(v_1,\ldots,v_{d+1})$ forms a cell of $\cL^*(\xi_n^0)$ but does not form a cell of $\cL^*(\xi_n)$, which by construction means that
\[
    \xi_n^0\cap\inter B(\bv)=\varnothing \text{ and } \xi_n \cap \inter (\Pi^-(\bv,\bh))^{\downarrow}\neq \varnothing,
\]
where we used the notation $\bv=(v_1,\ldots,v_{d+1})$, $(\bv,\bh)=((v_1,h_1),\ldots,(v_{d+1},h_{d+1}))$ and we recall that for $\bv\in {\rm AI}^{d+1}$ we write $B(\bv)$ for the unique ball containing $v_1,\ldots,v_{d+1}$ on its boundary and $\Pi^{-}(\bv,\bh)$ for the unique downward paraboloid containing $(v_1,h_1),\ldots,(v_{d+1},h_{d+1})$. Additionally, 
on $\cE_{n}$, and since all height coordinates of $\xi_n$ are non-negative, we have that the apex  $(z,s^2):=\apex(\Pi^-(\bv,\bh))$ satisfies $s\in [0,3r_n/2]$ and $\|z\|\leq \sqrt{9r^2_n/4-s^2}$. 
We denote the corresponding events in the following way
\begin{align*}
    \cB_\varnothing=\cB_{\varnothing}(\xi_n,\bv)&:=\{\xi_n^0\cap\inter B(\bv)=\varnothing\}=\{\xi_n\cap(\inter B(\bv)\times[0,9r_n^2/4])=\varnothing\},\\
    \cP_\varnothing=\cP_{\varnothing}(\xi_n,(\bv,\bh))&:=\{\xi_n \cap \inter (\Pi^-(\bv,\bh))^{\downarrow}=\varnothing\},\\
    \cA=\cA(\bv,\bh) &:= \{(z,s^2)=\apex(\Pi^-(\bv,\bh))\text{ satisfies }s\in [0,3r_n/2], \|z\|\leq \sqrt{9r^2_n/4-s^2}\}.
\end{align*}
Hence, we have
\[
    \PP(\cS_n^c \cap \cE_{n})\le \PP\big(\exists (v_1,h_1),\dots,(v_{d+1},h_{d+1})\in \xi_n: \cB_\varnothing\cap\cP_\varnothing^c\cap\cA\big).   
\]
Note that $\xi_n$ is by \cite[Theorem 5.2]{LP} a Poisson point process with intensity measure
\begin{equation}\label{eq:newInt}
\Lambda_{r_n}(\cdot)=\int_{B_{3r_n}}\int_{0}^{9r_n^2/4}\ind((v,h)\in\cdot)f_n(h)\dint h\dint v.
\end{equation}
Then using the inequality $\PP(X\neq 0)\leq \EE(X)$ for a non-negative integer-valued random variable $X$ and the multivariate Mecke formula \cite[Corollary 3.2.3]{SW} we get
\begin{equation}\label{eq:tessnotequal}
    \begin{aligned}
        \PP(\cS_n^c \cap \cE_{n})
        &\le \EE \sum_{(v_1,h_1),\dots,(v_{d+1},h_{d+1})\in (\xi_n)_{\neq}^{d+1}} \ind\big(\cB_\varnothing(\xi_n,\bv)\cap\cP_\varnothing^c(\xi_n,(\bv,\bh))\big)\ind(\cA(\bv,\bh))\\
        &\leq \int_{(B_{3r_n})^{d+1}} \int_{(\RR_+)^{d+1}} \PP_{\xi_n}\big(\cB_\varnothing\cap\cP_\varnothing^c\big)\ind(\cA)\prod_{i=1}^{d+1}f_n(h_i) \ind(h_i\le 9r_n^2/4) \dd h_i\,\dd v_i,
    \end{aligned}
\end{equation}
where $n\in \NN$ and $(\xi_n)_{\neq}^{d+1}$ is the process of $(d+1)$-tuples of distinct points of $\xi_n$, and $\cB_\varnothing\cap\cP_\varnothing^c$ is an event depending on fixed points $(\bv,\bh)$ and a point process $\xi_n$. 
Further, we note that by the definition of the events $\cB_{\varnothing}$ and $\cP_{\varnothing}$ (see Figure \ref{fig:parabolaSetmiusBall}) we have 
\begin{equation}\label{eq:03_01_26_eq1}
\begin{aligned}
\PP_{\xi_n}\big(\cB_\varnothing\cap\cP_\varnothing^c\big)
&=\PP\Big[\xi_n\Big(\inter (\Pi^-(\bv,\bh))^{\downarrow}\setminus \big(\inter B(\bv)\times [0,9r_n^2/4]\big)\Big)\neq 0\Big]\\
&\leq \EE\Big[\xi_n\Big((\Pi^-(\bv,\bh))^{\downarrow}\setminus \big(B(\bv)\times [0,9r_n^2/4]\big) \Big)\Big],
\end{aligned}
\end{equation}
\begin{figure}
    \centering
    \includegraphics[width=0.8\linewidth]{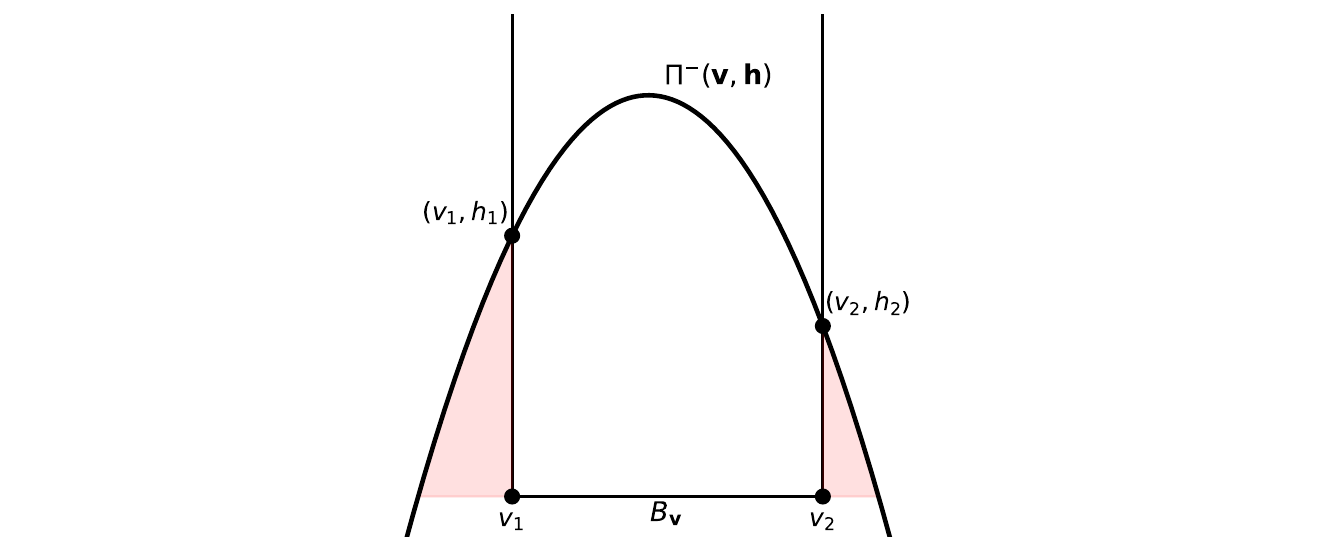}
    \caption{The downward parabola $(\Pi^-(\bv,\bh))^{\downarrow}$ and the stripe $\big(B(\bv)\times [0,9r_n^2/4]\big)$ in $\RR^2$. }
    \label{fig:parabolaSetmiusBall}
\end{figure}
where in the last step we again used the inequality $\PP(X\neq 0)\leq \EE(X)$ for a non-negative integer-valued random variable $X$ and the fact that the intensity measure of $\xi_n$ is absolutely continuous with respect to Lebesgue measure.
 
Recall that $(z,s^2)$ denotes the apex of the downward paraboloid $\Pi^-(\bv,\bh)$. Then we can write $v_i=z+sy_i$ with uniquely defined and pairwise distinct $y_1,\dots, y_{d+1}\in \BB^{d}$ and also $h_i=s^2(1-\|y_i\|^2)$. Consider the transformation
\begin{equation}\label{eq:transformation}
    \begin{aligned}
        \varphi\colon  &\RR^{d} \times \RR_+ \times (\BB^{d})^{d+1} &&\to &&(\RR^{d} \times \RR_+)^{d+1}\\
        &(z,s,y_1,\dots,y_{d+1}) &&\mapsto &&(z+s y_1,s^2(1-\|y_1\|^2),\dots,z+sy_{d+1},s^2(1-\|y_{d+1}\|^2)),
    \end{aligned}
\end{equation}
with Jacobian 
\begin{align}\label{eq:Jacobian}
|\det J(\varphi)| =2^{d+1}d!s^{(d+1)^2}\lambda_d\big(\conv(y_1,\ldots,y_{d+1})\big)\leq 2^{d+1}d!\kappa_d s^{(d+1)^2}    
\end{align}
(see \cite[Theorem 5.1]{GKT21}). Note that the event $\cA$ implies $z\in B_{3r_n/2}$ and $s\le 3r_n/2$. Applying the transformation $\varphi$ to \eqref{eq:tessnotequal} together with \eqref{eq:03_01_26_eq1} then yields
\begin{equation}\label{eq:p_no_b}
    \begin{aligned}
    \PP(\cS_n^c \cap \cE_{n}) & \le C\int_{(\BB^d)^{d+1}} \int_{0}^{3r_n/2} \int_{B_{3r_n/2}} s^{(d+1)^2}\EE\Big[\xi_n\Big((\Pi^-_{(z,s^2)})^{\downarrow} \setminus \big((s B(\by)+z)\times [0,9r_n^2/4]\big)\Big) \Big]\\
    &\hspace{1.5cm}  \times \prod_{i=1}^{d+1}f_n\big(s^{2} (1-\|y_i\|^2)\big)\ind\big(sy_i+z\in B_{3r_n}\big) \dd z\,\dd s \,\dd y_1,\dots, \dd y_{d+1}.
\end{aligned} 
\end{equation}

Let $s\in(0,3r_n/2]$, $z\in B_{3r_n/2}$ and $\by\in (\BB^d)^{d+1}\cap {\rm AI}^{d+1}$. Recall, that  $\rho(\by)$ denotes the radius and $c(\by)$ the center of the ball $B(\by)$. Then by \eqref{eq:newInt} we have
\begin{align*}
     E(n,s,\by,z)&:= \EE\Big[\xi_n\Big((\Pi^-_{(z,s^2)})^{\downarrow} \setminus ((sB(\by)+z)\times [0,9r_n^2/4])\Big)\Big]\notag\\
    &= \int_{B_{3r_n}} \int_0^{9r_n^2/4} f_n(h) \ind\Big((v,h)\in (\Pi^-_{(z,s^2)})^{\downarrow} \setminus \big((sB(\by)+z)\times [0,9r_n^2/4]\big)\Big) \dd h\, \dd v.
\end{align*}
Let us now analyze the indicator function. Note that $(v,h)\in (\Pi^-_{(z,s^2)})^\downarrow$ means $h\leq -\|v-z\|^2+s^2$, which in turn implies $\|v-z\|\leq \sqrt{s^2-h}\leq s$ and $h\le s^2$. At the same time the condition $(v,h)\not\in (s B(\by) + z) \times [0,9r_n^2/4]$ is equivalent to $\|v-sc(\by)-z\|> s\rho(\by)$ for any $h\in[0,9r_n^2/4]$. Hence,
\begin{align*}
    \ind\Big((v,h)&\in (\Pi^-_{(z,s^2)})^{\downarrow} \setminus \big((sB(\by)+z)\times [0,9r_n^2/4]\big)\Big)\\
    &\leq \ind\big(\|v-z\|\leq s,\|v-sc(\by)-z\|>s\rho(\by), h\leq s^2).
\end{align*}
Then applying the change of variables $v'=v-z$ and noting that $z\in B_{3r_n/2}$ and $v\in B_{3r_n}$ implies $v'\subset B_{5r_n}$, we get
\begin{align}
    E(n,s,\by,z)&\le \int_{B_{5r_n}} \int_0^{s^2} f_n(h) \ind\big(\|v'\|\leq s,\|v'-sc(\by)\|> s\rho(\by)\big)\dd h\, \dd v'\notag\\
    &\le \int_{B_{5r_n}} \int_0^{s^2} f_n(h) \ind\big(s(\rho(\by)-\|c(\by)\|)<\|v'\|\leq s\big)\dd h\, \dd v',\notag
\end{align} 
where in the second step we also used the triangle inequality.
Since 
\[
    \big|\rho(\by)-\|c(\by)\|\big|=\big|\|c(\by)-y'\|-\|c(\by)\|\big|\le \|y'\|\leq 1
\] 
for some $y'\in \BB^d$ the condition $s(\rho(\by)-\|c(\by)\|)\le s$ holds for any $\by\in (\BB^d)^{d+1}\cap{\rm AI}^{d+1}$. Thus,
\begin{align}
    E(n,s,\by,z)&\le Cs^d\big(\frI^1f_n\big)(s^2)\big(1-(\rho(\by)-\|c(\by)\|)_+^d\big),\label{eq:bound_p_no_b}
\end{align}
where $x_+=\max(x,0)$.
Further for $s\in [0,3r_n^2/2]$ and $\by\in (\BB^d)^{d+1}\cap{\rm AI}^{d+1}$ define
\begin{align}
    h(\by)&:=\big(1-(\rho(\by)-\|c(\by)\|)_+^d\big).\label{eq:06_01_26_eq1}
\end{align}
Recall the definition \eqref{eq:probDensity} of the probability density $g_{n,s}$ and let $Y^{(n,s)}_1,\ldots,Y^{(n,s)}_{d+1}$ be independent and identically distributed random points in $\BB^d$ with density $g_{n,s}$ for any $n\in \NN$ and $s\ge 0$. Then applying the bound in \eqref{eq:bound_p_no_b} to \eqref{eq:p_no_b}, estimating $\ind(sy_i+z\in B_{3r_n})\leq 1$ and integrating with respect to $z$, and finally using Fubini's theorem yields
\begin{align}
    \PP(\cS_n^c \cap \cE_{n}) 
    &\le Cr_n^d \int_{0}^{3r_n/2} s^{(d+1)^2+d} \big(\frI^1f_n\big)(s^2) \int_{(\BB^d)^{d+1}} h(s,\by)\prod_{i=1}^{d+1}f_n(s^{2}-s^{2}\|y_i\|^2))\dint y_i \dd s\notag\\
    &= Cr_n^d\, \int_{0}^{3r_n/2} s^{2d+1}\Big(\big(\frI^{\frac{d}{2}} f_n\big)(s^2)\Big)^{d+1} \big(\frI^1f_n\big)(s^2)\EE\big[h(Y^{(n,s)}_1,\ldots,Y^{(n,s)}_{d+1})\big] \dd s.\label{eq:04_01_25_eq2}
\end{align}
Further denote by 
\begin{equation}\label{eq:05_01_26_eq9}
F^{(1)}_n(s):=s^{2d+1}\Big(\big(\frI^{\frac{d}{2}} f_n\big)(s^2)\Big)^{d+1} \big(\frI^1f_n\big)(s^2)\EE\big[h(Y^{(n,s)}_1,\ldots,Y^{(n,s)}_{d+1})\big]
    \end{equation}
and 
combining \eqref{eq:Intermediate} with \eqref{eq:eventBounds}, \eqref{eq:02_01_26_eq5} and \eqref{eq:04_01_25_eq2} for any $n\ge \max (n_0,n_1)$ we finally obtain
\begin{align}
    \PP\big(\sk{L}^*(\widetilde{\eta}_n)&\cap B_R\neq \sk{L}^*(\widetilde{\eta}_0)\cap B_R\big)\notag\\
    &\leq C\Big(\exp\big(-cr_n^d\big)+r_n^d\Big(\int_{0}^{3r_n/2}F^{(1)}_n(s)\dd s+\left\lvert \gamma -\big(\frI^1 f_n\big)(9r_n^2/4)\right\rvert\Big)\Big).\label{eq:05_01_26_eq8}
\end{align}

\medskip
\textbf{Estimate for the Laguerre tessellation:} As in the previous part let $n_1(R)>0$ be such that $r_n\ge 2\max(R,1)$ for any $n>n_1$ and by the law of total probability in combination with \eqref{eq:04_01_26_eq4} and the fact that $K_2(R,r_n)\subset K_2(r_n/2,r_n)$, $\widetilde\eta_n\cap K_2(r_n/2,r_n)=\xi_n$ and $\xi_n^0=\widetilde\eta_0\cap B_{3r_n}$ on $\cY_n$ almost surely we get
\begin{align}
    \PP\Big(\big[&(\sk{L}(\widetilde{\eta}_n)\setminus (\sk{L}(\widetilde{\eta}_0))^{\varepsilon-})\cup (\sk{L}(\widetilde{\eta}_0)\setminus (\sk{L}(\widetilde{\eta}_n))^{\varepsilon-})\big]\cap B_R\neq \varnothing\Big)\notag\\
    &= \PP\Big(\Big\{\big[(\sk{L}(\widetilde{\eta}_n)\setminus (\sk{L}(\widetilde{\eta}_0))^{\varepsilon-})\cup (\sk{L}(\widetilde{\eta}_0)\setminus (\sk{L}(\widetilde{\eta}_n))^{\varepsilon-})\big]\cap B_R\neq \varnothing\Big\}\cap\cH_0^{\rm max}\cap\cH_n^{\rm max}\cap \cY_n\Big)\notag\\
    &\hspace{10cm}+\PP\big((\cH_0^{\rm max}\cap\cH_n^{\rm max}\cap \cY_n)^c\big)\notag\\
    &\leq \PP\Big((\sk{L}(\xi_n)\setminus (\sk{L}(\xi_n^0))^{\varepsilon-})\cup (\sk{L}(\xi_n^0)\setminus (\sk{L}(\xi_n))^{\varepsilon-})\neq \varnothing\Big)\notag\\
    &\hspace{7cm}+\PP\big((\cH_0^{\rm max})^c\big)+\PP\big((\cH_n^{\rm max})^c\big)+\PP\big(\cY_n^c\big).\label{eq:05_01_26_eq1}
\end{align}
Again using the law of total probability and recalling the definition of the event $\cS_n$ (see \eqref{eq:Sn}) and $\cE_n$ we have
\begin{equation}\label{eq:05_01_26_eq2}
\begin{aligned}
    \PP\Big((\sk{L}(\xi_n)&\setminus (\sk{L}(\xi_n^0))^{\varepsilon-})\cup (\sk{L}(\xi_n^0)\setminus (\sk{L}(\xi_n))^{\varepsilon-})\neq \varnothing\Big)\\
    &\le \PP\Big(\Big\{(\sk{L}(\xi_n)\setminus (\sk{L}(\xi_n^0))^{\varepsilon-})\cup (\sk{L}(\xi_n^0)\setminus (\sk{L}(\xi_n))^{\varepsilon-})\neq \varnothing\Big\}\cap \cS_n\cap \cE_n\Big)\\
    &\hspace{10cm}+\PP\big((\cS_n\cap \cE_n)^c\big)\\
    &\le \PP\Big(\big\{\sk{L}(\xi_n)\not \subset (\sk{L}(\xi_n^0))^{\varepsilon-}\big\}\cap \cS_n\cap \cE_n\Big)+\PP\Big(\big\{\sk{L}(\xi_n^0)\not \subset (\sk{L}(\xi_n))^{\varepsilon-}\big\}\cap \cS_n\cap \cE_n\Big)\\
    &\hspace{9.5cm}+ \PP(\cS_n^c\cap \cE_n)+\PP(\cE_n^c).
\end{aligned}
\end{equation}
 Note that
 \begin{align*}
     \sk{L}(\xi_n)&=\bigcup_{(v,h)\in\xi_n}{\rm bd}\,C((v,h),\xi_n),\\
     \sk{L}(\xi_n^0)&=\bigcup_{v\in \xi_n^0}{\rm bd}\,C(v,\xi_n^0).
\end{align*}
Further, on $\cS_n$ we have that $C((v,h),\xi_n)\cap C((v',h'),\xi_n)\neq \varnothing$ if and only if $C(v,\xi_n^0)\cap C(v',\xi_n^0)\neq \varnothing$, implying; in particular, that for any $(v,h)\in \xi_n$ the cells $C((v,h),\xi_n)$ and $C(v,\xi_n^0)$ have the same combinatorial structure (i.e. face lattice, see \cite[Definition 2.6.]{z-lop-95}). This follows by the duality relation and by the fact that $\cL^*(\xi_n)$ and $\cL^*(\xi_n^0)$ coincide on $\cS_n$. Moreover, by \eqref{eq:VertExtermal} we have 
\[
\ver(\Phi(\xi_n^0))=\ext(\Psi(\xi_n^0))=\xi_n^0,
\]
and, hence, $C(v,\xi_n^0)\neq \varnothing$ for any $v\in\xi_n^0$ implying that $C((v,h),\xi_n)\neq \varnothing$ for any $(v,h)\in\xi_n$.

 We will first consider the event $\big\{\sk{L}(\xi_n)\not \subset (\sk{L}(\xi_n^0))^{\varepsilon-}\big\}\cap \cS_n\cap\cE_n$, which means that there exists $(v,h)\in \xi_n$ and a point $x\in \bd C((v,h),\xi_n)$ such that
 \[
    \inf_{y\in \bd C(v,\xi^0)} \|x-y\| \ge \inf_{y\in \sk{L}(\xi_n^0)} \|x-y\| \ge \varepsilon.
 \] 
Assume that $C((v,h),\xi_n)$ has vertices $w_1,\ldots,w_m$, where $m\in \NN$. Then on $\cS_n$ the cell $C(v,\xi^0_n)$ has exactly $m$ vertices $w_1',\ldots,w_m'$ and, furthermore, for each $i=1,\ldots,m$ we have that almost surely there exist exactly $d+1$ points $(v_1,h_1),\ldots,(v_{d+1},h_{d+1})\in \xi_n$ such that 
\[
    w_i=\proj_{\RR^d} \big(\apex (\Pi^-(\bv,\bh))\big),\qquad  w_i'=c(\bv),
\]
where we recall the notation $\bv=(v_1,\ldots,v_{d+1})$, $(\bv,\bh)=((v_1,h_1),\ldots,(v_{d+1},h_{d+1}))$ and that for $\bv\in {\rm AI}^{d+1}$ we denote by $c(\bv)$ the center of the unique ball $B(\bv)$ containing $v_1,\ldots,v_{d+1}$ on its boundary and by $\Pi^{-}(\bv,\bh)$ the unique downward paraboloid containing $(v_1,h_1),\ldots,(v_{d+1},h_{d+1})$. Since $x\in \bd C((v,h),\xi_n)$ there exists a facet $F$ of $C((v,h),\xi_n)$ with vertices $w_{i_1},\ldots,w_{i_k}$, $k<m$, such that $x\in F$. Note that then there is a facet $F'$ of $C(v,\xi^0)$ with vertices $w_{i_1}',\ldots,w_{i_k}'$. 

In general, $F$ and $F'$ are $(d-1)$-dimensional convex polyhedral sets. Due to \eqref{eq:halfspaces} and $\cS_n$ the faces $F$ and $F'$ are contained in parallel hyperplanes, have the same face lattice and their corresponding faces are also parallel. 
Hence by the Minkowski-Weyl theorem for convex polyhedral sets the faces $F,F'$ can be represented as
\[
    F=\conv(w_{i_1},\ldots,w_{i_k})+{\rm pos}(r_1,\ldots, r_q),\qquad F'=\conv(w'_{i_1},\ldots,w'_{i_k})+{\rm pos}(r_1,\ldots, r_q),
\]
where $r_1,\ldots,r_q$ are some unit vectors and ${\rm pos}(r_1,\ldots,r_q)$
is a cone spanned by them. Hence, there exist $\lambda_{i_1},\ldots, \lambda_{i_k}\in [0,1]$ with $\lambda_{i_1}+\ldots + \lambda_{i_k}=1$ and $\mu_1,\ldots,\mu_q\ge 0$, such that $x = \sum_{j=1}^k \lambda_{i_j} w_{i_j}+\sum_{i=1}^q\mu_ir_i$. Moreover, $x':=\sum_{j=1}^k \lambda_{i_j} w'_{i_j}+\sum_{i=1}^q\mu_ir_i\in F'$ and we get
\begin{align*}
    \inf_{y\in \bd C(v,\xi^0)} \|x-y\| &\le \Big\|\sum_{j=1}^k \lambda_{i_j} w_{i_j}- \sum_{j=1}^k \lambda_{i_j} w_{i_j}'\Big\|\leq \max_{1\leq j\leq k}\|w_{i_j}-w'_{i_j}\|\\
    &\leq \max_{(v_1,h_1),\ldots, (v_{d+1},h_{d+1})\in \xi_n} \big\|\proj_{\RR^d}({\rm apex}(\Pi^-(\bv,\bh)))-c(\bv)\big\|.
\end{align*}

 Given $(\bv,\bh)=((v_1,h_1),\ldots,(v_{d+1},h_{d+1}))$, $\bv\in {\rm AI}^{d+1}$ define the events
\begin{align*}
 \cV_{\varepsilon}&:=\cV_{\varepsilon}(\bv,\bh)=\big\{\big\|\proj_{\RR^d}\big({\rm apex}(\Pi^-(\bv,\bh))\big)-c(\bv)\big\|>\varepsilon\big\},\\
    \cA=\cA(\bv,\bh) &:= \{(z,s^2)=\apex(\Pi^-(\bv,\bh))\text{ satisfies }s\in [0,3r_n/2], \|z\|\leq \sqrt{9r^2_n/4-s^2}\}.
\end{align*}
Recall that on $\cE_n$ the event $\cA$ occurs. Then
\begin{equation}\label{eq:05_01_26_eq7}
    \PP\big(\big\{\sk{L}(\xi_n)\not \subset (\sk{L}(\xi_n^0))^{\varepsilon-}\big\}\cap \cS_n\cap\cE_n\big)\leq \PP(\exists (v_1,h_1),\dots,(v_{d+1},h_{d+1})\in \xi_n: \cA\cap \cV_{\varepsilon})=:P(n).
\end{equation}
 Using the inequality $\PP(X\neq 0)\leq \EE(X)$ for a non-negative integer-valued random variable $X$ and the multivariate Mecke formula, and recalling the definition of the intensity measure of $\xi_n$ (see \eqref{eq:newInt}) yields
\begin{align*}
     P(n)&\leq \EE\sum_{(v_1,h_1),\dots,(v_{d+1},h_{d+1})\in (\xi_n)_{\neq}^{d+1}} \ind(\cV_{\varepsilon})\ind(\cA(\bv,\bh))\\
     &= \int_{(B_{3r_n})^{d+1}} \int_{(\RR_+)^{d+1}} \ind(\cV_{\varepsilon})\ind(\cA) \prod_{i=1}^{d+1} f_n(h_i) \ind(h_i\le 9r_n^2/4) \dd h_i \dd v_i,
 \end{align*}
 where $(\xi_n)_{\neq}^{d+1}$ is the process of $(d+1)$-tuples of distinct points of $\xi_n$. As in the proof of (i) denote by $(z,s^2)$ the apex of the downward paraboloid $\Pi^-(\bv,\bh)$ and consider the transformation $\varphi$ defined by \eqref{eq:transformation} with Jacobian given in \eqref{eq:Jacobian}. Recall also that on the event $\cA$ we have $z\in B_{3r_n/2}$ and $s\leq 3r_n/2$. Applying $\varphi$ to the integral above we get
 \begin{align*}
     P(n)&\leq C\int_{(\BB^d)^{d+1}} \int_0^{3r_n/2} \int_{B_{3r_n/2}} s^{(d+1)^2}  \ind(\|z-(s\,c(\by)+z)\|>\varepsilon)\\
     &\hspace{4cm}\prod_{i=1}^{d+1} f_n(s^2(1-\|y_i\|^2)\ind(sy_i+z \in B_{3r_n}) \dd z \dd s \dd y_1,\ldots \dd y_{d+1}.
 \end{align*}
 Using Fubini's Theorem, recalling the definition of the probability density $g_{n,s}$ in \eqref{eq:probDensity} and bounding $\ind(sy_i+z \in B_{3r_n})\le 1$ we have
\begin{align*}
    P(n)&\leq Cr_n^d \int_0^{3r_n/2} \int_{(\BB^d)^{d+1}} s^{d+1} \Big(\big(\frI^{\frac{d}{2}}f_n\big)(s^2)\Big)^{d+1}\ind(\|c(\by)\|>\varepsilon/s) \prod_{i=1}^{d+1} g_{n,s}(y_i) \dd y_1,\ldots \dd y_{d+1} \dd s\\
    &\leq C r_n^d \int_0^{3r_n/2} s^{d+1} \Big(\big(\frI^{\frac{d}{2}}f_n\big)(s^2)\Big)^{d+1} \PP(\|c(Y_1^{n,s},\ldots Y_{d+1}^{n,s})\|>\varepsilon/s) \dd s,
 \end{align*}
 where $Y^{(n,s)}_1,\ldots,Y^{(n,s)}_{d+1}$ are independent and identically distributed random points in $\BB^d$ with density $g_{n,s}$ for any $n\in \NN$ and $s>0$. Denoting by
 \begin{equation}\label{eq:05_02_26_eq10}
 F_n^{(2)}(s):=s^{d+1} \Big(\big(\frI^{\frac{d}{2}}f_n\big)(s^2)\Big)^{d+1} \PP(\|c(Y_1^{n,s},\ldots Y_{d+1}^{n,s})\|>\varepsilon/s), \quad s>0
 \end{equation}
 we obtain by \eqref{eq:05_01_26_eq7} that
 \begin{align*}
     \PP\big(\big\{\sk{L}(\xi_n)\not \subset (\sk{L}(\xi_n^0))^{\varepsilon-}\big\}\cap \cS_n\cap\cE_n\big)&\leq Cr_n^d\int_{0}^{3r_n/2}F_n^{(2)}(s)\dd s.
 \end{align*}
 The same argument applies to the event $\big\{\sk{L}(\xi_n^0)\not \subset (\sk{L}(\xi_n))^{\varepsilon-}\big\}\cap \cS_n\cap \cE_n$ leading to the same bound. Hence, combining this with \eqref{eq:05_01_26_eq1}, \eqref{eq:05_01_26_eq2}, \eqref{eq:04_02_26_eq6} and \eqref{eq:05_01_26_eq9} we conclude
 \begin{align}
     \PP\Big(\big[&(\sk{L}(\widetilde{\eta}_n)\setminus (\sk{L}(\widetilde{\eta}_0))^{\varepsilon-})\cup (\sk{L}(\widetilde{\eta}_0)\setminus (\sk{L}(\widetilde{\eta}_n))^{\varepsilon-})\big]\cap B_R\neq \varnothing\Big)\notag\\
     &\leq C\Big(\exp\big(-cr_n^d\big)+r_n^d\Big(\int_{0}^{3r_n/2}(F^{(1)}_n(s)+F^{(2)}_n(s))\dd s+\left\lvert \gamma -\big(\frI^1 f_n\big)(9r_n^2/4)\right\rvert\Big)\Big).\label{eq:05_01_26_eq11}
 \end{align}

\medskip

\textbf{Choosing the sequence $(r_n)_{n\in\NN}$:} To finish the proof it is sufficient to show that the right hand side of inequality \eqref{eq:05_01_26_eq11} goes to $0$ as $n\to\infty$. Since $F_n^{(2)}(s)\ge 0$ for any $s\ge 0$ this also implies that the right hand side of \eqref{eq:05_01_26_eq8} converges to $0$ as $n\to \infty$.

Recall the definitions of $F_n^{(1)}$ (see \eqref{eq:05_01_26_eq9}) and $F_n^{(2)}$ (see \eqref{eq:05_02_26_eq10}). By Lemma \ref{lm:conditionsImply} we have 
for any $s>0$ that $(Y^{(n,s)}_1,\ldots, Y^{(n,s)}_{d+1})$ converges in distribution to $(Z_1,\ldots,Z_{d+1})$, as $n \to \infty$, where $Z_1,\ldots,Z_{d+1}$ are independent and uniformly distributed on $\SS^{d-1}$. Note that the maps $\by\to h(\by)$ (see \eqref{eq:06_01_26_eq1}) and $\by \to \|c(\by)\|$ are bounded and by Lemma \ref{lm:continuity} they are also continuous at any $(x_1,\ldots,x_{d+1})\in (\SS^{d-1})^{d+1}\cap \rm{AI}^{d+1}$. Then using the Portmanteau theorem , and noting $\rho(Z_1,\ldots,Z_{d+1})=1$ and $\|c(Z_1,\ldots,Z_{d+1})\|=0$ almost surely, yields for any $s> 0$ and $\varepsilon>0$ that
\begin{align*}
    \lim_{n\to\infty}\EE \big(h(Y^{(n,s)}_1,\ldots,Y^{(n,s)}_{d+1})\big) &=\EE \big(h(s,Z_1,\ldots,Z_{d+1})\big)= 0,\\
    \lim_{n\to \infty}\PP(\|c(Y_1^{n,s},\ldots Y_{d+1}^{n,s})\|>\varepsilon/s) &= \PP(\|c(Z_1,\ldots Z_{d+1})\|>\varepsilon/s) = 0.
\end{align*}
This further implies that for any fixed $x>0$ we have
\begin{equation}\label{eq:04_01_26_eq3}
\lim_{n\to\infty}\int_{0}^x(F^{(1)}_n(s)+F_n^{(2)}(s))\dd s=0.
\end{equation}
Indeed, by the definition of $F_n^{(1)}$ and $F_n^{(2)}$ we have that $\lim_{n\to\infty}(F^{(1)}_n(s)+F_n^{(2)}(s))=0$. Moreover, due to \hyperref[item:C2]{(C2)} we have that for any $s\in [0,x]$,  $\alpha\ge 1$ and $n\ge n_2(x)$ that
\begin{align*}
    0\leq \Gamma(\alpha)\big(\frI^{\alpha+1} f_n\big)(s^2)\leq \int_{0}^{x^2}f_n(h)(x^2-h)^{\alpha}\dint h\leq \big(\frI^1f_n\big)(x^2)x^{2\alpha}\leq (\gamma+1)x^{2\alpha},
\end{align*}
meaning that $F^{(1)}_n(s)+F_n^{(2)}(s)$ is bounded for $s\leq x$ and by the dominated convergence theorem we conclude \eqref{eq:04_01_26_eq3}. 

Now by \eqref{eq:04_01_26_eq3} and \hyperref[item:C2]{(C2)} there is a (strictly) increasing sequence $(N_k)_{k\in\NN}$ of natural numbers such that for any $n\ge N_k$ we have
\[
    \int_{0}^{k}(F^{(1)}_n(s)+F_n^{(2)}(s))\dd s+\left\lvert \gamma -\big(\frI^1 f_n\big)(k^2)\right\rvert\leq 2^{-k}k^{-d}.
\]
Defining $r_n=2k/3$ for $N_k\leq n<N_{k+1}$ and $r_n=1/2$ for $1\leq n<N_k$ we obtain
\[
    r_n^d\Big(\int_{0}^{3r_n/2}(F^{(1)}_n(s)+F_n^{(2)}(s))\dd s+\left\lvert \gamma -\big(\frI^1 f_n\big)(9r_n^2/4)\right\rvert\Big)\leq 3^{-d}\cdot 2^{-k+d}\to 0,
\]
as $n\to\infty$. Further since $r_n\to\infty$ as $n\to\infty$ we get $\exp(-cr_n^d)\to 0$ as $n\to\infty$. Hence, with this choice of $(r_n)_{n\in\NN}$ the inequality \eqref{eq:05_01_26_eq11} implies \eqref{eq:04_02_26_eq6} and the bound \eqref{eq:05_01_26_eq8} implies \eqref{eq:02_01_26_eq6}. This finishes the proof.

\end{proof}

\section{Proofs of technical lemmas}\label{sec:proofs}

\paragraph*{Proof of Lemma \ref{lm:BoundaryBounds}:}

    The proof of 1. follows the same arguments as in \cite[Lemma 4.4]{GKT21} and \cite[Lemma 3]{GKT21b}. For the reader's convenience, we provide the proof here as well. 
    
    \textbf{Proof of 1. (a):} We start by noting that for any $a>0$, $T\in \RR$ we have
    \begin{align}
        \PP\big(\cH^{\rm{max}}(\eta_f, a, T)^c\big)&=\PP\big(\sup_{w\in B_a}\inf_{(v,h)\in \eta_f}\pow(w,(v,h))> T\big)\notag\\
        &\leq \PP\Big(\inf_{(v,h)\in\eta_f}\;\sup_{w\in B_{a}} \pow(w,(v,h)) > T \Big).\label{eq:17_12_25_1}
    \end{align}
    Let $T> 4a^2$ and note that for points $(v,h)\in \eta_f$ satisfying $h\le T-a^2$ and $v\in B_{\sqrt{T-h}-a}$ we have
    \[
        \sup_{w\in B_a} \pow(w,(v,h))=\sup_{w\in B_a} \|v-w\|^2+h\le \sup_{w\in B_a} (\|v\|+\|w\|)^2+h \le (\sqrt{T-h}-a+a)^2+h = T.
    \]
    Thus, defining
    \[
        K_3(a,T):=\big\{(v,h)\in \RR^d\times \RR:h \le T-a^2, v\in B_{\sqrt{T-h}-a}\big\}
    \]
    we obtain an estimate
    \begin{align}
        \PP\Big(\inf_{(v,h)\in\eta_f}\;\sup_{w\in B_a} \pow(w,(v,h)) > T \Big)\leq \PP\Big(\eta_f\cap K_3(a,T)=\varnothing\Big)
        =\exp\big(-\EE\big[\eta_f\big(K_3(a,T)\big)\big]\big).\label{eq:17_12_25_2}
    \end{align}
    Further using Campbell's theorem \cite[Proposition 2.7]{LP}, we get
    \begin{align*}
        \EE\big[\eta_f\big(K_3(a,T)\big)\big]&= \int_{-\infty}^{T-a^2} \int_{B_{\sqrt{T-h}-a}} f(h) \dd v \dd h
        \ge \kappa_d \int_{-\infty}^{T-4a^2} f(h) (\sqrt{T-h}-a)^{d} \dd h.
    \end{align*}
    For $h\le T-4a^2$ we have $a\le \frac{1}{2}\sqrt{T-h}$ and, hence,
    \begin{align}
        \EE\big[\eta_f\big(K_3(a,T)\big)\big]&\ge 2^{-d} \kappa_d \int_{-\infty}^{T-4a^2} f(h) (T-h)^{\frac{d}{2}} \dd h
        \ge 2^{-d}\pi^{\frac{d}{2}}(\frI^{\frac{d}{2}+1}f)(T-4a^2),\label{eq:17_12_25_3}
    \end{align}
    which together with \eqref{eq:17_12_25_1} and \eqref{eq:17_12_25_2} concludes the proof of the first bound. For the second bound we use the semigroup property \eqref{eq:SemigroupProp} and for any $x\ge x_0$ we get
    \[
        (\frI^{\frac{d}{2}+1}f_n)(x)=\big(\frI^{d\over 2}(\frI^1f_{n})\big)(x)\ge { (\frI^1f)(x_0)\over \Gamma(d/2)}\int_{x_0}^x(x-h)^{{d\over 2}-1}\dint h={(\frI^1f)(x_0)\over \Gamma(d/2+1)}(x-x_0)^{d\over 2}.
    \]
    Then for any $T\ge 4a^2+2x_0$ by \eqref{eq:17_12_25_3} we conclude 
    \[
        \EE\big[\eta_n\big(K_3(a,T)\big)\big]\ge 2^{-d}\pi^{\frac{d}{2}}{(\frI^1f)(x_0)\over \Gamma(d/2+1)}\big(T-4a^2-x_0\big)^{d\over 2}\ge \big(\frac{\pi}{8}\big)^{\frac{d}{2}} {(\frI^1f)(x_0)\over \Gamma(d/2+1)} (T-4a^2)^{\frac{d}{2}},
    \]
    and the proof follows by \eqref{eq:17_12_25_1} and \eqref{eq:17_12_25_2}.
    
    \textbf{Proof of 1. (b):} 
    Let $a>0$, $t\in \RR$ and define
    \[
        K_4(a,t):=\{(v,h)\in \RR^d\times (-\infty,t):\|v\|<a+\sqrt{t-h}\}.
    \]
    Note that $\pow(w,(v,h))< t$ for some $w\in B_a$ and $(v,h) \in \eta_f$ if and only if $\eta_f\cap K_4(a,t)\neq \varnothing$. Indeed, assume there exists $w\in B_a$ and $(v,h)\in \eta_f$, such that $\pow(w,(v,h))=\|v-w\|^2+h< t$. Then we have
    \[
        \|v\|\le \|v-w\|+\|w\|< \sqrt{t-h}+a,
    \]
    and $h<t$ and hence $(v,h)\in K_4(a,t)$. If otherwise there exists $(v,h)\in \eta_f$, such that $\|v\|<a+\sqrt{t-h}$ and $h< t$ we define $w:=a\frac{v}{\|v\|}\in B_a$ and note 
    \[
        \pow(w,(v,h))=\big\|v-a\frac{v}{\|v\|}\big\|^2+h < t-h+h=t.
    \]
    Hence,
    \begin{align*}
        \PP\big(\cH^{\rm{min}}(\eta_f, a, t)^c\big)
        =1-\PP(\eta_f\cap K_4(a,t)=\varnothing)=1-\exp\big(-\EE\big[\eta_f\big(K_4(a,t)\big)\big]\big).
    \end{align*}
    Using Jensen's inequality, the estimate
    \begin{align*}
        \EE\big[\eta_f\big(K_4(a,t)\big)\big]&=\int_{-\infty}^t \int_{B_{a+\sqrt{t-h}}}f(h)\dd v \dd h \\
        &= \kappa_d \int_{-\infty}^{t} (a+\sqrt{t-h})^{d} f(h) \dd h\\
        &\leq 2^{d-1}\kappa_d \Big(\int_{-\infty}^{t-a}(t-h)^{d\over 2}f(h)\dd h+a^{d}\int_{t-a}^{t} f(h) \dd h\Big)\\
        &\leq 2^{d-1}\kappa_d \big((\frI^{\frac{d}{2}+1}f)(t)+a^{d}(\frI^{1}f)(t)\big),
    \end{align*}
    together with the bound $1-e^{-x}\le x$ finishes the proof.

    \medskip
    
    \textbf{Proof of 2.:} We note that by the same arguments as before we have for any $a>0$ and $T\ge a^2$,
    \begin{align*}
        \PP\Big(\cH^{\rm{max}}(\eta^\gamma, a,T)\Big)&\leq \exp\big(-\EE\big[\eta^\gamma\big(K_3(a,T)\big)\big]\big)=\exp\big(-\gamma \kappa_d (\sqrt{T}-a)^d\big),
    \end{align*}
    where we recall that $\eta^{\gamma}$ is viewed as a point process in $\RR^d\times \RR$ by identifying every point $v\in  \eta^{\gamma}$ with the point $(v,0)\in \RR^d\times \RR$. Further, in this case we have that $\pow(w,(v,0))=\|v-w\|^2>0$ for any $(v,0)\in \eta^\gamma$ and hence for any $a>0$ and $t<0$ we obtain $\PP(\cH^{\rm{min}}(\eta^{\gamma}, a, t)^c)= 0$.

\paragraph*{Proof of Lemma \ref{lm:StabRadius}:}
    In what follows, let us first write $\eta$ for one of the Poisson point processes $\eta^\gamma$ or $\eta_{f}$, where $f$ is an admissible function. Throughout the proof $C$, $\tilde C$, $c$, $\tilde c$ will denote some positive constants, which only depend on $d$.  Their exact values might be different from line to line. 
    
    Let $R>0$ and $r\ge 2\max(1,R)$. By the law of total probability, for any $t\le 0$ we have
    \begin{align}\label{eq:totalprob}
        \PP(\cE(\eta,R,r)^c)\le P_1(\eta, R,r,t) + P_2(\eta,R,t),
    \end{align}
    where
    \begin{align*}
        &P_1(\eta,R,r,t):=\PP\big(\cE(\eta,R,r)^c\cap\cH^{\rm{min}}(\eta,R,t)\big),\\
        &P_2(\eta,R,t):=\PP\big(\cH^{\rm{min}}(\eta,R,t)^c\big).
    \end{align*}

    In order to bound $P_1(\eta, R,r,t)$, we first note that on $\cH^{\rm{min}}(\eta,R,t)$ we have $\bd \Psi(\eta)\cap (B_R\times \RR)\subset B_R\times [t,\infty)$, namely, the paraboloid growth process restricted to the the ball $B_R$ lies above $t$. This also implies that $\bd \Phi(\eta)\cap (B_R\times \RR)\subset B_R\times [t,\infty)$. Since for any $S\in\cL^*(\eta)$ the set $F(S):=\Pi^-(S)\cap\bd\Phi(\eta)$ is a paraboloid facet of $\Phi(\eta)$ by \eqref{eq:ParaboloidFacetSimplex} we conclude that if $S\cap B_R\neq \varnothing$ then $\varnothing\neq F(S)\cap (B_R\times \RR)\subset B_R\times [t,\infty)$ and, hence,
        \begin{align}\label{eq:apexCylinder}
        \Pi^{-}(S) \cap (B_R\times [t,\infty)) \neq \varnothing.
    \end{align}
    Let $(w,q)$ be the apex of $\Pi^-(S)$, then \eqref{eq:apexCylinder} means that there exists a point $(v',h')$ with $\|v'\|\le R$ and $h'\ge t$ such that $\|w-v'\|\le \sqrt{q-h'}$, implying 
    \[
        \|w\|\le R+\sqrt{q-t},\qquad q\ge t.
    \]
    At the same time on $\cE(\eta,R,r)^c$ we have that there exists a simplex $S\in\cL^*(\eta)$ such that $S\cap B_R\neq \varnothing$ and $\apex (\Pi^{-}(S))=(w,q)\not\in \big(\Pi_{(\origin,(R+r)^2)}^{-}\big)^{\downarrow}$, leading to
    \[
        \|w\|\ge \sqrt{\max((R+r)^2-q,0)}.
    \]
    This implies, that
    \begin{align*}
    P_1(\eta, R,r,t)&\leq \PP\big(\exists S\in \cL^*(\eta)\colon \apex (\Pi^{-}(S))=(w,q)\\
    &\qquad\qquad\text{ satisfies }q\ge t, \sqrt{\max((R+r)^2-q,0)}\leq \|w\|\leq R+\sqrt{q-t}\big).
    \end{align*}
    Further note that the condition $\sqrt{\max((R+r)^2-q,0)}\leq R+\sqrt{q-t}$ implies $q\ge r^2/2+t$. Indeed, if $\max((R+r)^2-q,0)=0$, then we immediately get $q\ge (R+r)^2\ge r^2/2$. On the other hand, for $q\le (R+r)^2$ we use that for any $t\leq 0$ it holds that
    \[
        R+ \sqrt{q-t}\ge \sqrt{(R+r)^2-q}\ge \sqrt{(R+r)^2-(q-t)}.
    \]
    Taking square on both sides yields
    \[
        R^2+2R\sqrt{q-t} + (q-t) \ge (R+r)^2-(q-t)
    \]
    and, hence,
    \[
        q-t+R\sqrt{q-t}-(Rr+\frac{r^2}{2})\ge 0.
    \]
    In particular this implies
    \[
        \sqrt{q-t} \ge \frac{-R+\sqrt{R^2+4Rr+2r^2}}{2}\ge\frac{-R+R+\sqrt{2}r}{2} = \frac{r}{\sqrt{2}}.
    \]
    
    Thus, defining
    \begin{align*}
        K_5(r,t):=\{(w,q)\in\RR^{d}\colon q\in \big[r^2/2+t,\infty), \|w\|\leq R+\sqrt{q-t}\},
    \end{align*}
    we conclude 
    \begin{equation}\label{eq:P1Estimate1}
        P_1(\eta, R,r,t)\leq \PP\big(\exists S\in \cL^*(\eta)\colon \apex (\Pi^{-}(S))=(w,q)\in K_5(r,t)\big).
    \end{equation}
    Next, we consider a covering of $\RR^{d}\times [t+r^2/2,\infty)$ by the boxes
    \[
        Q_m(z):=(z\oplus [0,\ell(m)]^{d})\times [t+r^2/2+m,t+r^2/2+m+1],\qquad m\in \NN_0,\,\, z\in \ell(m)\ZZ^{d},
    \]
    where $\ell(m)={1\over 2\sqrt{2d}}\sqrt{r^2/2+m}>0$ and $\oplus$ denotes the Minkowski addition. Applying the union bound to \eqref{eq:P1Estimate1} yields
    \[
        P_1(\eta,R,r,t)\le \sum_{m=0}^{\infty}\sum_{\substack{z\in \ell(m)\ZZ^d\colon \\
        Q_m(z)\cap K_5(r,t)\neq \varnothing}}\PP(\exists S\in \cL^*(\eta):\apex (\Pi^-(S))\in Q_m(z)).
    \]
    By stationarity of $\eta$ we get
    \[
    \PP(\exists S\in \cL^*(\eta):\apex (\Pi^-(S))\in Q_m(x))=\PP(\exists S\in \cL^*(\eta):\apex (\Pi^-(S))\in Q_m(\origin)).
    \]
    Further, since the set $\Pi^-(S)\cap \bd \Phi(\eta)$ is a paraboloid facet we have for $(w,q):=\apex (\Pi^-(S))$ that there exist distinct points $x_1=(v_1,h_1),\ldots,x_{d+1}=(v_{d+1},h_{d+1})\in \eta\cap \Pi^-(S)$, such that
    \[
        \pow(w,(v_1,h_1))=\ldots=\pow(w,(v_{d+1},h_{d+1}))=q,
    \]
    and $\inter (\Pi^-(S))^{\downarrow}\cap \eta=\varnothing$, namely $\pow(w,(v,h))\ge q$ for all $(v,h)\in \eta$. Hence, the condition $\apex \Pi^-(S)\in Q_m(\origin)$ implies that there exists $w\in[0,\ell(m)]^d$ such that $\pow(w,(v,h))\ge t+r^2/2+m$ for all $(v,h)\in \eta$. Then we have
    \begin{align*}
        \PP(\exists S\in \cL^*(\eta):\apex (\Pi^-(S))\in Q_m(z))
        &\le \PP\Big(\sup_{w\in [0,\ell(m)]^{d}}\inf_{(v,h)\in \eta}\pow(w,(v,h))>t+r^2/2+m\Big)\\
        &\le \PP\Big(\cH^{\rm{max}}\Big(\eta,{1\over 2\sqrt{2}}\sqrt{r^2/2+m},t+r^2/2+m\Big)^c\Big),
    \end{align*}
     which is independent of $z$. For fixed $m\in \NN_0$, denote by 
     \begin{align*}
        N(m)&:=\#\big\{z\in \ell(m)\ZZ^d\colon Q_m(z)\cap K_5(r,t)\neq \varnothing\big\}
     \end{align*}
     the number of boxes having non-empty intersection with the set $K_5(r,t)$. We note that $(v,h)\in Q_m(z)$ satisfies $h\leq t+r^2/2+m+1$ and, hence, $Q_m(z)\cap K_5(r,t)\neq \varnothing$ together with $\sqrt{r^2/2+m+1}\leq 2\sqrt{r^2/2+m}$ and $R\leq \sqrt{r^2/2+m}$ (due to $r\ge 2\max(1,R)$) implies 
     \[
     z\in B_{R+\sqrt{r^2/2+m+1}+\sqrt{d}\ell(m)}\subset B_{4\sqrt{r^2/2+m}}.
     \]
     Thus,
      \begin{align*}
         N(m)&\leq \#\big\{z\in \ell(m)\ZZ^{d}\colon z\in B_{4\sqrt{r^2/2+m}}\big\}\leq C(\ell(m)^{-1}\sqrt{r^2/2+m})^d\leq C,
    \end{align*}
    and we obtain
    \begin{equation}\label{eq_19.10_3}
        \begin{aligned}
           P_1(\eta,R,r,t)\le C\sum_{m=0}^{\infty}\PP\Big(\cH^{\rm{max}}\Big(\eta,{1\over 2\sqrt{2}}\sqrt{r^2/2+m},t+r^2/2+m\Big)^c\Big).
        \end{aligned}
    \end{equation}
    Next, we evaluate this sum in the cases $\eta=\eta^\gamma$ and $\eta=\eta_{f}$. Let first $\eta=\eta^\gamma$, then by Lemma \ref{lm:BoundaryBounds} and choosing $t=0$ we directly conclude
    \begin{equation*}
        \PP\Big(\cH^{\rm{max}}\Big(\eta^{\gamma},{1\over 2\sqrt{2}}\sqrt{r^2/2+m},r^2/2+m\Big)^c\Big)\le \exp\big(-c\gamma\,(r^2/2+m)^{\frac{d}{2}}\big).
    \end{equation*}
    Next by Lemma \ref{lm:BoundaryBounds} for $\eta=\eta_{f}$ with $f$ admissible, choosing $x_0=1/2$ and $t=-r^2/8$, and noting that $r\ge 4$ and $m\ge 0$ we have that
     \[
        \PP\Big(\cH^{\rm{max}}\Big(\eta_{f},{1\over 2\sqrt{2}}\sqrt{r^2/2+m},3r^2/8+m\Big)^c\Big) \le  \exp\Big(-c\,(\frI^1 f)(1/2)\big(r^2/4+m\big)^{\frac{d}{2}}\Big).
    \]   
    Now applying these bounds to \eqref{eq_19.10_3} we obtain
    \begin{align}\label{eq:P_1}
        P_1(\eta,R,r,t)&\leq C\sum_{m=0}^{\infty}\exp\big(-\widetilde c(m+r^2/4)^{\frac{d}{2}}\big)\leq C\exp(-\widetilde cr^d)\sum_{m=0}^{\infty}e^{-\tilde cm^{d\over 2}}\leq C\exp(-\widetilde cr^d),
    \end{align}
    where in the second step we apply $(a+b)^{d/2}\ge a^{d/2}+b^{d/2}$, which holds for all $a,b\ge 0$ and $d\ge 2$, and $\widetilde c=c\gamma$ for $\eta=\eta^{\gamma}$ and $\widetilde c=c(\frI^1 f)(1/2)$ for $\eta=\eta_f$.

    It remains to bound $P_2(\eta, R,t)$. For $\eta=\eta^\gamma$ we have by Lemma \ref{lm:BoundaryBounds} that $P_2(\eta, R,0)=0$ and the proof in this case follows by combining \eqref{eq:totalprob} and \eqref{eq:P_1}. Let now $f$ be admissible and choose $t=-r^2/8$. By Lemma \ref{lm:BoundaryBounds} we have
    \[
        \begin{aligned}
            P_2(\eta_f,R,-r^2/8)&\le C \big(\big(\frI^{\frac{d}{2}+1}f\big)\big(-r^2/8\big)+R^d(\frI^1 f)(-r^2/8)\big),
        \end{aligned}    
    \]
   and combining this with \eqref{eq:totalprob} and \eqref{eq:P_1} finishes the proof in this case.

\paragraph{Proof of Lemma \ref{lm:Stabilisation}:}

    Let $R,r>0$, $t<0$ and consider the events (see \eqref{eq:Eevent} and \eqref{eq:Hmin})
    \begin{align*}
            \cE&:=\cE({\eta}, R,r),\,&\cH^{{\rm min}}&:=\cH^{{\rm min}}({\eta},2\sqrt{(R+r)^2-t},t),\, &\cH^{\rm max}:=\cH^{{\rm max}}({\eta},R,(R+r)^2).
    \end{align*}

    \medskip
    
    \textbf{Proof of 1.:} First we show that the event $\cE\cap \cH^{{\rm min}}$ ensures that $\cL^*(\eta)\cap B_R$ depends only on the configuration of the corresponding Poisson point process $\eta$ within 
    \[
        K_0(R,r,t)=\big\{(v,h)\in \RR^d \times [t,(R+r)^2]: \|v\|\le 2\sqrt{(R+r)^2-t}\big\},
    \]
    namely that on $\cE\cap \cH^{{\rm min}}$, we have that each set $S\in \cL^*(\eta)$ satisfying $S\cap B_R\neq \varnothing$ is determined by points of $\eta\cap K_0(R,r,t)$. Indeed, let $S\in \cL^*(\eta)$ be such that $S\cap B_R\neq \varnothing$. The simplex $S$ is determined by $\Pi^{-}(S)$, which is defined as the downward paraboloid containing points $(v_i,h_i)\in \eta_n$, $1\leq i\leq d+1$, where $v_1,\ldots,v_{d+1}$ are the vertices of $S$. Note that on $\cE$ we have that $(w,q):=\apex (\Pi^{-}(S))\in \big(\Pi^{-}_{\origin,(R+r)^2}\big)^{\downarrow}$ and, hence,
    \begin{equation}\label{eq:22_12_25_eq1}
        \|w\|\le \sqrt{(R+r)^2-q},\quad q\le (R+r)^2.
    \end{equation}
    For $q\ge t$ the intersection of $\Pi^{-}_{(w,q)}$ with the subspace $\RR^d\times\{t\}$ is a ball of radius $\sqrt{q-t}$ and center $w$. Hence, for any $(w,q)$ satisfying \eqref{eq:22_12_25_eq1} and $q\ge t$ it follows, using the triangle inequality, that the intersection of $\Pi^{-}_{(w,q)}$ with $\RR^d\times\{t\}$ is included in the ball with radius 
    \[
        \sqrt{(R+r)^2-q}+\sqrt{q-t}\leq 2\sqrt{(R+r)^2-t}.
    \]
    Thus, the event $\cH^{{\rm min}}$ additionally implies that $S\subset B_{2\sqrt{(R+r)^2-t}}$ and for all vertices we have $(v_i,h_i)\in K_0(R,r,t)$, $1\leq i\leq d+1$. Hence, on $\cE\cap \cH^{{\rm min}}$ it holds that
    \[
        \bigcup_{S\in \cL^*(\eta), S\cap B_R\neq \varnothing} \bd S = \bigcup_{S\in \cL^*(\eta\cap K_0(R,r,t)), S\cap B_R\neq \varnothing} \bd S.
    \]

    \medskip

    \textbf{Proof of 2.:} Similarly, on $\cH_{n}^{\rm max}\cap \cH_n^{\rm min}$ we have that $\sk{L}(\eta)\cap B_R$ depends only on the configuration of $\eta$ within $K_0(R,r,t)$. In order to show that, we note that on $\cH^{\rm max}$ we have that only points $(v,h)\in\eta$ with $\Pi^+_{(v,h)}\cap (B_R\times (-\infty,(R+r)^2]\neq \varnothing$ may influence the boundary of $\Psi(\eta)$ within the cylinder $B_R\times \RR$ and, hence, the configuration of $\sk{L}(\eta)$ within $B_R$. Such points satisfy
    \begin{equation}\label{eq:02_01_26_eq1}
        \|v\|\leq R+\sqrt{(R+r)^2-h},\qquad h\leq (R+r)^2.
    \end{equation}
    Assume now that there is $(v,h)\in\eta$ satisfying \eqref{eq:02_01_26_eq1} with $h<t$. Then for sufficiently small $\delta>0$ there is $(w,s)\in(\Pi^+_{(v,h)})^{\uparrow}$ with $s=t-\delta<t$ and
    \[
        \|w\|= R+\sqrt{(R+r)^2-t+\delta}\leq 2\sqrt{(R+r)^2-t}.
    \]
    By definition $(w,s)\in \Psi(\eta)$ and, hence, this contradicts $\cH^{\rm min}$. Further $h\ge t$ and \eqref{eq:02_01_26_eq1} imply 
    \[
        \|v\|\leq R+\sqrt{(R+r)^2-t}\leq 2\sqrt{(R+r)^2-t},
    \]
     since $t\leq 0$ and $r>0$. Hence, on $\cH_{n}^{\rm max}\cap \cH_n^{\rm min}$ we have that 
    \[
        \sk{L}(\eta)\cap B_R=\sk{L}(\eta\cap K_0(R,r,t))\cap B_R.
    \]

\paragraph{Proof of Lemma \ref{lm:C1properties}:}
Using \hyperref[item:C1]{(C1)}-(iii) let $x_0\in\RR$ and $\delta>0$ be such that
\[
    M:=\sup_{n\in\NN}\int_{-\infty}^{x_0}|h|^{{d\over 2}+\delta}f_n(h)\dd h<\infty.
\]
\textbf{Proof of 1.:} Let $k\in \NN$. By \hyperref[item:C1]{(C1)}-(ii)
we get
\begin{align*}
    \Big|\int_{-\infty}^{x_0}|h|^{d/2+\delta}{\bf 1}\{h\in [-k,x_0]\}(f_n(h)-f(h))\dd h\Big|&\leq \max(|k|,x_0)^{d/2+\delta}\int_{-\infty}^{x_0}|f_n(h)-f(h)|\dd h\to 0,
\end{align*}
as $n\to \infty$, implying that for any $k\in\NN$ we have
\begin{align*}
    \int_{-\infty}^{x_0}|h|^{d/2+\delta}{\bf 1}\{h\in [-k,x_0]\}f(h)\dd h&=\lim_{n\to\infty}\int_{-\infty}^{x_0}|h|^{d/2+\delta}{\bf 1}\{h\in [-k,x_0]\}f_n(h)\dd h\\
&\leq \limsup_{n\to\infty}\int_{-\infty}^{x_0}|h|^{d/2+\delta}f_n(h)\dd h\leq M.
\end{align*}
Then by the monotone convergence theorem we conclude 
\[
    \int_{-\infty}^{x_0}|h|^{d/2+\delta}f(h)\dd h=\lim_{k\to\infty}\int_{-\infty}^{x_0}|h|^{d/2+\delta}{\bf 1}\{h\in [-k,x_0]\}f(h)\dd h\leq M<\infty.
\]

\medskip

\textbf{Proof of 2.:} Let $n_0>0$ be such that $-x_n\leq x_0$ for all $n\ge n_0$. Further we note that for any $h\leq -x_n$ we have that $x_n^{-1}|h|\ge 1$, since $x_n>0$ for all $n\in\NN$. Then by \hyperref[item:C1]{(C1)}-(iii) we obtain that
\[
    x_n^{d/2}(\frI^1 f_n)(-x_n)\leq x_n^{d/2}\cdot x_n^{-d/2-\delta}\int_{-\infty}^{x_0}|h|^{d/2+\delta}f_n(h)\dd h\leq M x_n^{-\delta}\to 0,
\]
as $n\to\infty$. Similarly, we get
\[
    \Gamma\big(\frac{d}{2}+1\big)\big(\frI^{{d\over 2}+1} f_n\big)(-x_n)\leq \int_{-\infty}^{-x_n}|h|^{d/2}f_n(h)\dd h\leq x_n^{-\delta}\int_{-\infty}^{x_0}|h|^{d/2+\delta}f_n(h)\dd h\leq M x_n^{-\delta}\to 0,
\]
as $n\to\infty$. Due to \hyperref[item:C1]{(C1)}-(ii) the same arguments hold for the function $f$ and we conclude that $x_n^{d/2}(\frI^1 f)(-x_n)\to 0$ and $(\frI^{{d\over2}+1} f)(-x_n)\to 0$ as $n\to\infty$.
    
\paragraph{Proof of Lemma \ref{lm:continuity}:}

    Let $\by:=(y_1,\dots,y_{d+1}) \in {\rm AI}^{d+1}$. Then the center $c(\by)$ and the radius $\rho(\by)$ of the unique ball, containing $y_1,\dots,y_{d+1}$ on its boundary, satisfy
    \begin{align}\label{eq:w_and_s}
        \|y_1-c(\by)\|=\ldots=\|y_{d+1}-c(\by)\|=\rho(\by).
    \end{align}
    Hence, we have $\|y_1\|^2-2\langle c(\by),y_1\rangle=\ldots= \|y_{d+1}\|^2-2\langle c(\by),y_{d+1}\rangle$ and taking the pairwise difference gives us
    \begin{align*}
        2\langle c(\by),y_j-y_1\rangle = \|y_j\|^2-\|y_1\|^2,\qquad 2\leq j\leq d+1.
    \end{align*}
    Let $b(\by):=(\|y_2\|^2-\|y_1\|^2,\ldots, \|y_{d+1}\|^2-\|y_1\|^2)^{\top}$ and $A(\by)$ be the $d\times d$ matrix with rows $2(y_i-y_1)$, $2\leq j\leq d+1$. We note that $\det A(\by)\neq 0$, since $y_1,\ldots, y_{d+1}$ are affinely independent. Since the matrix inversion is a continuous operation on the set of matrices $A$ with $\det A\neq 0$ we get that $c(\by)=A(\by)^{-1}b(\by)$ is continuous. The continuity of $\rho(\by)$ follows by the continuity of $c(\by)$ and \eqref{eq:w_and_s}.
    
\paragraph{Proof of Lemma \ref{lm:conditionsImply}:}

    \textbf{Proof of 1.:}  The function $g_{n,s}$ is positive and
    \begin{align*}
        \int_{\BB^d}g_{n,s}(x)\dint x&=\frac{s^d}{\pi^{d/2}(\frI^{d/2}f_n)(s^2)}\int_{\BB^d}f_n(s^2-s^2\|x\|^2)\dint x\\
        &=\frac{s^d}{\Gamma({d\over 2})(\frI^{d/2}f_n)(s^2)}\int_{0}^{1}r^{{d\over 2}-1}f_n(s^2-s^2r)\dint r=1.
    \end{align*}
    Hence, the measure $\PP_{n,s}(\cdot):=\int_{\BB^d}g_{n,s}(x)\ind(x\in\cdot)\dd x$ defines a probability measure on $\BB^d$.

    \medskip
    
    \textbf{Proof of 2.:} Let $s>0$. We leave out the index $s$ for now and write $\PP_n:=\PP_{n,s}$ and $g_n:=g_{n,s}$.
    Note that $\BB^d$ is a compact metric space. Hence, for any sequence of probability measures $(\PP_n)_{n\in \NN}$, there exists a subsequence $(\PP_{n_j})_{j\in \NN}$ such that 
    \[
        \PP_{n_j}{\underset{j\to\infty}\longrightarrow} \PP\quad\text{weakly}
    \]
    for some probability measure $\PP$ on $\RR^d$ (see \cite[Lemma 8.11.]{koralov2007theory}). This means that the sequence $(\PP_n)_{n \in \NN}$ is weakly compact, which is equivalent to being tight in $\RR^d$ by Helly's Theorem, see \cite[Theorem 8.9]{billing}. The probability measure $\PP$ is concentrated on $\SS^{d-1}$ since 
    \[
        \PP_{n}(r\BB^d){\underset{n\to\infty}\longrightarrow} 0, \qquad r\in(0,1).
    \]
    Indeed, using spherical coordinates $\alpha=\|x\|$ and the change of variables $t=s^2\alpha^2$ yields
    \begin{align*}
        \PP_{n}(r\BB_d) &= \frac{s^d}{\pi^{d/2}(\frI^{d/2}f_n)(s^2)}\int_{r\BB^d}f_n(s^2-s^2\|x\|^2)\dd x\\
        &= \frac{1}{\Gamma(\frac{d}{2})(\frI^{d/2}f_n)(s^2)}\int_{s^2(1-r^2)}^{s^2}f_n(t)(s^2-t)^{\frac{d}{2}-1}\dd t\\
        &\le \frac{1}{\Gamma(\frac{d}{2})(\frI^{d/2}f_n)(s^2)}(sr)^{d-2}\int_{s^2(1-r^2)}^{s^2}f_n(t)\dd t,
    \end{align*}
    since $\frac{d}{2}-1\ge 0$, and moreover
    \[
        \lim_{n\to \infty}\int_{s^2(1-r^2)}^{s^2}f_n(t)\dd t = \lim_{n\to \infty}\int_{0}^{s^2}f_n(t)\dd t - \lim_{n\to \infty}\int_{0}^{s^2(1-r^2)}f_n(t)\dd t = \gamma - \gamma = 0.
    \]
    
    Furthermore, $\PP$ is rotational invariant. As stated in Section \ref{subsec:FUN} the unique rotational invariant measure on $\SS^{d-1}$ is $\tilde\sigma_{d-1}$. Therefore each subsequence that converges weakly, converges to the same limit, i.e. $\tilde\sigma_{d-1}$. Then \cite[Corollary of Theorem 5.1.]{billing} yields $\PP_{n}{\underset{j\to\infty}\longrightarrow} \tilde\sigma_{d-1}$ weakly. 
    
\section*{Acknowledgments}

The authors were supported by the DFG under Germany's Excellence Strategy  EXC 2044 -- 390685587, \textit{Mathematics M\"unster: Dynamics - Geometry - Structure} and RTG 3027 \textit{Rigorous Analysis of Complex Random Systems}. AG was supported by the DFG priority program SPP 2265 \textit{Random Geometric Systems}.

\end{document}